\documentclass[a4paper,11pt,reqno]{amsart}
\usepackage{amssymb,amsmath,hyperref,comment}
\title[Polar finite decompositions for odd spin]{On odd-spin $A_{1}^{(1)}$-string functions, cross-spin identities,\\ and mock theta conjecture-like identities}

\author{Stepan Konenkov}
\address{Department of Mathematics and Computer Science, Saint Petersburg State University, Saint Petersburg,  Russia, 199178}
\email{konenkov.stepan@yandex.ru)}

\author{Eric T. Mortenson}
\address{Department of Mathematics and Computer Science, Saint Petersburg State University, Saint Petersburg,  Russia, 199178}
\email{etmortenson@gmail.com}

\renewcommand\theta{\vartheta}

\newtheorem{theorem}{Theorem}
\newtheorem{lemma}[theorem]{Lemma}
\newtheorem{corollary}[theorem]{Corollary}
\newtheorem{proposition}[theorem]{Proposition}
\theoremstyle{definition}

\newtheorem{remark}[theorem]{Remark}

\numberwithin{theorem}{section} 
\numberwithin{equation}{section}

\newcommand{\Z}{\mathbb{Z}}

\newcommand{\im}{\textnormal{Im}}

\makeatletter
\@namedef{subjclassname@2020}{\textup{2020} Mathematics Subject Classification}
\makeatother

\begin{document}

\date{8 March 2026}

\subjclass[2020]{Primary 11F37, 11F27, 33D90, 11B65; Secondary 17B67, 81R10, 81T40}

\keywords{string functions, parafermionic characters, admissible characters, polar-finite decomposition of Jacobi forms, Appell functions, mock theta functions}

\begin{abstract}
Determining the explicit forms and modularity for string functions and branching coefficients for Kac--Moody algebras after Kac, Peterson, and Wakimoto is a long-standing, yet wide-open, problem and recently a connection has been made between positive admissible-level  $A_{1}^{(1)}$-string functions and Ramanujan's mock theta functions.   In this paper we obtain the polar-finite decomposition for the admissible-level $A_{1}^{(1)}$ character of odd spin, and we also find new mock theta conjecture-like identities for the odd-spin, $2/3$-level and $2/5$-level $A_{1}^{(1)}$-string functions.  
\end{abstract}

\maketitle

 \tableofcontents

\section{Introduction}

Determining the explicit forms and modularity for string functions and branching coefficients for Kac--Moody algebras after Kac, Peterson, and Wakimoto is a long-standing, yet wide-open, problem \cite{KP84, KW88pnas, KW88advmath, KW90}, and recently a connection has been made between positive admissible-level for $A_{1}^{(1)}$-string functions and Ramanujan's mock theta functions \cite{BoMo2026, BoMo2025, KM25, KM26}.   Kac and Peterson \cite{KP84} derived the modular properties of integral-level string functions, and calculated them for certain cases in terms of theta functions. Subsequently, Kac and Wakimoto generalized these results and computed certain branching functions using methods of modular and conformal invariance \cite{KW88advmath}. 

\smallskip
Physicists and mathematicians have taken the work of Kac, Peterson, and Wakimoto in multiple directions. In one direction, Ahn, Chung, and Tye \cite{ACT} introduced the generalized Fateev--Zamolodchikov parafermionic theories \cite{FZ85}, which have applications in statistical physics and string theory. In another direction, Auger, Creutzig, and Ridout pursued the corresponding logarithmic parafermionic vertex algebra for negative admissible-levels  \cite{ACR}. 

\smallskip
In recent works, Borozenets and Mortenson \cite{BoMo2026, BoMo2025} and then Konenkov and Mortenson \cite{KM25} determined explicit forms for certain admissible-level string functions for the affine Kac--Moody algebra $A_{1}^{(1)}$.  Admissible-level string functions are parametrized in terms of level, quantum number, and spin.

\smallskip
For positive fractional-level string functions Borozenets and Mortenson obtained mock theta conjecture-like identities.  For even spin, and some for odd spin, the string functions of levels $1/2$, $1/3$, $2/3$, $1/5$ and $2/5$ may be expressed as mock theta conjecture-like identities in terms of the second-order, third-order, and tenth-order mock theta functions found in Ramanujan's last letter to G. H. Hardy \cite{AB2018} and in Ramanujan's lost notebook \cite{RLN}.   The levels $2/3$ and $2/5$ were not found for odd spin \cite{BoMo2025, KM25}.

\smallskip
The work of Borozenets and Mortenson \cite{BoMo2025} was not actually limited to even-spin.  They also obtained a general cross-spin identity that allows one to express an odd-spin string function in terms of an even-spin string function when the admissible-level is of the reduced form $p/q$ and $p$ is odd.  For $p$ even, the spin for both string functions in the cross-spin relation is odd; thus, odd-spin string functions of admissible-level $p/q$, $p$ even, were not contained within the results of \cite{BoMo2025}.

\smallskip
In order to express the positive admissible-level string functions in terms of mock theta conjecture-like identities, Borozenets and Mortenson focused on even-spin and introduced quasi-periodicity \cite{BoMo2025} and polar-finite decompositions after Zwegers \cite{Zw2} and Dabholkar, Murthy, and Zagier \cite{DMZ}.  One is thus motivated to find applications of quasi-periodicity and polar-finite decompositions in other settings.  Subsequently, Konenkov and Mortenson \cite{KM26} applied quasi-peridocity to generalized Euler identities of Schilling and Warnaar \cite{SW} after work of Kac and Wakimoto on branching functions \cite{KW90}.

\smallskip
 Dabholkar, Murthy, and Zagier \cite{DMZ} built upon work of Zwegers \cite{Zw2} and introduced a canonical decomposition of a meromorphic Jacobi form into a ``finite-part'' and a ``polar-part.''  The finite-part is a finite linear combination of theta functions with mock modular forms as coefficients, and the polar-part is entirely determined by the poles of the meromorphic Jacobi form.   Borozenets and Mortenson \cite{BoMo2025} used their notion of quasi-periodic relations to extend Zagier--Zwegers' analysis to the case of admissible characters, which are vector-valued meromorphic Jacobi forms \cite{KW88advmath}. Using the polar-finite decomposition of admissible characters, they discovered mock theta conjecture-like identities for string functions of certain positive admissible-levels of even-spin, and thus extended their initial results \cite{BoMo2026}. 
 
\smallskip 
Konenkov and Mortenson \cite{KM26} applied quasi-peridocity to generalized Euler identities of Schilling and Warnaar \cite{SW}.  First Konenkov and Mortenson developed quasi-periodicity for odd spin, and then discovered mock theta conjecture-like identities for both even and odd-spin for levels $1/2$, $1/3$, and $2/3$.  Their identities were of a different nature and did not explicitly give the string functions in terms of mock theta conjecture-like identities.  Moreover, they were unable to express the odd-spin identities for level $2/3$ in terms of the same two mock theta functions used for even-spin level $2/3$!  This was unexpected and is what motivated this paper.

 \smallskip
 In this paper, we obtain the analogous polar-finite decomposition for an admissible $A_{1}^{(1)}$ character of odd spin.  The polar-finite decomposition is rather big, and we would like to convince the reader that there are no mistakes.  So we then obtain the corresponding mock theta conjecture-like identities for odd-spin characters levels $1/2$, $1/3$, and $1/5$ and then compare them with the analogous even-spin results found in \cite{BoMo2025} via the cross-spin identity. 
 
 \smallskip
We then obtain the corresponding mock theta conjecture-like identities for odd-spin string functions of levels $2/3$ and $2/5$.  Similar to results in Konenkov and Mortenson \cite{KM26}, we find that we cannot use the exact same set of mock theta functions as in analogous even-spin results found in \cite{BoMo2025, KM25}.  However, for odd-spin, level $2/3$ we also discover that there are multiple mock theta conjecture type identities involving different sets of third-order mock theta functions!  We then put our findings into the context of the cross-spin identity.  
 
 \smallskip
 In summary, what is new in this paper is the polar-finite decomposition for the admissible-level $A_{1}^{(1)}$ character of odd spin, and the new mock theta conjecture-like identities for the odd-spin, level $2/3$ and $2/5$ $A_{1}^{(1)}$-string functions.


 \subsection{On $A_{1}^{(1)}$-string functions}

We will focus on the case of the affine Kac--Moody algebra $A_1^{(1)}$, and we will continue the focus on positive admissible-levels as in \cite{BoMo2026, BoMo2025, KM25, KM26}.  We review the basic notation.

\smallskip
As has been customary, we begin by following the terminology in Schilling and Warnaar \cite[Section 3]{SW} and introduce the string function as follows.  We let $p^{\prime}\ge 2$, $p \ge 1$ be coprime integers and define the admissible-level to be
\begin{equation} \label{equation:admlevel}
N:=\frac{p^{\prime}}{p}-2.
\end{equation}
Let $q := e^{2\pi i \tau}$ with $\im(\tau) >0$ and $z$ a non-zero complex number. For an admissible $A_{1}^{(1)}$  character $\chi_{\ell}^N (z;q)$ of level $N$ and spin $\ell$, we define the string function for level $N$, quantum number $m$, and spin $\ell$ with $m+\ell \equiv 0 \pmod 2$, through
\begin{equation} \label{equation:fourcoefexp}
\chi_{\ell}^N (z;q)=\sum_{m\in 2\mathbb{Z}+\ell}
C_{m,\ell}^{N}(q) q^{\frac{m^2}{4N}}z^{-\frac{1}{2}m}.
\end{equation}

For the purpose of the Weyl--Kac formula, we will temporarily define the theta function as
\begin{equation}\label{equation:SW-thetaDef}
\Theta_{n,m}(z;q):=\sum_{j\in\mathbb{Z}+n/2m}q^{mj^2}z^{-mj}.
\end{equation}
Using the Weyl--Kac formula, one can express the admissible $A_{1}^{(1)}$ character as
\begin{equation}\label{equation:WK-formula}
\chi_{\ell}^N(z;q)=\frac{\sum_{\sigma=\pm 1}\sigma \Theta_{\sigma (\ell+1),p^{\prime}}(z;q^{p})}
{\sum_{\sigma=\pm 1}\sigma\Theta_{\sigma,2}(z;q)}.
\end{equation}

\smallskip
Although symmetries and periodicity of integral-level string functions have long been known, symmetries of admissible-level string functions have only recently been developed in a quantified manner \cite{BoMo2025}.  Integer-level string functions enjoy classical symmetries \cite[(3.4), (3.5)]{SW}, \cite[(2.40)]{ACT}
\begin{equation*}
C_{m,\ell}^{N}(q) = C_{-m,\ell}^{N}(q), \ C_{m,\ell}^{N}(q) = C_{N-m,N-\ell}^{N}(q).
\end{equation*}
For the integral-level $N$ we have the important periodicity property \cite[(3.5)]{SW}
\begin{equation} \label{eq:inglevelperiod}
C_{m,\ell}^{N}(q) = C_{m+2N,\ell}^{N}(q),
\end{equation}
and hence from \eqref{equation:fourcoefexp} the theta-expansion
\begin{equation}\label{eq:intlevelthetadecomp}
\chi_{\ell}^N(z;q)=\sum_{\substack{0\le m <2N\\m \in 2\Z + \ell}}C_{m,l}^{N}(q)\Theta_{m,N}(z;q).
\end{equation}
In  \cite{BoMo2025}, Borozenets and Mortenson developed quasi-periodicity for even-spin string functions \cite[Theorem 2.1]{BoMo2025} and applied it to \eqref{equation:fourcoefexp}.  They thus obtained a polar-finite expansion for characters of even-spin  \cite[Theorem 2.3]{BoMo2025}.   In more detail, the expression they obtained is similar to (\ref{eq:intlevelthetadecomp}) but with significant differences.  Like the right-hand side of (\ref{eq:intlevelthetadecomp}), they obtained a finite linear-combination of string functions and theta functions, but the string functions are now mock modular.  This is the finite part of the polar-finite decomposition. They also obtained a second term which is now a finite linear combination of theta functions and Appell functions; this is the polar part.

\smallskip
In \cite{KM26}, Konenkov and Mortenson found the analogous form for  quasi-periodicity for odd-spin string functions \cite[Theorem 3.1]{KM26} in their study of generalized Euler identities for $A_{1}^{(1)}$-string functions.  Here we will find and prove the analogous polar-finite decomposition for odd-spin characters.

\smallskip
Before we review the mock modular results of Borozenets, Konenkov, and Mortenson for admissible-level string functions, we give some examples of the modular results of Kac and Peterson for integral-level string functions, see also calculations found in \cite{Mo24, MPS}.  Let us first recall the $q$-Pochhammer notation
\begin{equation*}
(x)_n=(x;q)_n:=\prod_{i=0}^{n-1}(1-q^ix), \ \ (x)_{\infty}=(x;q)_{\infty}:=\prod_{i\ge 0}(1-q^ix),
\end{equation*}
and the theta function
\begin{equation}\label{equation:JTPid}
j(x;q):=(x)_{\infty}(q/x)_{\infty}(q)_{\infty}=\sum_{n=-\infty}^{\infty}(-1)^nq^{\binom{n}{2}}x^n,
\end{equation}
where the last equality is the Jacobi triple product identity.   We will frequently use the notation,
\begin{equation*} 
J_{a,b}:=j(q^a;q^b),
 \ \overline{J}_{a,b}:=j(-q^a;q^b), \ {\text{and }}J_a:=J_{a,3a}=\prod_{i\ge 1}(1-q^{ai}),
\end{equation*}
where $a,b$ are positive integers.    One notes that the two definitions for theta functions are equivalent.

An important example of a theta functions is the Dedekind eta-function.  We define the Dedekind eta-function as follows
\begin{equation*} 
\eta(q) :=q^{1/24}\prod_{n\ge 1}(1-q^{n}).
\end{equation*}
Among more general results, Kac and Peterson \cite[p. 220]{KP84} showed
{\allowdisplaybreaks \begin{gather*}
C_{0,0}^{1}(q) = \eta(q)^{-1},\\
C_{1,1}^{2}(q)= \eta(q)^{-2}\eta(q^2),\\
C_{1,1}^{3}(q)= \eta(q)^{-2} q^{3/40} J_{6,15},\\
C_{2,0}^{4}(q) = \eta(q)^{-2} \eta(q^6)^{-1} \eta(q^{12})^{2}.
\end{gather*}}%

To make our subsequent results easier to state, we introduce the additional notation
\begin{equation}
\mathcal{C}_{m,\ell}^{N}(q) := q^{-s_{\lambda}}C_{m,\ell}^{N}(q) \in \Z[[q]],\label{equation:mathCalCtoStringC}
\end{equation} 
where
\begin{equation*}
s_{\lambda} := -\frac{1}{8}+\frac{(\ell+1)^2}{4(N+2)} - \frac{m^2}{4N}.
\end{equation*}

Although we will not be using Hecke-type double-sums explicitly, one does want to have a feel for admissible-level string function.  We will use the following definition for Hecke-type double-sums:  let $x,y$ be non-zero complex numbers, then
\begin{equation} \label{equation:fabc-def2}
f_{a,b,c}(x,y;q):=\left( \sum_{r,s\ge 0 }-\sum_{r,s<0}\right)(-1)^{r+s}x^ry^sq^{a\binom{r}{2}+brs+c\binom{s}{2}}.
\end{equation}
One way to express admissible-level string functions in terms of Hecke-type double-sums is by using the classical partial fraction expansion for the reciprocal of Jacobi's theta product
\begin{equation}
\frac{1}{j(z;q)}=\frac{1}{(q;q)_{\infty}^3}\sum_{n\in\mathbb{Z}}\frac{(-1)^nq^{\binom{n+1}{2}}}{1-q^{n}z}
\label{equation:jacobiThetaReciprocal}
\end{equation}
and the Weyl--Kac formula \eqref{equation:WK-formula} as found in for example \cite[Section 2.4]{ACT}, \cite[Proposition 3]{L},  \cite[(3.8)]{SW}:  for $p'\geq 2$, $p\geq 1$ coprime integers, $0\leq \ell \leq p'-2$ and $m\in 2\Z +\ell$, then
\begin{equation}
\mathcal{C}_{m,\ell}^{N}(q)
 =\frac{1}{(q)_{\infty}^3}\left ( f_{1,p^{\prime},2pp^{\prime}}(q^{1+\frac{m+\ell}{2}},-q^{p(p^{\prime}+\ell+1)};q)
 -f_{1,p^{\prime},2pp^{\prime}}(q^{\frac{m-\ell}{2}},-q^{p(p^{\prime}-(\ell+1))};q)\right).
 \label{equation:modStringFnHeckeForm}
\end{equation}   
In the case of positive integer level $N>0$, we have the compact form \cite[Example $1.3$]{HM}, \cite{SW}
\begin{equation*}
\mathcal{C}_{m,\ell}^{N}(q)=\frac{1}{(q)_{\infty}^3}
f_{1,1+N,1}(q^{1+\tfrac{1}{2}(m+\ell)},q^{1-\tfrac{1}{2}(m-\ell)};q).
\end{equation*}
 
The double-sum form of admissible-level string functions is important to understanding quasi-periodicity.  The form of quasi-periodicity developed by Borozenets and Mortenson \cite[Theorem 2.1]{BoMo2025} is an application of the building block form of functional equation properties for double-sums as found in \cite[Section 6]{HM}.  Schilling and Warnaar did consider the difference of two admissble-level string functions \cite[(7.6)]{SW}, but they did not write the difference in terms of building blocks.


\subsection{On admissible-level $A_{1}^{(1)}$-string functions and mock theta conjecture-like identities}
In Ramanujan's last letter to Hardy \cite[Chapter 14]{AB2018}, he gave a list of seventeen so-called mock theta functions.   The descriptive ``mock'' is a loose defining term meaning that the mock theta functions have certain asymptotic properties similar to those of ordinary theta functions, but they are not theta functions \cite{GOR}. 

\smallskip
Each mock theta function was defined as a $q$-series convergent for $|q|<1$.  For an example of a mock theta function from Ramanujan's letter, his third-order mock theta function $f_{3}(q)$ can be written \cite[Chapter 14]{AB2018}: 
\begin{equation*}
f_3(q):=\sum_{n= 0}^{\infty}\frac{q^{n^2}}{(-q)_n^2}=1+\frac{q}{(1+q)^2}+\frac{q^4}{(1+q)^2(1+q^2)^2}+\dots
=1+q-2q^2+3q^3+\dots
\end{equation*}
 In Ramanujan's letter, one finds an additional three third-order mock theta functions, ten fifth-order functions, and three seventh-order functions, as well as several identities relating the mock theta functions to each other.  The notion of order is not well-defined.

\smallskip
In the decades that followed Ramanujan's last letter, mock theta functions were shrouded in mystery and their development was somewhat sleepy, but then came the discovery of the lost notebook \cite{RLN}, Hickerson's proof of the mock theta conjectures \cite{AG, H1,H2},  and Zwegers breakthrough work on mock modularity \cite{Zw1, Zw2}.  Suddenly mock theta functions became more central to mathematics.

\smallskip
While on a visit to Trinity College Library at Cambridge University, George Andrews discovered a collection of papers of Ramanujan now known as the lost notebook \cite{RLN}.   The collection of papers contained more mock theta functions and mock theta function identities.  For example, we find the mock theta conjectures, which are a collection of ten identities in which each identity expresses a different fifth-order mock theta function in terms of a building-block mock theta function and a simple quotient of theta functions.   We also find the four tenth-order mock theta functions, and six identities which they satisfy \cite{C1, C2, C3}.   

\smallskip
The particular building-block for the mock theta conjectures is the so-called universal mock theta function $g_3(x;q)$, which is defined
\begin{equation*}
g_3(x;q):=x^{-1}\Big ( -1 +\sum_{n=0}^{\infty}\frac{q^{n^2}}{(x)_{n+1}(q/x)_{n}} \Big ).
\end{equation*}
Using our theta function notation, two of the ten mock theta conjectures read \cite{AG, H1}
{\allowdisplaybreaks \begin{align*}
f_0(q)&:=\sum_{n= 0}^{\infty}\frac{q^{n^2}}{(-q;q)_n}=\frac{J_{5,10}J_{2,5}}{J_1}-2q^2g_3(q^2;q^{10}),\\
f_1(q)&:=\sum_{n= 0}^{\infty}\frac{q^{n(n+1)}}{(-q;q)_n}=\frac{J_{5,10}J_{1,5}}{J_1}-2q^3g_3(q^4;q^{10}).
\end{align*}}%

The universal mock theta function $g_3(x;q)$ is only related to the odd-ordered mock theta functions; however,  both even and odd-ordered mock theta functions can be expressed in terms of Appell functions \cite[Section 5]{HM}.

\smallskip
 In a series of papers, Borozenets and Mortenson \cite{BoMo2026, BoMo2025}, and Konenkov and Mortenson \cite{KM25} expressed string functions of certain admissible-levels in terms of mock theta conjecture-like identities.  
 
 \smallskip
We first recall two of Ramanujan's second-order mock theta functions:
\begin{equation*}
A_2(q)
:=\sum_{n\ge 0}\frac{q^{n+1}(-q^2;q^2)_n}{(q;q^2)_{n+1}}, \ 
\mu_2(q)
:=\sum_{n\ge 0}\frac{(-1)^nq^{n^2}(q;q^2)_n}{(-q^2;q^2)_{n}^2},
\end{equation*}
both of which appearing in the ``Lost Notebook''  \cite[p. 8]{RLN}.   For even-spin, $1/2$-level string functions Borozenets and Mortenson found, slightly rewritten, \cite[Corollaries 2.5, 2.6]{BoMo2026}:  For $(p,p^{\prime})=(2,5)$, $(m,\ell)=(0,2r)$, $r\in\{0,1\}$, we have
    \begin{gather}
(q)_{\infty}^3\mathcal{C}_{0,2r}^{1/2}(q)
=(-1)^{r}
\frac{J_{1}^4J_{4}}{J_{2}^4}j(q^{4r+12};q^{20})
 -2q^{-r} j(q^{1+2r};q^{5})  A_{2}(-q),
 \label{equation:mockThetaConj2502r-2ndA}\\
(q)_{\infty}^3\mathcal{C}_{0,2r}^{1/2}(q)
=(-q)^{-r}\frac{1}{2}
\frac{J_{1}^3}{J_{2}J_{4}}j(-q^{2r+1};-q^{5})
+ q^{-r} \frac{1}{2}  j(q^{1+2r};q^{5}) \mu_{2}(q).
 \label{equation:mockThetaConj2502r-2ndmu}
\end{gather}

In order to state more recent results, we recall two of the third-order mock theta functions found in Ramanujan's last letter to Hardy \cite[Chapter 14]{AB2018}:
\begin{gather*}
f_3(q):=\sum_{n\ge 0}\frac{q^{n^2}}{(-q)_n^2}, \ \ 
\omega_3(q)
:=\sum_{n\ge 0}\frac{q^{2n(n+1)}}{(q;q^2)_{n+1}^2}.
\end{gather*}

 For the even-spin, $1/3$-level string functions Borozenets and Mortenson obtained
\begin{theorem}\label{theorem:newMockThetaIdentitiespP37m0ell2r}\cite[Theorem 2.6]{BoMo2025}
For $(p,p^{\prime})=(3,7)$, $r\in \{0,1,2\}$ we have
\begin{align*}
(q)_{\infty}^3&\mathcal{C}_{0,2r}^{1/3}(q)
=(-q)^{-r}\frac{(q)_{\infty}^3}{J_{2}}
\frac{j(-q^{1+2r};q^{14})j(q^{16+4r};q^{28})}
{j(-1;q)J_{28}} \\
&\qquad   -q^{2-2r}\frac{j(q^{6-2r};q^{14})j(q^{26-4r};q^{28})}{J_{28}} \omega_3(-q)
+  \frac{q^{-r}}{2} \frac{j(q^{1+2r};q^{14})j(q^{16+4r};q^{28})}{J_{28}} 
f_{3}(q^2).
\end{align*}
\end{theorem}

We recall the floor function $\lfloor \cdot \rfloor$.  For the even-spin, $2/3$-level string functions, they found
\begin{theorem}\label{theorem:newMockThetaIdentitiespP38m0ell2r}
\cite[Theorem 2.7]{BoMo2025}
 For $(p,p^{\prime})=(3,8)$, $r\in\{0,1,2,3\}$, we have
\begin{align*}
(q)_{\infty}^3&\mathcal{C}_{0,2r}^{2/3}(q) 
=(-1)^{\lfloor (r+1)/2\rfloor}\cdot
\frac{q^{-r}}{2}\frac{J_{1}^2J_{2}}{J_{4}^2J_{8}}
j(-q^{7-2r};q^{16})j(q^{1+2r};q^{8})\\
&\qquad 
- q^{3-2r}    
\frac{j(q^{7-2r};q^{16})j(q^{30-4r};q^{32})}{J_{32}}
\omega_{3}(-q^{2}) 
 + q^{-r} 
\frac{j(q^{1+2r};q^{16})j(q^{18+4r};q^{32})}{J_{32}} 
 \frac{1}{2}f_{3}(q^4).
 \end{align*}
\end{theorem}

\begin{theorem}\label{theorem:newMockThetaIdentitiespP38m2ell2r} 
\cite[Theorem 2.8]{BoMo2025}
For $(p,p^{\prime})=(3,8)$, $r\in\{0,1,2,3\}$, we have
\begin{align*}
(q)_{\infty}^3\mathcal{C}_{2,2r}^{2/3}(q) 
&=(-1)^{\lfloor (r+1)/2\rfloor}\cdot 
\frac{q^{3-2r}}{2}\frac{J_{1}^2J_{2}}{J_{4}^2J_{32}}j(q^{2+4r};q^{32})j(q^{7-2r};q^{16})\\
&\qquad 
- q^{3-2r}   
\frac{j(q^{7-2r};q^{16})j(q^{30-4r};q^{32})}{J_{32}}
\left (1- \frac{1}{2}f_{3}(q^4) \right )  \\
&\qquad + q^{1-r} 
\frac{j(q^{1+2r};q^{16})j(q^{18+4r};q^{32})}{J_{32}} 
 \left (1-q^2\omega_{3}(-q^{2}) \right ).
\end{align*}
\end{theorem}

Borozenets and Mortenson \cite{BoMo2025} and then Konenkov and Mortenson \cite{KM25} also found mock theta conjecture-like identities for $1/5$-level string functions \cite[Theorem $2.9$]{BoMo2025} and $2/5$-level string functions \cite[Theorem $2.5$, $2.6$]{KM25} but in terms of Ramanujan's four tenth-order mock theta functions, where the tenth-order functions were first found in the lost notebook \cite{RLN}.  We recall the four tenth-order mock theta functions \cite{C1, C2, C3, RLN}. 
\begin{align*}
{\phi}_{10}(q)&:=\sum_{n\ge 0}\frac{q^{\binom{n+1}{2}}}{(q;q^2)_{n+1}}, \ \ {\psi}_{10}(q):=\sum_{n\ge 0}\frac{q^{\binom{n+2}{2}}}{(q;q^2)_{n+1}}, \\
& \ \ \ \ \ {X}_{10}(q):=\sum_{n\ge 0}\frac{(-1)^nq^{n^2}}{(-q;q)_{2n}}, \ \  {\chi}_{10}(q):=\sum_{n\ge 0}\frac{(-1)^nq^{(n+1)^2}}{(-q;q)_{2n+1}}.
\end{align*}

For even-spin, $1/5$-level string functions, they found the result:
\begin{theorem} \label{theorem:newMockThetaIdentitiespP511m0ell2r} \cite[Theorem $2.9$]{BoMo2025}
 For $(p,p^{\prime})=(5,11)$, $(m,\ell)=(0,2r)$, $r\in\{0,1,2,3,4\}$, we have
\begin{align*}
(q)_{\infty}^3\mathcal{C}_{0,2r}^{1/5}(q)
&=-q^{r^2-3r+1}J_{1,2}j(q^{4+8r};q^{22})
 \\
&\qquad -
q^{6-4r} \times \left ( j(-q^{16+10r};q^{110}) 
-q^{4+8r}j(-q^{6-10r};q^{110})\right )
 \times q\phi_{10}(-q) \\
& \qquad +q^{3-3r}\times \left ( j(-q^{27+10r};q^{110}) 
-q^{3+6r}j(-q^{17-10r};q^{110})\right )
\times \chi_{10}(q^2) \\
&\qquad 
-q^{1-2r}  \times \left ( j(-q^{38+10r};q^{110}) 
-q^{2+4r}j(-q^{28-10r};q^{110})\right )
\times \left ( -\psi_{10}(-q)\right ) \\
&\qquad  
+q^{-r} \times \left ( j(-q^{49+10r};q^{110}) 
-q^{1+2r}j(-q^{39-10r};q^{110})\right )\times X_{10}(q^2).
\end{align*}
\end{theorem}

For the $2/5$-level string functions we have two families of results.
\begin{theorem}\label{theorem:newMockThetaIdentitiespP512m0ell2r}  \cite[Theorem $2.5$]{KM25}
For $(p,p^{\prime})=(5,12)$, $(m,\ell)=(0,2r)$, $r\in\{0,1,2,3,4,5\}$, we have
\begin{align*}
(q)_{\infty}^3\mathcal{C}_{0,2r}^{2/5}(q)
&=-q^{\frac{1}{2}r^2-\frac{5}{2}r+1}j(q^{2+4r};q^{12})j(-q^{1+2r};q^{8})\frac{J_{1}}{J_{4}}\\
&\qquad   - q^{6-4r}
    \times\Big (  j(-q^{17+10r};q^{120} )
 -  q^{4+8r}j(-q^{7-10r};q^{120})\Big )
     \times q^2\phi_{10}(-q^2)\\
&\qquad     +q^{3-3r}
    \times\Big (  j(-q^{29+10r};q^{120} )
 -  q^{3+6r}j(-q^{19-10r};q^{120})\Big )
     \times \chi_{10}(q^4)\\
&\qquad   -q^{1-2r}
  \times\Big (  j(-q^{41+10r};q^{120} )
 -  q^{2+4r}j(-q^{31-10r};q^{120})\Big )
    \times \left ( -\psi_{10}(-q^2)\right)\\
&\qquad     +q^{-r}
   \times\Big (  j(-q^{53+10r};q^{120} )
 -  q^{1+2r}j(-q^{43-10r};q^{120})\Big )
    \times X_{10}(q^4). 
\end{align*}
\end{theorem}

\begin{theorem}\label{theorem:newMockThetaIdentitiespP512m2ell2r}  \cite[Theorem $2.6$]{KM25}
For $(p,p^{\prime})=(5,12)$, $(m,\ell)=(2,2r)$, $r\in\{0,1,2,3,4,5\}$, we have
{\allowdisplaybreaks    \begin{align*}
(q)_{\infty}^3\mathcal{C}_{2,2r}^{2/5}(q)
&=q^{\frac{1}{2}r^2-\frac{3}{2}r+3}j(q^{2+4r};q^{12})j(-q^{5+2r};q^8)\frac{J_{1}}{J_{4}}\\
&\qquad  -q^{10-4r}     \times\left (  j(-q^{17+10r};q^{120} )
 -  q^{4+8r}j(-q^{7-10r};q^{120})\right )   \times
\left (1-X_{10}(q^4)\right )\\
&\qquad    +q^{6-3r} 
    \times\left (  j(-q^{29+10r};q^{120} )
 -  q^{3+6r}j(-q^{19-10r};q^{120})\right )
     \times
\left (1+\psi_{10}(-q^2)\right )\\
&\qquad     -q^{3-2r}    \times\left (  j(-q^{41+10r};q^{120} )
 -  q^{2+4r}j(-q^{31-10r};q^{120})\right )
   \times
\left (1-\chi_{10}(q^4)\right )\\
&\qquad   +q^{1-r} 
    \times\left (  j(-q^{53+10r};q^{120} )
 -  q^{1+2r}j(-q^{43-10r};q^{120})\right )
     \times
\left (1-q^2\phi_{10}(-q^2)\right ). 
 \end{align*}}%
\end{theorem}

 \smallskip
 The results of Borozenets and Mortenson were not limited to even-spin.  They also found a cross-spin identity which allows one to write odd-spin string functions in terms of even-spin string functions; however, there is a catch.  The theorem reads
\begin{theorem}\label{theorem:crossSpin-j-Odd}\cite[Theorem 2.2]{BoMo2025}  For $(p,p^{\prime})=(p,2p+j)$, we have the cross-spin identity
\begin{align*}
(q)_{\infty}^3&\mathcal{C}_{2i-1,2r-1}^{(p,2p+j)}(q)
-(-1)^{p+1}q^{p(i-r)+\binom{p}{2}}(q)_{\infty}^{3}\mathcal{C}_{2i-1-j,2p-2r+j-1}^{(p,2p+j)}(q)\\
&=(-1)^{p}q^{\binom{p}{2}+p(i+r)}
\sum_{m=1}^{p-1}(-1)^{m}q^{\binom{m+1}{2}-m(i+p+r)}\\
&\qquad \qquad \times
\left ( j(-q^{m(2p+j)-2pr};q^{2p(2p+j)})
-q^{2r(m-p)}j(-q^{m(2p+j)+2pr};q^{2p(2p+j)})\right ) .
\end{align*}
\end{theorem}
In the original statement of the theorem, $j$ was odd; however, in the proof of the theorem, the parity of $j$ does not matter.  When $j$ is odd, for example $1/2$, $1/3$ and $1/5$, we can write odd-spin string functions in terms of even-spin string functions.  We will see this in Section \ref{section:crossSpinExamples}.  When $j$ is even, for example levels $2/3$ and $2/$, we cannot express odd-spin in terms of even-spin.  Instead, we can only write the odd-spin string functions in terms of other odd-spin string functions, see Section \ref{section:mockTheta23-level}. 

\section{New results}

\subsection{An odd-spin polar-finite decomposition and some corollaries}
\label{subsection:polarFinite}
Our first result, which also drives the paper, is the polar-finite decomposition for characters of odd spin.  The polar part of the decomposition is in terms of Appell functions, so following Hickerson and Mortenson \cite{HM} we define the Appell function as
\begin{equation}
m(x,z;q):=\frac{1}{j(z;q)}\sum_{r\in\Z}\frac{(-1)^rq^{\binom{r}{2}}z^r}{1-q^{r-1}xz}.\label{equation:m-def}
\end{equation}
All of Ramanujan's classical mock theta functions can be written in terms of Appell functions \cite[Section 5]{HM}.  For example, for the third-order mock theta function we have
\begin{equation*}
f_3(q):=\sum_{n= 0}^{\infty}\frac{q^{n^2}}{(-q)_n^2}
=2m(-q,q;q^3)+2m(-q,q^2;q^3)=4m(-q,q;q^3)+\frac{J_{3,6}^2}{J_1}.
\end{equation*}

For comparison, we recall the polar-finite decomposition for even-spin characters:
\begin{theorem}\label{theorem:generalPolarFiniteEvenSpin}\cite[Theorem 2.3]{BoMo2025} For $(p,p^{\prime})=(p,2p+j)$, we have the polar-finite decomposition for even spin 
{\allowdisplaybreaks \begin{align*}
&\chi_{2r}^{(p,2p+j)} (z;q)\\
&\ \ =\sum_{s=0}^{j-1}z^{-s}q^{\frac{p}{j}s^2}C_{2s,2r}^{(p,2p+j)}(q)j(-z^{j}q^{p(j-2s)};q^{2pj})\\
&\ \ \ \  +\frac{1}{(q)_{\infty}^3}\sum_{s=0}^{j-1}
(-1)^{p}q^{-\frac{1}{8}+\frac{p(2r+1)^2}{4(2p+j)}}q^{\binom{p}{2}-p(r-s)}z^{-s}
j(-q^{p(j-2s)}z^{j};q^{2jp})
 \sum_{m=1}^{p-1}(-1)^{m}q^{\binom{m+1}{2}+m(r-p)} \\
&\qquad      \times\Big (  j(-q^{m(2p+j)+p(2r+1)};q^{2p(2p+j)} )
 -  q^{m(2p+j)-m(2r+1)}j(-q^{-m(2p+j)+p(2r+1)};q^{2p(2p+j)})\Big )\\
&\qquad      \times
\Big (q^{ms-2ps}m(-q^{jm-2ps},-q^{p(j+2s)}z^{-j};q^{2jp}) 
 + q^{-ms}m(-q^{jm+2ps},-q^{p(j-2s)}z^{j};q^{2jp})\Big ).
\end{align*}}%
\end{theorem}

Now we state our first new result:
\begin{theorem}\label{theorem:generalPolarFiniteOddSpin} (new!) For $(p,p^{\prime})=(p,2p+j)$, we have the polar-finite decomposition for odd spin 
{\allowdisplaybreaks \begin{align*}
&\chi_{2r+1}^{(p,2p+j)}(z;q)\\
& =    \sum_{s=0}^{j-1}z^{-\frac{1}{2}(2s+1)}q^{\frac{p}{4j}(2s+1)^2}C_{2s+1,2r+1}^{(p,2p+j)}(q)
    j(-z^{j}q^{p(j-2s-1)};q^{2pj})\\
& \quad +     \frac{1}{(q)_{\infty}^3}
\sum_{s=0}^{j-1}(-1)^{p}q^{-\frac{1}{8}+\frac{p(2r+2)^2}{4(2p+j)}}q^{\binom{p}{2}-p(r-s)}z^{-\frac{1}{2}(2s+1)} 
j(-q^{p(j-2s-1)}z^{j};q^{2jp})\\
& \qquad  \times \sum_{m=1}^{p-1}(-1)^{m}q^{\binom{m+1}{2}+m(r-p)}\\
& \qquad \times\Big (  j(-q^{m(2p+j)+p(2r+2)};q^{2p(2p+j)}) 
 -q^{m(2p+j)-m(2r+2)}j(-q^{-m(2p+j)+p(2r+2)};q^{2p(2p+j)})\Big )\\
& \qquad \times \Big ( q^{m(s+1)-p(2s+1)}m(-q^{jm-p(2s+1)},-q^{p(j+2s+1)}z^{-j};q^{2jp}) \\
&\qquad  \qquad + q^{-ms} m(-q^{jm+p(2s+1)},-q^{p(j-2s-1)}z^{j};q^{2jp})\Big ).
\end{align*} }%
\end{theorem}

For $j=1$ we have the following immediate corollary.
\begin{corollary}\label{corollary:generalPolarFiniteOddSpin1p} (new!) For $(p,p^{\prime})=(p,2p+1)$, $(m,\ell)=(2k+1,2r+1)$ we have 
{\allowdisplaybreaks \begin{align*}
&\chi_{2r+1}^{(p,2p+1)}(z;q)\\
& =   z^{-\frac{1}{2}}q^{\frac{p}{4}}C_{1,2r+1}^{(p,2p+1)}(q)
    j(-z;q^{2p})\\
& \quad +     \frac{1}{(q)_{\infty}^3}
(-1)^{p}q^{-\frac{1}{8}+\frac{p(2r+2)^2}{4(2p+1)}}q^{\binom{p}{2}-pr}z^{-\frac{1}{2}} 
j(-z;q^{2p})
 \sum_{m=1}^{p-1}(-1)^{m}q^{\binom{m+1}{2}+m(r-p)}\\
& \qquad \times\Big (  j(-q^{m(2p+1)+p(2r+2)};q^{2p(2p+1)}) 
 -q^{m(2p+1)-m(2r+2)}j(-q^{-m(2p+1)+p(2r+2)};q^{2p(2p+1)})\Big )\\
& \qquad \times \Big ( q^{m-p}m(-q^{m-p},-q^{2p}z^{-1};q^{2p}) 
+  m(-q^{m+p},-z;q^{2p})\Big ).
\end{align*} }%
\end{corollary}

Theorem \ref{theorem:generalPolarFiniteOddSpin} can be used to obtain the mock theta conjecture-like identities for odd-spin $A_{1}^{(1)}$-string functions of levels $1/2$, $1/3$, $1/5$.  We will demonstrate this for levels $1/2$ and $1/3$.  The three expansions can also be obtained from the corresponding even-spin results found in \cite{BoMo2025} by using the cross-spin identity Theorem \ref{theorem:crossSpin-j-Odd}.  We will also demonstrate this for levels $1/2$ and $1/3$.  For level $1/5$, we will simply state the result.

For level $1/2$, we have
\begin{corollary} \label{corollary:newMockThetaIdentitiespP25m1ell2rPlus1} 
 For $(p,p^{\prime})=(2,5)$, $(m,\ell)=(1,2r+1)$, $r\in\{0,1\}$, we have
    \begin{gather}
(q)_{\infty}^3\mathcal{C}_{1,2r+1}^{(2,5)}(q)
=(-1)^{r}q^{1-2r}
\frac{J_{1}^4J_{4}}{{J_{2}^4}}j(q^{4r+4};q^{20})
+   
q^{-r}j(q^{2+2r};q^5)
  \Big ( 1+2A_2(-q)\Big ),
 \label{equation:mockThetaConj2512rPlus1-2ndA}\\
(q)_{\infty}^3\mathcal{C}_{1,2r+1}^{(2,5)}(q)
=(-1)^{r}\frac{q^{-r}}{2}
\frac{J_{1}^3}{J_{2}J_{4}}j(q^{2r+2};-q^{5})
+ q^{-r}
 j(q^{2+2r};q^5) \Big ( 1-\frac{1}{2}\mu_{2}(q)\Big ).
 \label{equation:mockThetaConj2512rPlus1-2ndmu}
\end{gather}
\end{corollary}

For level $1/3$, we have
\begin{corollary}\label{corollary:newMockThetaIdentitiespP37m1ell2rPlus1} 
For $(p,p^{\prime})=(3,7)$, $r\in\{0,1,2\}$ we have
\begin{align*}
(q)_{\infty}^3\mathcal{C}_{1,2r+1}^{(3,7)}(q)
&=
(-1)^{r}q^{1-2r}\frac{(q)_{\infty}^3}{J_{2}}
\frac{j(-q^{5-2r};q^{14})j(q^{24-4r};q^{28})}
{j(-1;q)J_{28}} \\
&\qquad   +q^{-r}\frac{j(q^{2+2r};q^{14})j(q^{18+4r};q^{28})}{J_{28}}\left ( 1-q \omega_3(-q)\right)\\
&\qquad -  q^{1-2r} \frac{j(q^{9+2r};q^{14})j(q^{4+4r};q^{28})}{J_{28}} 
\left ( 1-\frac{1}{2}f_{3}(q^2)\right ).
\end{align*}
\end{corollary}

For level $1/5$, we have
\begin{corollary} \label{corollary:newMockThetaIdentitiespP511m1ell2rPlus1} For $(p,p^{\prime})=(5,11)$ and $(m,\ell)=(1,2r+1)$, $r\in\{0,1,2,3,4\}$, we have
\begin{align*}
  (q)_{\infty}^3\mathcal{C}_{1,2r+1}^{(5,11)}(q)
&=q^{(r-1)^2}J_{1,2}j(q^{8(r+1)};q^{22})\\
& \quad -q^{6-4r} \times\Big (  j(-q^{21+10r};q^{110}) 
 -q^{8+8r}j(-q^{1-10r};q^{110})\Big )
 \times \Big (1-X_{10}(q^2)\Big )\\
& \quad + q^{3-3r} \times\Big (  j(-q^{32+10r};q^{110}) 
 -q^{6+6r}j(-q^{12-10r};q^{110})\Big )
 \times \Big ( 1+\psi_{10}(-q)\Big )\\
&\quad -q^{1-2r} \times\Big (  j(-q^{43+10r};q^{110}) 
 -q^{4+4r}j(-q^{23-10r};q^{110})\Big )
 \times \Big ( 1- \chi_{10}(q^2)  \Big )\\
& \quad +q^{-r} \times\Big (  j(-q^{54+10r};q^{110}) 
 -q^{2+2r}j(-q^{34-10r};q^{110})\Big )
 \times \Big (1-q\phi_{10}(-q)\Big ).
\end{align*}
\end{corollary}


\subsection{Mock theta conjecture-like identities for $2/3$-level and $2/5$-level string functions}
\label{subsection:mockThetaIdentities}

For the $2/3$-level string functions we have two pairs of results using different sets of third-order mock theta functions!  To state the first pair of results we recall two more of Ramanujan's third-order mock theta functions:
\begin{equation*}
\psi_3(q):=\sum_{n= 1}^{\infty}\frac{q^{n^2}}{(q;q^2)_n}, \  
\chi_3(q):=\sum_{n= 0}^{\infty}\frac{q^{n^2}(-q)_n}{(-q^3;q^3)_n}.
\end{equation*}
Our first pair of results for odd-spin, $2/3$-level string functions then reads
\begin{theorem} \label{theorem:newMockThetaIdentitiespP38m1ell2rPlus1}
(new!) For $(p,p^{\prime})=(3,8)$, $(m,\ell)=(1,2r+1)$, $r\in\{0,1,2\}$ we have
\begin{align*}
 (q)_{\infty}^3 \mathcal{C}_{1,2r+1}^{(3,8)}(q)&=\Theta_{1,2r+1}^{(3,8)}(q)
   -q^{1-2r}
\frac{j(q^{6-2r};q^{16})j(q^{28-4r};q^{32})}{J_{32}}
 \Big ( -\psi_{3}(-q)\Big )\\
& \qquad  +q^{-r}
\frac{j(q^{2+2r};q^{16})j(q^{20+4r};q^{32})}{J_{32}}
  \Big ( 1-\chi_{3}(q)\Big ),
\end{align*} 
where
\begin{equation*}
\Theta_{1,2r+1}^{(3,8)}(q)
:=\begin{cases}
\medskip
\frac{J_{3}J_{4}^2}{J_{2}J_{12}}j(-q^6;q^{16}) &\textup{if} \  r=0,\\
\medskip
3\frac{J_{3}J_{12}^3}{J_{6}^2} &\textup{if} \ r=1,\\
\medskip
q^{-2}\frac{J_{3}J_{4}^2}{J_{2}J_{12}}j(-q^2;q^{16}) & \textup{if} \ r=2.
\end{cases}
\end{equation*}
Using the cross-spin identity in Theorem \ref{theorem:crossSpin-j-Odd}, we can obtain
\end{theorem}
\begin{corollary} \label{corollary:newMockThetaIdentitiespP38m3ell2rPlus1}
(new!) For $(p,p^{\prime})=(3,8)$, $(m,\ell)=(3,2r+1)$, $r\in\{0,1,2\}$ we have
\begin{align*}
 (q)_{\infty}^3 \mathcal{C}_{3,2r+1}^{(3,8)}(q)&=q^{6-3r}\Theta_{1,2(2-r)+1}^{(3,8)}(q)
   -q^{4-2r}
\frac{j(q^{6-2r};q^{16})j(q^{28-4r};q^{32})}{J_{32}}
 q^{-1}\left (1-q\left (1  -\chi_{3}(q)\right) \right )\\
& \qquad  +q^{3-r}
\frac{j(q^{2+2r};q^{16})j(q^{20+4r};q^{32})}{J_{32}}
  \left ( q^{-2}+\psi_{3}(-q)\right ),
\end{align*} 
where $\Theta_{1,2r+1}^{(3,8)}(q)$ is as in  Theorem \ref{theorem:newMockThetaIdentitiespP38m1ell2rPlus1}.
\end{corollary}

For our second pair of results, it is not possible to use $\omega_{3}(q)$.  Instead, we can write
\begin{theorem} \label{theorem:newMockThetaIdentitiespP38m1ell2rPlus1Alt}
(new!) For $(p,p^{\prime})=(3,8)$, $(m,\ell)=(1,2r+1)$, $r\in\{0,1,2\}$ we have
\begin{align*}
 (q)_{\infty}^3 \mathcal{C}_{1,2r+1}^{(3,8)}(q)&=\Psi_{1,2r+1}^{(3,8)}(q)
   -q^{1-2r}
\frac{j(q^{6-2r};q^{16})j(q^{28-4r};q^{32})}{J_{32}}
 \frac{1}{4}f_{3}(q)\\
& \qquad  +q^{-r}
\frac{j(q^{2+2r};q^{16})j(q^{20+4r};q^{32})}{J_{32}}
\left ( 1- \frac{1}{4}f_{3}(q)\right ),
\end{align*} 
where
\begin{equation*}
\Psi_{1,2r+1}^{(3,8)}(q)
:=\begin{cases}
\medskip
\frac{1}{4}\frac{J_{1}^2J_{2}}{J_{4}}&\textup{if} \  r=0,\\
\medskip
-\frac{1}{2q}\frac{J_{1}^3J_{4}}{J_{2}^2} &\textup{if} \ r=1,\\
\medskip
\frac{1}{4q^3}\frac{J_{1}^2J_{2}}{J_{4}} & \textup{if} \  r=2.
\end{cases}
\end{equation*}
\end{theorem}

Using the cross-spin identity in Theorem \ref{theorem:crossSpin-j-Odd}, we can obtain
\begin{corollary} \label{corollary:newMockThetaIdentitiespP38m3ell2rPlus1Alt}
(new!) For $(p,p^{\prime})=(3,8)$, $(m,\ell)=(1,2r+1)$, $r\in\{0,1,2\}$ we have
\begin{align*}
(q)_{\infty}^3\mathcal{C}_{3,2r+1}^{(3,8)}(q)
&=q^{6-3r}
\Psi_{1,2(2-r)+1}^{(3,8)}(q)
 +q^{1-r}
 \frac{j(q^{2+2r};q^{16})j(q^{20+4r};q^{32})}{J_{32}}
\left ( 1- \frac{q^2}{4}f_{3}(q)\right )\\
&\qquad  -q^{3-2r}
 \frac{j(q^{6-2r};q^{16})j(q^{28-4r};q^{32})}{J_{32}}
 \left ( 1-q\left ( 1- \frac{1}{4}f_{3}(q)\right ) \right ),
\end{align*}
where $\Psi_{1,2r+1}^{(3,8)}(q)$ is as in  Theorem \ref{theorem:newMockThetaIdentitiespP38m1ell2rPlus1Alt}.
\end{corollary}

\begin{remark}  One could combine the two coefficients of $f_{3}(q)$ in Theorem \ref{theorem:newMockThetaIdentitiespP38m1ell2rPlus1Alt} and Corollary \ref{corollary:newMockThetaIdentitiespP38m3ell2rPlus1Alt} using Lemma \ref{lemma:unusualThetaIdentity2}, but the two remaining simple theta quotients still do not combine.
\end{remark}
For odd-spin, level $2/5$ something similar occurs.  We note that it is not possible to use the tenth order functions $\psi_{10}(q)$ and $\phi_{10}(q)$.  Instead, we have
\begin{theorem} (new!) \label{theorem:newMockThetaIdentitiespP512m1ell2rPlus1}  
For $(p,p^{\prime})=(5,12)$, $(m,\ell)=(0,2r)$, $r\in\{0,1,2,3,4\}$, we have
{\allowdisplaybreaks \begin{align*}
(q)_{\infty}^3\mathcal{C}_{1,2r+1}^{(5,12)}(q)
&=\Theta_{1,2r+1}^{(5,12)}(q)  \\
&\quad  - q^{6-4r} \times\Big (  j(-q^{22+10r};q^{120}) 
 -q^{8+8r}j(-q^{2-10r};q^{120})\Big )
  \times \left ( \frac{1}{2}\chi_{10}(q)\right )\\
& \quad  +q^{3-3r}
\times\Big (  j(-q^{34+10r};q^{120}) 
 -q^{6+6r}j(-q^{14-10r};q^{120})\Big )
  \times \left (\frac{1}{2}X_{10}(q)\right )\\
& \quad   -q^{1-2r}
\times\Big (  j(-q^{46+10r};q^{120}) 
 -q^{4+4r}j(-q^{26-10r};q^{120})\Big )
  \times \left ( 1-\frac{1}{2}X_{10}(q) \right )\\
& \quad  +q^{-r} \times\Big (  j(-q^{58+10r};q^{120}) 
 -q^{2+2r}j(-q^{38-10r};q^{120})\Big )
  \times \left (1-\frac{1}{2}\chi_{10}(q)\right ),
\end{align*} }%
where
\begin{equation*}
\Theta_{1,2r+1}^{(5,12)}(q):=
\begin{cases}
0 &\textup{if} \  r=0,2,4,\\
-\frac{1}{2}q^{-1}J_{1,2}J_{1}  & \textup{if} \ r=1,\\
-\frac{1}{2}q^{-6}J_{1,2}J_{1}  &\textup{if} \  r=3.
\end{cases}
\end{equation*}
\end{theorem}
Using the cross-spin identity in Theorem \ref{theorem:crossSpin-j-Odd}, we can obtain
\begin{corollary} (new!) \label{corollary:newMockThetaIdentitiespP512m3ell2rPlus1} 
For $(p,p^{\prime})=(5,12)$, $(m,\ell)=(2,2r)$, $r\in\{0,1,2,3,4\}$, we have
{\allowdisplaybreaks \begin{align*}
(q)_{\infty}^3\mathcal{C}_{3,2r+1}^{(5,12)}(q)
&=  q^{15-5r} \Theta_{1,2(4-r)+1}^{(5,12)}(q)\\
& \quad  - q^{11-5r} \times\Big (  j(-q^{22+10r};q^{120}) 
 -q^{8+8r}j(-q^{2-10r};q^{120})\Big )
  \times \Big ( -1+ q^{-1} +\frac{1}{2}\chi_{10}(q) \Big )\\
&  \quad  +q^{8-3r}
\times\Big (  j(-q^{34+10r};q^{120}) 
 -q^{6+6r}j(-q^{14-10r};q^{120})\Big )
  \times \Big ( -1 + q^{-2}+\frac{1}{2}X_{10}(q)\Big ) \\
& \quad  -q^{6-2r}
\times \Big (  j(-q^{46+10r};q^{120}) 
 -q^{4+4r}j(-q^{26-10r};q^{120})\Big )
  \times \Big ( q^{-3}-\frac{1}{2}X_{10}(q) \Big )\\
&   \quad +q^{5-r} \times\Big (  j(-q^{58+10r};q^{120}) 
 -q^{2+2r}j(-q^{38-10r};q^{120})\Big )
  \times \Big (  q^{-4}-\frac{1}{2}\chi_{10}(q) \Big ),
\end{align*} }%
where $\Theta_{1,2r+1}^{(5,12)}(q)$ is as in Theorem \ref{theorem:newMockThetaIdentitiespP512m1ell2rPlus1}. 
\end{corollary}

\section{An overview of the paper}
What is new in this paper is the polar-finite decomposition for the admissible-level $A_{1}^{(1)}$ character of odd spin, and the new mock theta conjecture-like identities for the odd-spin, $2/3$-level and $2/5$-level $A_{1}^{(1)}$-string functions.  For brevity, we will omit the proofs for the $2/5$-level results in Theorem \ref{theorem:newMockThetaIdentitiespP512m1ell2rPlus1} and Corollary \ref{corollary:newMockThetaIdentitiespP512m3ell2rPlus1}; there is nothing new.

\smallskip
In Section \ref{section:technicalPrelim}, we review theta functions, Appell functions, odd-spin quasi-periodicity, and the Weyl--Kac formula.  In Section \ref{section:alternateAppellForms}, we review Appell function forms of the relevant second and third-order mock theta function, and we introduce some new Appell function forms.  In Section \ref{section:fryeGarvan}, we introduce two families of theta function identities need to prove the identities for the odd-spin, $2/3$-level string functions in Theorems \ref{theorem:newMockThetaIdentitiespP38m1ell2rPlus1} and \ref{theorem:newMockThetaIdentitiespP38m1ell2rPlus1Alt}.  In Section \ref{section:polarFiniteOddSpin}, we find the polar-finite decomposition for odd-spin admissible $A_{1}^{(1)}$ characters.  In Section 
\ref{section:mockTheta12-level}, respectively \ref{section:mockTheta13-level}, we use the new odd-spin polar-finite decomposition to give new a proof the mock theta conjecture-like identities in Corollary \ref{corollary:newMockThetaIdentitiespP25m1ell2rPlus1}, respectively \ref{corollary:newMockThetaIdentitiespP37m1ell2rPlus1}.  In Section \ref{section:crossSpinExamples}, we confirm the results of Sections \ref{section:mockTheta12-level} and \ref{section:mockTheta13-level} by using the cross-spin identity Theorem \ref{theorem:crossSpin-j-Odd}.  In Section \ref{section:mockTheta23-level}, we use our new odd-spin polar-finite decomposition to  prove the odd-spin, $2/3$-level string function identities in Theorems \ref{theorem:newMockThetaIdentitiespP38m1ell2rPlus1} and \ref{theorem:newMockThetaIdentitiespP38m1ell2rPlus1Alt}.  Using the results of these two theorems, we use our cross-spin identity to prove the string function identities in Corollaries \ref{corollary:newMockThetaIdentitiespP38m3ell2rPlus1} and \ref{corollary:newMockThetaIdentitiespP38m3ell2rPlus1Alt}.


\section{Technical preliminaries}
\label{section:technicalPrelim}

We will frequently use the following product rearrangements:
\begin{subequations}
\begin{gather}
\overline{J}_{0,1}=2\overline{J}_{1,4}=\frac{2J_2^2}{J_1},  \overline{J}_{1,2}=\frac{J_2^5}{J_1^2J_4^2},   J_{1,2}=\frac{J_1^2}{J_2},   \overline{J}_{1,3}=\frac{J_2J_3^2}{J_1J_6}, \notag\\
J_{1,4}=\frac{J_1J_4}{J_2},   J_{1,6}=\frac{J_1J_6^2}{J_2J_3},   \overline{J}_{1,6}=\frac{J_2^2J_3J_{12}}{J_1J_4J_6}.\notag
\end{gather}
\end{subequations}

\noindent Following from the definitions are the general identities:
\begin{subequations}
{\allowdisplaybreaks \begin{gather}
j(q^n x;q)=(-1)^nq^{-\binom{n}{2}}x^{-n}j(x;q), \ \ n\in\mathbb{Z},\label{equation:j-elliptic}\\
j(x;q)=j(q/x;q)=-xj(x^{-1};q)\label{equation:j-flip},\\
j(x;q)={J_1}j(x,qx,\dots,q^{n-1}x;q^n)/{J_n^n} \ \ {\text{if $n\ge 1$,}}\label{equation:1.10}\\
j(x;-q)={j(x;q^2)j(-qx;q^2)}/{J_{1,4}},\label{equation:1.11}\\
j(x^n;q^n)={J_n}j(x,\zeta_nx,\dots,\zeta_n^{n-1}x;q^n)/{J_1^n} \ \ {\text{if $n\ge 1$.}}\label{equation:1.12}\\
j(z;q)=\sum_{k=0}^{m-1}(-1)^k q^{\binom{k}{2}}z^k
j\big ((-1)^{m+1}q^{\binom{m}{2}+mk}z^m;q^{m^2}\big ),\label{equation:jsplit}\\
j(qx^3;q^3)+xj(q^2x^3;q^3)=j(-x;q)j(qx^2;q^2)/J_2={J_1j(x^2;q)}/{j(x;q)},\label{equation:quintuple}
\end{gather}}%
\end{subequations}
\noindent  where identity (\ref{equation:quintuple}) is the quintuple product identity.   For later use, we state the $m=2$ specialization of (\ref{equation:jsplit})
\begin{equation}
j(z;q)=j(-qz^2;q^4)-zj(-q^3z^2;q^4).\label{equation:jsplit-m2}
\end{equation}
We then recall \cite[Theorem 1.1]{H1}
\begin{equation}
j(x;q)j(y;q)
=j(-xy;q^{2})j(-qx^{-1}y;q^{2})-xj(-qxy;q^{2})j(-x^{-1}y;q^{2}),
\label{equation:H1Thm1.1}
\end{equation}
as well as the more general Weierstrauss three-term relation.  Let us define
\begin{equation*}
j(x_1,x_2,\dots,x_n;q)=j(x_1;q)j(x_2;q)\cdots j(x_n;q),
\end{equation*}
then for generic non-zero $a,b,c,d\in \mathbb{C}$,
\begin{align}
j(ac,a/c,bd,b/d;q)=j(ad,a/d,bc,b/c;q)+b/c \cdot j(ab,a/b,cd,c/d;q).
\label{equation:Weierstrauss}
\end{align}

We will need to use a few elementary properties of Appell functions \cite[Section 3]{HM}, \cite{Zw2}.
\begin{proposition}  For generic $x,z\in \mathbb{C}^*$
{\allowdisplaybreaks \begin{subequations}
\begin{gather}
m(x,z;q)=m(x,qz;q),\label{equation:mxqz-fnq-z}\\
m(x,z;q)=x^{-1}m(x^{-1},z^{-1};q),\label{equation:mxqz-flip}\\
m(qx,z;q)=1-xm(x,z;q).\label{equation:mxqz-fnq-x}
\end{gather}
\end{subequations}}
\end{proposition}
We have the well-known changing-$z$ property:
\begin{theorem} \label{theorem:changing-z-theorem}  For generic $x,z_0,z_1\in \mathbb{C}^*$
\begin{equation}
m(x,z_1;q)-m(x,z_0;q)=\Psi(x,z_1,z_0;q),
\end{equation}
where
\begin{equation}
\ \Psi(x,z_1,z_0;q):=\frac{z_0J_1^3j(z_1/z_0;q)j(xz_0z_1;q)}{j(z_0;q)j(z_1;q)j(xz_0;q)j(xz_1;q)}.\label{equation:PsiDef}
\end{equation}
\end{theorem}
The changing-$z$ theorem has the following immediate corollary
\begin{corollary} \label{corollary:mxqz-flip-xz} We have
\begin{equation}
m(x,q,z)=m(x,x^{-1}z^{-1};q).
\end{equation}
\end{corollary}

For future use, we the specialization $n=2$ of Theorem 3.5 of \cite{HM}.  In a sense, \cite[Theorem 3.5]{HM} is a generalization of Theorem \ref{theorem:changing-z-theorem}. 
\begin{corollary}\label{corollary:msplit-m2} We have
{\allowdisplaybreaks \begin{align*}
m(&x,z;q) =  m\big({-}q x^2, z_1;q^{4} \big) - q^{-1} x m\big({-}q^{-1} x^2, z_1;q^{4} \big)\\
& +z_1 \frac{J_2^3}{j(xz;q) j(z_1;q^{4})} \left [  
\frac{ j\big({-}q x^2 zz_1;q^2\big) j(z^2/z_1;q^{4})}
{ j\big({-}q x^2z_1;q^{2}\big ) j\big ( z;q^2\big )}
-\frac{ xz j\big({-}q^{2} x^2 zz_1;q^2\big)j(q^{2} z^2/z_1;q^{4})}
{ j\big({-}q x^2z_1;q^{2}\big ) j\big ( q z;q^2\big )}
\right ].
\end{align*}}%
\end{corollary}

We recall the quasi-periodic relations for odd-spin admissible-level string functions.   We will need this result for the proof of Theorem 
\ref{theorem:generalPolarFiniteOddSpin}. 

\begin{theorem}\label{theorem:generalQuasiPeriodicityOddSpin}\cite[Theorem 2.1]{KM26}
For $(p,p^{\prime})=(p,2p+j)$, we have the quasi-periodic relation for odd spin 
{\allowdisplaybreaks \begin{align*}
& (q)_{\infty}^{3}C_{2jt+2s+1,2r+1}^{(p,2p+j)}(q)
 -(q)_{\infty}^{3}C_{2s+1,2r+1}^{(p,2p+j)}(q)\\
& \ = (-1)^{p}q^{-\frac{1}{8}+\frac{p(2r+2)^2}{4(2p+j)}}q^{\binom{p+1}{2}-p(r+1-s)-\frac{p}{4j}(2s+1)^2}
\sum_{i=1}^{t}q^{-2pj\binom{i}{2}-p(2s+1)i}\\
&\quad \times \sum_{m=1}^{p-1}(-1)^{m}q^{\binom{m+1}{2}+m(r-p)}
\left (q^{m(ji+s-j+1)} -q^{-m(ji+s)}\right ) \\
&\quad  \times\Big (  j(-q^{m(2p+j)+p(2r+2)};q^{2p(2p+j)}) 
 -q^{m(2p+j)-m(2r+2)}j(-q^{-m(2p+j)+p(2r+2)};q^{2p(2p+j)})\Big ).
\end{align*}}
\end{theorem}

We also have a more convenient form of the Weyl--Kac theorem, which we need need in obtaining the mock conjecture-like identities for various string functions.
\begin{proposition} \label{proposition:WeylKac} \cite[Proposition 4.6]{BoMo2025}
We have
\begin{align*}
\chi_{\ell}^{(p,p^{\prime})}(z;q)
&=z^{-\frac{\ell+1}{2}}q^{p\frac{(\ell+1)^2}{4p^{\prime}}}
\frac{j(-q^{p(\ell+1)+pp^{\prime}}z^{-p^{\prime}};q^{2pp^{\prime}})
-z^{\ell+1}j(-q^{-p(\ell+1)+pp^{\prime}}z^{-p^{\prime}};q^{2pp^{\prime}}) }
{z^{-\frac{1}{2}}q^{\frac{1}{8}}j(z;q)}.
\end{align*}
\end{proposition}


\section{Ramanujan's mock theta functions and alternate Appell function forms}
\label{section:alternateAppellForms}

From \cite[Section 5]{HM}, we have the following classical mock theta functions in Appell function form:

\smallskip
\noindent {\bf `2nd order' functions}
\begin{align}
A_2(q)
&:=\sum_{n\ge 0}\frac{q^{n+1}(-q^2;q^2)_n}{(q;q^2)_{n+1}}
=\sum_{n\ge 0}\frac{q^{(n+1)^2}(-q;q^2)_n}{(q;q^2)_{n+1}^2}
=-m(q,q^2;q^4)\label{equation:2nd-A(q)}\\
\mu_2(q)
&:=\sum_{n\ge 0}\frac{(-1)^nq^{n^2}(q;q^2)_n}{(-q^2;q^2)_{n}^2}
=2m(-q,-1;q^4)+2m(-q,q;q^4)\label{equation:2nd-mu(q)}
\end{align}

\noindent {\bf `3rd order' functions}
{\allowdisplaybreaks \begin{align}
f_3(q)&:=\sum_{n= 0}^{\infty}\frac{q^{n^2}}{(-q)_n^2}
=2m(-q,q;q^3)+2m(-q,q^2;q^3)=4m(-q,q;q^3)+\frac{J_{3,6}^2}{J_1}\label{equation:3rd-f(q)}\\
\omega_3(q)
&:=\sum_{n= 0}^{\infty}\frac{q^{2n(n+1)}}{(q;q^2)_{n+1}^2}
=-q^{-1}m(q,q^2;q^6)-q^{-1}m(q,q^4;q^6)
 =-2q^{-1}m(q,q^2;q^6)+\frac{J_6^3}{J_2 J_{3,6}}\label{equation:3rd-omega(q)}\\
 \psi_3(q)&:=\sum_{n\ge 1}\frac{q^{n^2}}{(q;q^2)_n}
=-m(q,-q;-q^3)+\frac{qJ_{12}^3}{J_4 J_{3,12}}\label{equation:3rd-psi(q)}\\
\chi_3(q)&:=\sum_{n\ge 0}\frac{q^{n^2}(-q)_n}{(-q^3;q^3)_n}=m(-q,q;q^3)+\frac{J_{3,6}^2}{J_1}\label{equation:3rd-chi(q)}
\end{align}}%
We also have some classical identities for the third-order mock theta functions, \cite[Chapter 14]{AB2018}:
{\allowdisplaybreaks \begin{gather}
f_{3}(q)+4\psi_{3}(-q)=\frac{J_{1}^3}{J_{2}^2}\label{equation:mockIdentity-f(q)psi(q)},\\
4\chi_{3}(q)-f_{3}(q)=3\frac{J_{3}^4}{J_{1}J_{6}^2},\label{equation:mockIdentity-chi(q)f(q)}\\
\chi_3(q)+\psi_3(-q)=\frac{J_{3}J_{4}^3}{J_{2}^2J_{12}}.\label{equation:mockIdentity-chi(q)psi(q)}
\end{gather}}%

We find more convenient Appell function forms for the mock theta functions that we will be using.

\begin{proposition} \label{proposition:alternat3rdAppellFormsLvl23}  We have
\begin{align}
m(-q^5, q^6;q^{12}) - q^{-1} m(-q, q^6;q^{12})
&=\chi_3(q)-\frac{J_{3}J_{4}^3}{J_{2}^2J_{12}},
\label{equation:altAppellForm3rd-chi}\\
m(-q^5, q^6;q^{12}) - q^{-1} m(-q, q^6;q^{12})
&=-\psi_3(-q).
\label{equation:altAppellForm3rd-psi}
\end{align}
\end{proposition}

\begin{proof} [Proof of Proposition \ref{proposition:alternat3rdAppellFormsLvl23}] 
Using Corollary (\ref{corollary:msplit-m2}), we have
{\allowdisplaybreaks \begin{align*}
m(-q,q;q^3) &=  m\big({-}q^{5}, q^6;q^{12} \big) + q^{-2}  m\big({-}q^{-1} , q^6;q^{12} \big)\\
& \qquad +q^6 \frac{J_6^3}{j(-q^2;q^3) j(q^6;q^{12})} \cdot  
\Big [ \frac{ j\big({-}q^{12} ;q^6\big)j(q^{-4};q^{12})}
{ j\big({-}q^{11};q^{6}\big ) j\big ( q ;q^6\big )}
+
\frac{ q^{2} j\big({-}q^{15} ;q^6\big)j(q^{2};q^{12})}
{ j\big({-}q^{11};q^{6}\big ) j\big ( q^{4} ;q^6\big )}\Big ].
\end{align*}}%
Applying (\ref{equation:j-flip}), (\ref{equation:j-elliptic}), and pulling out common factors gives
{\allowdisplaybreaks \begin{align*}
m(-q,q;q^3) &=  m\big({-}q^{5}, q^6;q^{12} \big) + q^{-2}  m\big({-}q^{-1} , q^6;q^{12} \big)\\
& \qquad -q \frac{J_6^3}{\overline{J}_{1,3}J_{6,12}\overline{J}_{1,6} } \cdot  
\Big [\frac{ \overline{J}_{0,6} J_{4,12} }
{ J_{1,6}}
- \frac{ \overline{J}_{3,6}J_{2,12}}
{ J_{2,6}}\Big ].
\end{align*}}%
Combining fractions gives
{\allowdisplaybreaks \begin{align*}
m(-q,q;q^3) &=  m\big({-}q^{5}, q^6;q^{12} \big) + q^{-2}  m\big({-}q^{-1} , q^6;q^{12} \big)\\
& \qquad -q \frac{J_6^3}{\overline{J}_{1,3}J_{6,12}\overline{J}_{1,6} } \cdot  
\Big [\frac{ \overline{J}_{0,6} J_{4,12} J_{2,6}-\overline{J}_{3,6}J_{2,12}J_{1,6}}
{ J_{1,6}J_{2,6}}\Big ].
\end{align*}}%
Using (\ref{equation:1.12}) gives
{\allowdisplaybreaks \begin{align*}
m(-q,q;q^3) &=  m\big({-}q^{5}, q^6;q^{12} \big) + q^{-2}  m\big({-}q^{-1} , q^6;q^{12} \big)\\
& \qquad -q \frac{J_6^3}{\overline{J}_{1,3}J_{6,12}\overline{J}_{1,6} } \cdot  \frac{J_{12}}{J_{6}^2} \cdot
\Big [\frac{ \overline{J}_{0,6} J_{2,6} \overline{J}_{2,6}J_{2,6}-\overline{J}_{3,6}J_{1,6}\overline{J}_{1,6}J_{1,6}}
{ J_{1,6}J_{2,6}}\Big ].
\end{align*}}%
Let us try the Weierstrauss three-term relation (\ref{equation:Weierstrauss}) with $(a,b,c,d)=(-iq,-iq^2,iq,i)$.  This gives
\begin{align*}
&j(q^2;q^6)j(-1;q^6)j(q^2;q^6)j(-q^2;q^6)\\
&\qquad =j(q;q^6)j(-q;q^6)j(q^3;q^6)j(-q;q^6)-qj(-q^3;q^6)j(q^{-1};q^6)j(-q;q^6)j(q;q^6)\\
&\qquad =j(q;q^6)j(-q;q^6)j(q^3;q^6)j(-q;q^6)+j(-q^3;q^6)j(q;q^6)j(-q;q^6)j(q;q^6).
\end{align*}
Rearranging terms gives
\begin{align*}
j(q^2;q^6)&j(-1;q^6)j(q^2;q^6)j(-q^2;q^6)
-j(-q^3;q^6)j(q;q^6)j(-q;q^6)j(q;q^6)\\
&=j(q;q^6)j(-q;q^6)j(q^3;q^6)j(-q;q^6).
\end{align*}
Substituting brings us to
{\allowdisplaybreaks \begin{align*}
m(-q,q;q^3) &=  m\big({-}q^{5}, q^6;q^{12} \big) + q^{-2}  m\big({-}q^{-1} , q^6;q^{12} \big)\\
& \qquad -q \frac{J_6^3}{\overline{J}_{1,3}J_{6,12}\overline{J}_{1,6} } \cdot  \frac{J_{12}}{J_{6}^2} \cdot
\Big [\frac{J_{1,6}  \overline{J}_{1,6}J_{3,6} \overline{J}_{1,6}}
{ J_{1,6}J_{4,6}}\Big ].
\end{align*}}%
Product rearrangements and (\ref{equation:mxqz-flip}) yield
{\allowdisplaybreaks \begin{align*}
m(-q,q;q^3) &=  m\big({-}q^{5}, q^6;q^{12} \big) - q^{-1}  m\big({-}q , q^6;q^{12} \big)
 -q\frac{J_{3}J_{12}^3}{J_{4}J_{6}^2}.
\end{align*}}%
The second identity follows from (\ref{equation:3rd-psi(q)}).  The first identity then follows from the second identity and (\ref{equation:mockIdentity-chi(q)psi(q)}).
\end{proof}

We will need the following proposition for the proof of Theorem \ref{theorem:newMockThetaIdentitiespP38m1ell2rPlus1}.
\begin{proposition}\label{proposition:alternateAppellFormsLvl23FirstTwoPairs} We have
\begin{align*}
-q^{-1}m(-q,q^{6};q^{12}) 
+  m(-q^{5},q^{6};q^{12})
&=-\psi_{3}(-q),\\
-q^{-2}m(-q^{-1},q^{6};q^{12}) 
+  m(-q^{7},q^{6};q^{12})
&=1-\left (\chi_{3}(q)-\frac{J_{3}J_{4}^3}{J_{2}^2J_{12}}\right ). 
\end{align*}
\end{proposition}

\begin{proof}[Proof of Proposition \ref{proposition:alternateAppellFormsLvl23FirstTwoPairs}]
The first identity of the proposition is just identity (\ref{equation:altAppellForm3rd-psi}).  For the second identity of the proposition, we apply (\ref{equation:mxqz-flip}) to the first sum and (\ref{equation:mxqz-fnq-x}) to the second sum.  This gives
\begin{align*}
-q^{-2}m(-q^{-1},q^{6};q^{12}) 
+  m(-q^{7},q^{6};q^{12})
&=q^{-1}m(-q,q^{6};q^{12}) 
+  m(-q^{7},q^{6};q^{12}).
\end{align*}
We then apply (\ref{equation:mxqz-fnq-x}) and then (\ref{equation:mxqz-flip}) to the second sum to get
{\allowdisplaybreaks \begin{align*}
-q^{-2}m(-q^{-1},q^{6};q^{12}) 
+  m(-q^{7},q^{6};q^{12})
&=q^{-1}m(-q,q^{6};q^{12}) 
+  1+q^{-5}m(-q^{-5},q^{6};q^{12})\\
&=q^{-1}m(-q,q^{6};q^{12}) 
+  1-m(-q^{5},q^{6};q^{12}).
\end{align*}}%
The result then follows from identity (\ref{equation:altAppellForm3rd-chi})
\end{proof}

We will need the following proposition for the proof of Theorem \ref{theorem:newMockThetaIdentitiespP38m1ell2rPlus1Alt}.
\begin{proposition} \label{proposition:alternateAppellFormsLvl23FirstTwoPairsAlt} We have
\begin{align*}
- q^{-1} m(-q, q^6;q^{12})+m(-q^5, q^6;q^{12}) 
&=\frac{1}{4}f_{3}(q)-\frac{1}{4}\frac{J_{1}^3}{J_{2}^2},\\
-q^{-2}m(-q^{-1}, q^6;q^{12}) + m(-q^{7}, q^6;q^{12})
&=1- \left ( \frac{1}{4}f_{3}(q)-\frac{1}{4}\frac{J_{1}^3}{J_{2}^2}\right ). 
\end{align*}
\end{proposition}
\begin{proof}[Proof of Proposition \ref{proposition:alternateAppellFormsLvl23FirstTwoPairsAlt}]  The first identity is just 
(\ref{equation:altAppellForm3rd-psi}) combined with the classic third-order identity (\ref{equation:mockIdentity-f(q)psi(q)}).  The second identity can be obtained from the first after using (\ref{equation:mxqz-flip}) and (\ref{equation:mxqz-fnq-x}).
\end{proof}


\section{Families of theta function identities}
\label{section:fryeGarvan}

In this section we will prove several families of theta function identities.  Unlike previous works \cite{BoMo2026, BoMo2025, KM25, KM26}, we will not use the methods of Frye and Garvan \cite{FG}.

\smallskip
The following proposition will be used to prove Theorem \ref{theorem:newMockThetaIdentitiespP38m1ell2rPlus1}.

\begin{proposition} \label{proposition:masterThetaIdentitypP38m1ell2rPlus1} Let us define
\begin{align*}
F(r)&:=-i(iq^{3/2})^{-r}
q^{\frac{1}{2}}\frac{J_{1}^3j((-q^{3})^{r+1};q^{12})}
{J_{2,4}J_{6,12}}
 + q^{-r} \frac{j(q^{2+2r};q^{16})j(q^{20+4r};q^{32})}{J_{32}}\frac{J_{3}J_{4}^3}{J_{2}^2J_{12}}\\
& \quad -iq^{-5/2}\frac{J_{1}^2J_{4}J_{12}}{J_{2}^2J_{3}}
  \times \left ( -q^{3-r}\frac{j(q^{2+2r};q^{8})}{J_{8}}j(q^{1-r};q^{4})\right ),
\end{align*}
then
\begin{equation*}
F(r)=\begin{cases}
j(-q^{6};q^{16})\frac{J_{3}J_{4}^2}{J_{2}J_{12}} & \textup{for} \ r=0,\\
3\frac{J_{3}J_{12}^3}{J_{6}^2} & \textup{for} \ r=1,\\
q^{-2} j(-q^2;q^{16})\frac{J_{3}J_{4}^2}{J_{2}J_{12}} & \textup{for} \ r=2.
\end{cases}
\end{equation*}
\end{proposition}

The following proposition will be used to prove Theorem \ref{theorem:newMockThetaIdentitiespP38m1ell2rPlus1Alt}.

\begin{proposition} \label{proposition:masterThetaIdentitypP38m1ell2rPlus1Alt}  We have
{\allowdisplaybreaks \begin{align*}
F(r):=-i(iq^{3/2})^{-r}&
q^{\frac{1}{2}}\frac{J_{1}^3j((-q^{3})^{r+1};q^{12})}
{J_{2,4}J_{6,12}} +   \frac{1}{4}\frac{J_{1}^3}{J_{2}^2} 
  \times \left ( q^{-r}
\frac{j(q^{2+2r};q^{8})}{J_{8}}j(-q^{1-r};q^{4})
  \right ) \\  
& \quad +   (iq^{3/2})^{-1} 
 \Big (q^{2}\frac{J_{1}^2J_{4}J_{12}}{J_{2}^2J_{3}}\Big )
   \times \left (  -q^{-r}
\frac{j(q^{2+2r};q^{8})}{J_{8}}j(q^{1-r};q^{4}) \right ),
\end{align*} }%
where
\begin{equation*}
F(r):=\begin{cases}
\medskip
\frac{1}{4}\frac{J_{1}^2J_{2}}{J_{4}} & \textup{if} \ r=0,\\
\medskip
- \frac{1}{2q}\frac{J_{1}^3J_{4}}{J_{2}^2} &\textup{if} \ r=1,\\
\medskip
 \frac{1}{4q^3}\frac{J_{1}^2J_{2}}{J_{4}} & \textup{if} \ r=2.
\end{cases}
\end{equation*}
\end{proposition}

The following lemma is used in the proof of Proposition \ref{proposition:masterThetaIdentitypP38m1ell2rPlus1}. 
\begin{lemma}  \label{lemma:unusualThetaIdentity1} We have
\begin{equation*}
\frac{J_{3}J_{4}^4}{J_{2}^2J_{12}}-\frac{J_{1}^3}  {J_{2,4}}
 =3q\frac{J_{3}J_{12}^3}{J_{6}^2}.
\end{equation*}
\end{lemma}
\begin{proof}[Proof of Lemma \ref{lemma:unusualThetaIdentity1}]   The three is problematic, so we rewrite the identity as
\begin{equation*}
\frac{J_{3}J_{4}^4}{J_{2}^2J_{12}}-q\frac{J_{3}J_{12}^3}{J_{6}^2}
 =2q\frac{J_{3}J_{12}^3}{J_{6}^2}+\frac{J_{1}^3J_{4}}  {J_{2}^2}.
\end{equation*}
We will prove
\begin{gather}
\frac{J_{3}J_{4}^4}{J_{2}^2J_{12}}-q\frac{J_{3}J_{12}^3}{J_{6}^2}
=\frac{J_{1}J_{4}^2J_{6}^7}{J_{2}^3J_{3}^2J_{12}^3},
\label{equation:idLHS}\\
2q\frac{J_{3}J_{12}^3}{J_{6}^2}+\frac{J_{1}^3J_{4}}  {J_{2}^2}
=\frac{J_{1}J_{4}^2J_{6}^7}{J_{2}^3J_{3}^2J_{12}^3}.
\label{equation:idRHS}
\end{gather}
We prove (\ref{equation:idLHS}).  Clearing out the denominators, shows that (\ref{equation:idLHS})  is equivalent to
\begin{equation*}
J_{2}J_{3}^3J_{4}^4J_{6}^2J_{12}^2-qJ_{2}^3J_{3}^3J_{12}^6=J_{1}J_{4}^2J_{6}^{9}.
\end{equation*}
Pulling out a common factor and using elementary product rearrangements gives that (\ref{equation:idLHS}) is equivalent to
\begin{equation*}
J_{2}J_{3}^3J_{4}^2J_{6}^2J_{12}^2\left ( J_{4}^2-qJ_{2,12}^2 \right )
 =J_{1}J_{4}^2J_{6}^{9},
\end{equation*}
which is true by (\ref{equation:H1Thm1.1}) with $(x,y,q)\to(q,-q^3,q^6)$ and product rearrangements.

We prove (\ref{equation:idRHS}).  Clearing out the denominators, shows that (\ref{equation:idLHS})  is equivalent to
\begin{equation*}
2qJ_{2}^3J_{3}^3J_{12}^6+J_{1}^3J_{2}J_{3}^2J_{4}J_{6}^2J_{12}^3=J_{1}J_{4}^2J_{6}^9.
\end{equation*}
We again set up for (\ref{equation:H1Thm1.1}).  Multiplying by an appropriate quotient of theta functions gives that (\ref{equation:idLHS}) is now equivalent to 
\begin{equation*}
2q\frac{J_{2}^2J_{3}J_{12}^3}{J_{1}J_{4}J_{6}^2}+J_{1}^2=\frac{J_{4}J_{6}^7}{J_{2}J_{3}^2J_{12}^3},
\end{equation*}
which by product rearrangements is just 
\begin{equation*}
q\overline{J}_{1,6}\overline{J}_{0,6}+J_{1}^2=\overline{J}_{2,6}\overline{J}_{3,6},
\end{equation*}
which is true by (\ref{equation:H1Thm1.1}) with $(x,y,q)\to(q,q,q^3)$.
\end{proof}
The following lemma is used in the proofs of Theorems \ref{theorem:newMockThetaIdentitiespP38m1ell2rPlus1} and  \ref{theorem:newMockThetaIdentitiespP38m1ell2rPlus1Alt} in order to set up for Propositions \ref{proposition:masterThetaIdentitypP38m1ell2rPlus1}  and \ref{proposition:masterThetaIdentitypP38m1ell2rPlus1Alt}.
\begin{lemma}  \label{lemma:unusualThetaIdentity2} We have
\begin{gather*}
q^{1-2r}
\frac{j(q^{6-2r};q^{16})j(q^{28-4r};q^{32})}{J_{32}}
- q^{-r}
\frac{j(q^{2+2r};q^{16})j(q^{20+4r};q^{32})}{J_{32}}
=-q^{-r}
\frac{j(q^{2+2r};q^{8})}{J_{8}}j(q^{1-r};q^{4}),\\
q^{1-2r}
\frac{j(q^{6-2r};q^{16})j(q^{28-4r};q^{32})}{J_{32}}
+ q^{-r}
\frac{j(q^{2+2r};q^{16})j(q^{20+4r};q^{32})}{J_{32}}
=q^{-r}
\frac{j(q^{2+2r};q^{8})}{J_{8}}j(-q^{1-r};q^{4}).
\end{gather*}
\end{lemma}
\begin{proof}[Proof of Lemma \ref{lemma:unusualThetaIdentity2}]
The proofs of the two identities are the same, so we will only do the first.  We use theta function properties (\ref{equation:j-flip}) and (\ref{equation:1.12}) to get
{\allowdisplaybreaks \begin{align*}
q^{1-2r}
&\frac{j(q^{6-2r};q^{16})j(q^{28-4r};q^{32})}{J_{32}}
- q^{-r}
\frac{j(q^{2+2r};q^{16})j(q^{20+4r};q^{32})}{J_{32}}\\
&=q^{1-2r}
\frac{j(q^{10+2r};q^{16})j(q^{4+4r};q^{32})}{J_{32}}
- q^{-r}
\frac{j(q^{2+2r};q^{16})j(q^{20+4r};q^{32})}{J_{32}}\\
&=q^{1-2r}
\frac{j(q^{10+2r};q^{16})j(q^{2+2r};q^{16})j(-q^{2+2r};q^{16})}{J_{32}}\frac{J_{32}}{J_{16}^2}\\
&\qquad - q^{-r}
\frac{j(q^{2+2r};q^{16})j(q^{10+2r};q^{16})j(-q^{10+2r};q^{16})}{J_{32}}\frac{J_{32}}{J_{16}^2}.
\end{align*}}%
We then combine theta functions in the numerator with (\ref{equation:1.10}) to get
{\allowdisplaybreaks \begin{align*}
q^{1-2r}
&\frac{j(q^{6-2r};q^{16})j(q^{28-4r};q^{32})}{J_{32}}
- q^{-r}
\frac{j(q^{2+2r};q^{16})j(q^{20+4r};q^{32})}{J_{32}}\\
&=q^{1-2r}
\frac{j(q^{2+2r};q^{8})j(-q^{2+2r};q^{16})}{J_{32}}\frac{J_{32}}{J_{16}^2}\frac{J_{16}^2}{J_{8}}\\
&\qquad - q^{-r}
\frac{j(q^{2+2r};q^{8})j(-q^{10+2r};q^{16})}{J_{32}}\frac{J_{32}}{J_{16}^2}\frac{J_{16}^2}{J_{8}}.
\end{align*}}%
We then simplify, rewrite two theta functions with (\ref{equation:j-flip}), and then combine the two theta functions with (\ref{equation:jsplit-m2}) to get
{\allowdisplaybreaks \begin{align*}
q^{1-2r}
&\frac{j(q^{6-2r};q^{16})j(q^{28-4r};q^{32})}{J_{32}}
- q^{-r}
\frac{j(q^{2+2r};q^{16})j(q^{20+4r};q^{32})}{J_{32}}\\
&=-q^{-r}
\frac{j(q^{2+2r};q^{8})}{J_{8}}\left (j(-q^{10+2r};q^{16})-q^{1-r} j(-q^{2+2r};q^{16})\right ) \\
&=-q^{-r}
\frac{j(q^{2+2r};q^{8})}{J_{8}}\left (j(-q^{6-2r};q^{16})-q^{1-r} j(-q^{14-2r};q^{16})\right ) \\
&=-q^{-r}
\frac{j(q^{2+2r};q^{8})}{J_{8}}j(q^{1-r};q^{4}),
\end{align*}}%
and the result follows.
\end{proof}

\begin{proof}[Proof of Proposition  \ref{proposition:masterThetaIdentitypP38m1ell2rPlus1}]
We start with the case $r=0$.  Two terms cancel resulting in
\begin{equation*}
F(0)=\frac{j(q^{2};q^{16})j(q^{20};q^{32})}{J_{32}}\frac{J_{3}J_{4}^3}{J_{2}^2J_{12}}.
\end{equation*}
We break apart a theta function with (\ref{equation:1.12}), combine two theta functions with (\ref{equation:1.10}), and then simplify to get
{\allowdisplaybreaks \begin{align*}
F(0)&=\frac{j(q^{2};q^{16})j(q^{10};q^{16})j(-q^{10};q^{16})}{J_{32}}\frac{J_{32}}{J_{16}^2}\frac{J_{3}J_{4}^3}{J_{2}^2J_{12}}\\
&=j(q^2;q^8)\frac{J_{16}^2}{J_{8}}j(-q^{10};q^{16})\frac{1}{J_{16}^2}\frac{J_{3}J_{4}^3}{J_{2}^2J_{12}}\\
&=j(-q^{6};q^{16})\frac{J_{3}J_{4}^2}{J_{2}J_{12}}.
\end{align*}}%

For the case $r=1$, one term vanishes.  Product rearrangements give us 
\begin{align*}
F(1)&=-q^{-1}\frac{J_{1}^3}  {J_{2,4}}
 + q^{-1} \frac{j(q^{4};q^{16})j(q^{24};q^{32})}{J_{32}}\frac{J_{3}J_{4}^3}{J_{2}^2J_{12}}\\
 &=-q^{-1}\frac{J_{1}^3}  {J_{2,4}}
 + q^{-1}\frac{J_{4}J_{16}}{J_{8}}\frac{J_{8}J_{32}}{J_{16}}\frac{1}{J_{32}}\frac{J_{3}J_{4}^3}{J_{2}^2J_{12}}\\
 &=-q^{-1}\frac{J_{1}^3}  {J_{2,4}}
 + q^{-1}\frac{J_{3}J_{4}^4}{J_{2}^2J_{12}}.
\end{align*}
The result then follows from Lemma \ref{lemma:unusualThetaIdentity1}.

For the case $r=2$, two terms cancel giving us
\begin{equation*}
F(2)= q^{-2} \frac{j(q^{6};q^{16})j(q^{28};q^{32})}{J_{32}}\frac{J_{3}J_{4}^3}{J_{2}^2J_{12}}.
 \end{equation*}
 We break apart a theta function with (\ref{equation:1.12}), and then combine two theta functions with (\ref{equation:1.10}) to get
 \begin{align*}
 F(2)&= q^{-2} \frac{j(q^{8};q^{16})j(q^2;q^{16})j(-q^2;q^{16})}{J_{32}}\frac{J_{32}}{J_{16}^2}\frac{J_{3}J_{4}^3}{J_{2}^2J_{12}}\\
 &=q^{-2} \frac{j(q^2;q^{8})j(-q^2;q^{16})}{J_{32}}\frac{J_{16}^2}{J_{8}}\frac{J_{32}}{J_{16}^2}\frac{J_{3}J_{4}^3}{J_{2}^2J_{12}},
 \end{align*}
 and the result follows from simplifying and elementary product rearrangements.
\end{proof}

\begin{proof}[Proof of Proposition \ref{proposition:masterThetaIdentitypP38m1ell2rPlus1Alt}] We consider $r=0$.  Elementary product rearrangements give that two terms cancel.  The remaining term easily simplifies
\begin{align*}
F(0)&=-iq^{\frac{1}{2}}\frac{J_{1}^3j(-q^{3};q^{12})}
{J_{2,4}J_{6,12}} +   \frac{1}{4}\frac{J_{1}^3}{J_{2}^2} 
\frac{j(q^{2};q^{8})}{J_{8}}j(-q;q^{4}) 
+iq^{\frac{1}{2}} 
 \frac{J_{1}^2J_{4}J_{12}}{J_{2}^2J_{3}}
\frac{j(q^{2};q^{8})}{J_{8}}j(q;q^{4}) \\
&=    \frac{1}{4}\frac{J_{1}^2J_{2}}{J_{4}}.
\end{align*} 
We consider $r=1$.  One term vanishes, and the other two terms easily combine to yield
\begin{align*}
F(1)=-q^{-1}\frac{J_{1}^3}
{J_{2,4}} +   \frac{1}{4}\frac{J_{1}^3}{J_{2}^2} 
q^{-1}\frac{j(q^{4};q^{8})}{J_{8}}j(-1;q^{4})=-  q^{-1} \frac{1}{2}\frac{J_{1}^3J_{4}}{J_{2}^2}.
\end{align*}
We consider $r=2$.  Identity (\ref{equation:j-elliptic}) and elementary product rearrangements give that two terms cancel
\begin{equation*}
F(2)=  \frac{1}{4q^3}\frac{J_{1}^2J_{2}}{J_{4}}.\qedhere
\end{equation*} 
\end{proof}


\section{The polar-finite decompositions for odd-spin positive admissible-level characters} 
\label{section:polarFiniteOddSpin}
We prove Theorem \ref{theorem:generalPolarFiniteOddSpin}.    We begin with  propositions whose proofs we delay until the end of the section.
\begin{proposition} \label{proposition:polarFinitePreAppell}
We have
{\allowdisplaybreaks \begin{align}
(q)_{\infty}^3&\chi_{2r+1}^{(p,2p+j)}(z;q)
\label{equation:polarFinitePreFinal}\\
&=   (q)_{\infty}^3 \sum_{s=0}^{j-1}z^{-\frac{1}{2}(2s+1)}q^{\frac{p}{4j}(2s+1)^2}C_{2s+1,2r+1}^{(p,2p+j)}(q)
    j(-z^{-j}q^{p(2s+1+j)};q^{2pj})\notag\\
&\qquad +     (-1)^{p}q^{-\frac{1}{8}+\frac{p(2r+2)^2}{4(2p+j)}}
\sum_{s=0}^{j-1}q^{\binom{p+1}{2}-p(r+1-s)}z^{-\frac{1}{2}(2s+1)} 
\sum_{m=1}^{p-1}(-1)^{m}q^{\binom{m+1}{2}+m(r-p)}
\notag\\
&\qquad \qquad \times\Big (  j(-q^{m(2p+j)+p(2r+2)};q^{2p(2p+j)}) 
\notag \\
&\qquad \qquad \qquad -q^{m(2p+j)-m(2r+2)}j(-q^{-m(2p+j)+p(2r+2)};q^{2p(2p+j)})\Big )
\notag\\
&\qquad \qquad \times \sum_{t\in\mathbb{Z}}q^{pjt^2+p(2s+1)t}z^{-jt}
\sum_{i=1}^{t}q^{-pji(i-1)-p(2s+1)i}\left (q^{m(ji+s-j+1)} -q^{-m(ji+s)}\right ).
\notag
\end{align} }%
\end{proposition}

What we will do is to write (\ref{equation:polarFinitePreFinal}) in terms of Appell functions; however, we only need to focus on the last line:
\begin{equation}
\sum_{t\in\mathbb{Z}}q^{pjt^2+p(2s+1)t}z^{-jt}
\sum_{i=1}^{t}q^{-pji(i-1)-p(2s+1)i}\left (q^{m(ji+s-j+1)} -q^{-m(ji+s)}\right ).
\label{equation:initSumOver_i_t}
\end{equation}
This leads us to another proposition.

\begin{proposition}\label{proposition:initSumOver_i_tPreAppellFinal}  We have
{\allowdisplaybreaks \begin{align}
\sum_{t\in\mathbb{Z}}&q^{pjt^2+p(2s+1)t}z^{-jt}
\sum_{i=1}^{t}q^{-pji(i-1)-p(2s+1)i}\left (q^{m(ji+s-j+1)} -q^{-m(ji+s)}\right )
\label{equation:initSumOver_i_tPreAppellFinal}\\
&=-q^{-m(j-s-1)}j(-q^{p(j-2s-1)}z^{j};q^{2jp})m(-q^{-jm+p(2s+1)},-q^{p(j-2s-1)}z^{j};q^{2jp})\notag \\
&\qquad + q^{-ms} j(-q^{p(j-2s-1)}z^{j};q^{2jp})m(-q^{jm+p(2s+1)},-q^{p(j-2s-1)}z^{j};q^{2jp}).\notag
\end{align}}%
\end{proposition}

\noindent Now we have the pieces to prove our general polar-finite decomposition.
\begin{proof}[Proof of Theorem \ref{theorem:generalPolarFiniteOddSpin}] 
Inserting (\ref{equation:initSumOver_i_tPreAppellFinal}) into (\ref{equation:polarFinitePreFinal}) yields
{\allowdisplaybreaks \begin{align*}
&\chi_{2r+1}^{(p,2p+j)}(z;q)\\
&\qquad =    \sum_{s=0}^{j-1}z^{-\frac{1}{2}(2s+1)}q^{\frac{p}{4j}(2s+1)^2}C_{2s+1,2r+1}^{(p,2p+j)}(q)
    j(-z^{-j}q^{p(2s+1+j)};q^{2pj})\\
&\qquad \qquad +     \frac{1}{(q)_{\infty}^3}
\sum_{s=0}^{j-1}(-1)^{p}q^{-\frac{1}{8}+\frac{p(2r+2)^2}{4(2p+j)}}q^{\binom{p}{2}-p(r-s)}z^{-\frac{1}{2}(2s+1)} 
\sum_{m=1}^{p-1}(-1)^{m}q^{\binom{m+1}{2}+m(r-p)}\\
&\qquad \qquad \qquad \times\Big (  j(-q^{m(2p+j)+p(2r+2)};q^{2p(2p+j)}) \\
&\qquad \qquad \qquad \qquad -q^{m(2p+j)-m(2r+2)}j(-q^{-m(2p+j)+p(2r+2)};q^{2p(2p+j)})\Big )\\
&\qquad \qquad \qquad \times \Big ( {-}q^{-m(j-s-1)}j(-q^{p(j-2s-1)}z^{j};q^{2jp})m(-q^{-jm+p(2s+1)},-q^{p(j-2s-1)}z^{j};q^{2jp}) \\
&\qquad \qquad \qquad \qquad + q^{-ms} j(-q^{p(j-2s-1)}z^{j};q^{2jp})m(-q^{jm+p(2s+1)},-q^{p(j-2s-1)}z^{j};q^{2jp})\Big ).
\end{align*} }%
Applying (\ref{equation:j-flip}) to the coefficient of the string function, pulling out a common theta function, and using (\ref{equation:mxqz-flip}), gives us
{\allowdisplaybreaks \begin{align*}
&\chi_{2r+1}^{(p,2p+j)}(z;q)\\
& =    \sum_{s=0}^{j-1}z^{-\frac{1}{2}(2s+1)}q^{\frac{p}{4j}(2s+1)^2}C_{2s+1,2r+1}^{(p,2p+j)}(q)
    j(-z^{j}q^{p(j-2s-1)};q^{2pj})\\
& \quad +     \frac{1}{(q)_{\infty}^3}
\sum_{s=0}^{j-1}(-1)^{p}q^{-\frac{1}{8}+\frac{p(2r+2)^2}{4(2p+j)}}q^{\binom{p}{2}-p(r-s)}z^{-\frac{1}{2}(2s+1)} 
j(-q^{p(j-2s-1)}z^{j};q^{2jp})\\
& \qquad  \times \sum_{m=1}^{p-1}(-1)^{m}q^{\binom{m+1}{2}+m(r-p)}\\
& \qquad \times\Big (  j(-q^{m(2p+j)+p(2r+2)};q^{2p(2p+j)}) 
 -q^{m(2p+j)-m(2r+2)}j(-q^{-m(2p+j)+p(2r+2)};q^{2p(2p+j)})\Big )\\
& \qquad \times \Big ( q^{m(s+1)-p(2s+1)}m(-q^{jm-p(2s+1)},-q^{-p(j-2s-1)}z^{-j};q^{2jp}) \\
&\qquad  \qquad + q^{-ms} m(-q^{jm+p(2s+1)},-q^{p(j-2s-1)}z^{j};q^{2jp})\Big ).
\end{align*} }%
Applying (\ref{equation:mxqz-fnq-z}) to the first Appell function gives the result.
\end{proof}

\begin{proof}[Proof of Proposition \ref{proposition:polarFinitePreAppell}] 
We begin by recalling the defining equation for string functions (\ref{equation:fourcoefexp}) with the specialization $(p,p^{\prime})=(p,2p+j)$, $1\le j \le p-1$, $\ell=2r+1$:
\begin{equation*}
\chi_{2r+1}^{(p,2p+j)} (z;q)=\sum_{k\in\mathbb{Z}}
C_{2k+1,2r+1}^{(p,2p+j)}(q) q^{\frac{p}{4j}(2k+1)^2}z^{-\frac{1}{2}(2k+1)}.
\end{equation*}
Because our quasi-periodicity formula of Theorem \ref{theorem:generalQuasiPeriodicityOddSpin} has quasi-period $2j$, we need to consider $k \pmod j$.  We write
\begin{equation*}
\chi_{2r+1}^{(p,2p+j)}(z;q)
=\sum_{t\in\mathbb{Z}}
\sum_{s=0}^{j-1}C_{2jt+2s+1,2r+1}^{(p,2p+j)}(q)q^{\frac{p}{4j}(2(jt+s)+1)^2}z^{-\frac{1}{2}(2(jt+s)+1)}.
\end{equation*}
Multiplying by $(q)_{\infty}^3$ and using Theorem \ref{theorem:generalQuasiPeriodicityOddSpin}, this reads
{\allowdisplaybreaks \begin{align}
(q)_{\infty}^3&\chi_{2r+1}^{(p,2p+j)}(z;q)
\label{equation:prePolarFiniteV1}\\
& =\sum_{t\in\mathbb{Z}}
    \sum_{s=0}^{j-1}(q)_{\infty}^{3}C_{2s+1,2r+1}^{(p,2p+j)}(q)
    q^{\frac{p}{4j}(2(jt+s)+1)^2}z^{-\frac{1}{2}(2(jt+s)+1)}
    \notag\\
&\qquad   +\sum_{t\in\mathbb{Z}}
    \sum_{s=0}^{j-1}  (-1)^{p}q^{-\frac{1}{8}+\frac{p(2r+2)^2}{4(2p+j)}}q^{\binom{p+1}{2}-p(r+1-s)-\frac{p}{4j}(2s+1)^2}
    q^{\frac{p}{4j}(2(jt+s)+1)^2}z^{-\frac{1}{2}(2(jt+s)+1)}
\notag \\
&\qquad \qquad \times \sum_{i=1}^{t}q^{-2pj\binom{i}{2}-p(2s+1)i}\sum_{m=1}^{p-1}(-1)^{m}q^{\binom{m+1}{2}+m(r-p)}
\notag \\
&\qquad  \qquad \qquad \times\left (q^{m(ji+s-j+1)} -q^{-m(ji+s)}\right )
\Big (  j(-q^{m(2p+j)+p(2r+2)};q^{2p(2p+j)}) \notag \\
 &\qquad  \qquad \qquad \qquad -q^{m(2p+j)-m(2r+2)}j(-q^{-m(2p+j)+p(2r+2)};q^{2p(2p+j)})\Big ).\notag
\end{align}}%

For the first double-sum over $t$ and $s$ in (\ref{equation:prePolarFiniteV1}), we use the Jacobi triple product identity (\ref{equation:JTPid}) to obtain
{\allowdisplaybreaks \begin{align*}
(q)_{\infty}^{3}&\sum_{t\in\mathbb{Z}}
    \sum_{s=0}^{j-1}C_{2s+1,2r+1}^{(p,2p+j)}(q)
    q^{\frac{p}{4j}(2(jt+s)+1)^2}z^{-\frac{1}{2}(2(jt+s)+1)}\\
    &=(q)_{\infty}^{3}
    \sum_{s=0}^{j-1}C_{2s+1,2r+1}^{(p,2p+j)}(q)\sum_{t\in\mathbb{Z}}
    q^{2pj\binom{t}{2}+p(2s+1+j)t+\frac{p}{4j}(2s+1)^2}z^{-jt-(\frac{2s+1}{2})}\\
    &=(q)_{\infty}^3\sum_{s=0}^{j-1}z^{-(\frac{2s+1}{2})}q^{\frac{p}{4j}(2s+1)^2}C_{2s+1,2r+1}^{(p,2p+j)}(q)
    \sum_{t\in\mathbb{Z}}
    q^{2pj\binom{t}{2}+p(2s+1+j)t}z^{-jt}\\
&=(q)_{\infty}^3\sum_{s=0}^{j-1}z^{-(\frac{2s+1}{2})}q^{\frac{p}{4j}(2s+1)^2}C_{2s+1,2r+1}^{(p,2p+j)}(q)
    j(-z^{-j}q^{p(2s+1+j)};q^{2pj}).
\end{align*}}%
For the second double-sum over $t$ and $s$ in (\ref{equation:prePolarFiniteV1}), rewriting the exponents yields
{\allowdisplaybreaks\begin{align*}
(-1)^{p}&q^{-\frac{1}{8}+\frac{p(2r+2)^2}{4(2p+j)}}q^{\binom{p+1}{2}}
\sum_{s=0}^{j-1}q^{-p(r+1-s)}z^{-\frac{1}{2}(2s+1)}\sum_{t\in\mathbb{Z}}q^{pjt^2+p(2s+1)t}
z^{-jt}
\sum_{i=1}^{t}q^{-pji(i-1)-p(2s+1)i}
 \\
&\qquad  \times\sum_{m=1}^{p-1}(-1)^{m}q^{\binom{m+1}{2}+m(r-p)}\left (q^{m(ji+s-j+1)} -q^{-m(ji+s)}\right )\\
&\qquad \times\Big (  j(-q^{m(2p+j)+p(2r+2)};q^{2p(2p+j)}) 
-q^{m(2p+j)-m(2r+2)}j(-q^{-m(2p+j)+p(2r+2)};q^{2p(2p+j)})\Big ).
\end{align*}}%
Interchanging the sums in order to isolate the sums over $i$ and $t$ brings us to 
{\allowdisplaybreaks\begin{align*}
(-1)^{p}&q^{-\frac{1}{8}+\frac{p(2r+2)^2}{4(2p+j)}}
\sum_{s=0}^{j-1}q^{\binom{p+1}{2}-p(r+1-s)}z^{-\frac{1}{2}(2s+1)}
 \\
&  \times\sum_{m=1}^{p-1}(-1)^{m}q^{\binom{m+1}{2}+m(r-p)}\\
& \times\Big (  j(-q^{m(2p+j)+p(2r+2)};q^{2p(2p+j)}) 
-q^{m(2p+j)-m(2r+2)}j(-q^{-m(2p+j)+p(2r+2)};q^{2p(2p+j)})\Big )\\
& \times \sum_{t\in\mathbb{Z}}q^{pjt^2+p(2s+1)t}z^{-jt}
\sum_{i=1}^{t}q^{-pji(i-1)-p(2s+1)i}\left (q^{m(ji+s-j+1)} -q^{-m(ji+s)}\right ).
\end{align*}}%
The result follows.
\end{proof}


\begin{proof}[Proof of Proposition \ref{proposition:initSumOver_i_tPreAppellFinal}] 
We will focus on the last line.  For $t=0$, the inner sum vanishes by the summation convention:  For $b<a$,
\begin{equation}
\sum_{r=a}^{b}c_r:=-\sum_{r=b+1}^{a-1}c_r, \ \textup{e.g.} \ 
\sum_{r=0}^{-1}c_r=-\sum_{r=0}^{-1}c_r=0, \label{equation:sumconvention}
\end{equation}
so we can write
{\allowdisplaybreaks \begin{align*}
\sum_{t\in\mathbb{Z}}&q^{pjt^2+p(2s+1)t}z^{-jt}
\sum_{i=1}^{t}q^{-pji(i-1)-p(2s+1)i}\left (q^{m(ji+s+1-j)} -q^{-m(ji+s)}\right )\\
&=\sum_{t\ge 1}\sum_{i=1}^{t}q^{pjt^2+p(2s+1)t}z^{-jt}
q^{-pji(i-1)-p(2s+1)i}   (q^{m(ji+s+1-j)}-q^{-m(ji+s)})\\
&\quad \qquad + \sum_{t\le -1}\sum_{i=1}^{t}q^{pjt^2+p(2s+1)t}z^{-jt}
q^{-pji(i-1)-p(2s+1)i}   (q^{m(ji+s+1-j)}-q^{-m(ji+s)}).
\end{align*}}%
Again using the summation convention (\ref{equation:sumconvention}), we rewrite the second double-sum
{\allowdisplaybreaks \begin{align*}
\sum_{t\in\mathbb{Z}}&q^{pjt^2+p(2s+1)t}z^{-jt}
\sum_{i=1}^{t}q^{-pji(i-1)-p(2s+1)i}\left (q^{m(ji+s+1-j)} -q^{-m(ji+s)}\right )\\
&=\sum_{t\ge 1}\sum_{i=1}^{t}q^{pjt^2+p(2s+1)t}z^{-jt}
q^{-pji(i-1)-p(2s+1)i}   (q^{m(ji+s+1-j)}-q^{-m(ji+s)})\\
&\quad \qquad - \sum_{t\le -1}\sum_{i=t+1}^{0}q^{pjt^2+p(2s+1)t}z^{-jt}
q^{-pji(i-1)-p(2s+1)i}   (q^{m(ji+s+1-j)}-q^{-m(ji+s)}).
\end{align*}}%
In the last line we have make the substitutions $t\to -t$ and $i\to -i+1$ to get
{\allowdisplaybreaks \begin{align*}
\sum_{t\in\mathbb{Z}}&q^{pjt^2+p(2s+1)t}z^{-jt}
\sum_{i=1}^{t}q^{-pji(i-1)-p(2s+1)i}\left (q^{m(ji+s+1-j)} -q^{-m(ji+s)}\right )\\
&=\sum_{t\ge 1}\sum_{i=1}^{t}q^{pjt^2+p(2s+1)t}z^{-jt}
q^{-pji(i-1)-p(2s+1)i}   (q^{m(ji+s+1-j)}-q^{-m(ji+s)})\\
&\quad \qquad + \sum_{t\le 1}\sum_{i=1}^{t}q^{pjt^2-p(2s+1)t}z^{jt}
q^{-pji(i-1)+p(2s+1)(i-1)}   (q^{m(ji-s-j)}-q^{-m(ji-s-1)}).
\end{align*}}%
Replacing $i$ with $i+1$ and $t$ with $t+1$ brings us to
{\allowdisplaybreaks \begin{align*}
\sum_{t\in\mathbb{Z}}&q^{pjt^2+p(2s+1)t}z^{-jt}
\sum_{i=1}^{t}q^{-pji(i-1)-p(2s+1)i}\left (q^{m(ji+s+1-j)} -q^{-m(ji+s)}\right )\\
&=\sum_{t\ge 1}\sum_{i=0}^{t-1}q^{pjt^2+p(2s+1)t}z^{-jt}
q^{-pj(i+1)i-p(2s+1)(i+1)}   (q^{m(j(i+1)+s+1-j)}-q^{-m(j(i+1)+s)})\\
&\qquad + \sum_{t\le 1}\sum_{i=0}^{t-1}q^{pjt^2-p(2s+1)t}z^{jt}
q^{-pj(i+1)i+p(2s+1)i}   (q^{m(j(i+1)-s-j)}-q^{-m(j(i+1)-s-1)})\\
&=\sum_{t\ge 0}\sum_{i=0}^{t}q^{pj(t+1)^2+p(2s+1)(t+1)}z^{-jt-j}
q^{-pj(i+1)i-p(2s+1)(i+1)}   (q^{m(j(i+1)+s+1-j)}-q^{-m(j(i+1)+s)})\\
& \qquad + \sum_{t\le 0}\sum_{i=0}^{t}q^{pj(t+1)^2-p(2s+1)(t+1)}z^{jt+j}
q^{-pj(i+1)i+p(2s+1)i}   (q^{m(j(i+1)-s-j)}-q^{-m(j(i+1)-s-1)}).
\end{align*}}%
We interchange summation symbols and then simplify to obtain
{\allowdisplaybreaks \begin{align*}
\sum_{t\in\mathbb{Z}}&q^{pjt^2+p(2s+1)t}z^{-jt}
\sum_{i=1}^{t}q^{-pji(i-1)-p(2s+1)i}\left (q^{m(ji+s+1-j)} -q^{-m(ji+s)}\right )\\
&=\sum_{i= 0}^{\infty}\sum_{t=i}^{\infty}q^{pj(t+1)^2+p(2s+1)(t+1)}z^{-jt-j}
q^{-pj(i+1)i-p(2s+1)(i+1)}   (q^{m(j(i+1)+s+1-j)}-q^{-m(j(i+1)+s)})\\
& \qquad + \sum_{i= 0}^{\infty}\sum_{t=i}^{\infty}q^{pj(t+1)^2-p(2s+1)(t+1)}z^{jt+j}
q^{-pj(i+1)i+p(2s+1)i}   (q^{m(j(i+1)-s-j)}-q^{-m(j(i+1)-s-1)}).
\end{align*}}%
We then replace $t$ with $t+i$ to get
{\allowdisplaybreaks \begin{align*}
\sum_{t\in\mathbb{Z}}&q^{pjt^2+p(2s+1)t}z^{-jt}
\sum_{i=1}^{t}q^{-pji(i-1)-p(2s+1)i}\left (q^{m(ji+s+1-j)} -q^{-m(ji+s)}\right )\\
&=\sum_{i= 0}^{\infty}\sum_{t=0}^{\infty}q^{pj(t+i+1)^2+p(2s+1)(t+i+1)}z^{-j(t+i)-j}
q^{-pj(i+1)i-p(2s+1)(i+1)}  \\
&\qquad  \times  (q^{m(j(i+1)+s+1-j)}-q^{-m(j(i+1)+s)})\\
& \quad + \sum_{i= 0}^{\infty}\sum_{t=0}^{\infty}q^{pj(t+i+1)^2-p(2s+1)(t+i+1)}z^{j(t+i)+j}
q^{-pj(i+1)i+p(2s+1)i}  \\
&\qquad \quad \times  (q^{m(j(i+1)-s-j)}-q^{-m(j(i+1)-s-1)}).
\end{align*}}%
Simplifying the exponents and distributing the sums brings us to
{\allowdisplaybreaks \begin{align*}
\sum_{t\in\mathbb{Z}}&q^{pjt^2+p(2s+1)t}z^{-jt}
\sum_{i=1}^{t}q^{-pji(i-1)-p(2s+1)i}\left (q^{m(ji+s+1-j)} -q^{-m(ji+s)}\right )\\
&=  q^{-p(2s+1)+m(s+1)}\sum_{i=0}^{\infty} \sum_{t=0}^{\infty}q^{2jp\binom{t+1}{2}+p(j+2s+1)(t+1)+i(2jpt+jp+jm)}z^{-j(t+1)-ji}
\\
&\qquad -q^{-p(2s+1)-m(j+s)}\sum_{i=0}^{\infty} \sum_{t=0}^{\infty}q^{2jp\binom{t+1}{2}+p(j+2s+1)(t+1)+i(2jpt+jp-jm)}z^{-j(t+1)-ji}\\
&\qquad + q^{-ms}\sum_{i=0}^{\infty} \sum_{t=0}^{\infty}q^{2jp\binom{t+1}{2}+p(j-2s-1)(t+1)+i(2jpt+jp+jm)}z^{j(t+1)+ji}\\
&\qquad - q^{-m(j-s-1)}\sum_{i=0}^{\infty} \sum_{t=0}^{\infty}q^{2jp\binom{t+1}{2}+p(j-2s-1)(t+1)+i(2jpt+jp-jm)}z^{j(t+1)+ji}.
\end{align*}}%
Using the geometric series yields
{\allowdisplaybreaks \begin{align}
\sum_{t\in\mathbb{Z}}&q^{pjt^2+p(2s+1)t}z^{-jt}
\sum_{i=1}^{t}q^{-pji(i-1)-p(2s+1)i}\left (q^{m(ji+s+1-j)} -q^{-m(ji+s)}\right )\
\label{equation:initSumOver_i_tPreAppell}\\
&=  q^{-p(2s+1)+m(s+1)} \sum_{t=0}^{\infty}\frac{q^{2jp\binom{t+1}{2}+p(j+2s+1)(t+1)}z^{-j(t+1)}}{1-q^{2jpt+jp+jm}z^{-j}}
\notag \\
&\qquad -q^{-p(2s+1)-m(j+s)} \sum_{t=0}^{\infty}\frac{q^{2jp\binom{t+1}{2}+p(j+2s+1)(t+1)}z^{-j(t+1)}}{1-q^{2jpt+jp-jm}z^{-j}}
\notag \\
&\qquad + q^{-ms}\sum_{t=0}^{\infty}\frac{q^{2jp\binom{t+1}{2}+p(j-2s-1)(t+1)}z^{j(t+1)}}{1-q^{2jpt+jp+jm}z^{j}}
\notag \\
&\qquad - q^{-m(j-s-1)} \sum_{t=0}^{\infty}\frac{q^{2jp\binom{t+1}{2}+p(j-2s-1)(t+1)}z^{j(t+1)}}{1-q^{2jpt+jp-jm}z^{j}}.\notag
\end{align}}%
We rework the right-hand side of (\ref{equation:initSumOver_i_tPreAppell}). In the first and second lines, we replace $t$ with $-t$, and in the third and fourth lines, we replace $t$ with $t-1$.  This gives
{\allowdisplaybreaks \begin{align*}
\sum_{t\in\mathbb{Z}}&q^{pjt^2+p(2s+1)t}z^{-jt}
\sum_{i=1}^{t}q^{-pji(i-1)-p(2s+1)i}\left (q^{m(ji+s+1-j)} -q^{-m(ji+s)}\right )\\
&=  q^{-p(2s+1)+m(s+1)} \sum_{t=-\infty}^{0}\frac{q^{2jp\binom{-t+1}{2}+p(j+2s+1)(-t+1)}z^{-j(-t+1)}}{1-q^{-2jpt+jp+jm}z^{-j}}
\frac{q^{2jpt-jp-jm}z^{j}}{q^{2jpt-jp-jm}z^{j}}\\
&\qquad -q^{-p(2s+1)-m(j+s)} \sum_{t=-\infty}^{0}\frac{q^{2jp\binom{-t+1}{2}+p(j+2s+1)(-t+1)}z^{-j(-t+1)}}{1-q^{-2jpt+jp-jm}z^{-j}}
 \frac{q^{2jpt-jp+jm}z^{j}}{q^{2jpt-jp+jm}z^{j}}\\
&\qquad + q^{-ms}\sum_{t=1}^{\infty}\frac{q^{2jp\binom{t}{2}+p(j-2s-1)t}z^{jt}}{1-q^{2jp(t-1)+jp+jm}z^{j}}
 - q^{-m(j-s-1)} \sum_{t=1}^{\infty}\frac{q^{2jp\binom{t}{2}+p(j-2s-1)t}z^{jt}}{1-q^{2jp(t-1)+jp-jm}z^{j}}.
\end{align*}}%
Simplifying brings us to
{\allowdisplaybreaks \begin{align}
\sum_{t\in\mathbb{Z}}&q^{pjt^2+p(2s+1)t}z^{-jt}
\sum_{i=1}^{t}q^{-pji(i-1)-p(2s+1)i}\left (q^{m(ji+s+1-j)} -q^{-m(ji+s)}\right )
\label{equation:initSumOver_i_tPreAppellV2}\\
&=  -q^{-m(j-s-1)} \sum_{t=-\infty}^{0}\frac{q^{2jp\binom{t}{2}+p(j- 2s-1)t}z^{jt}}{1-q^{2jp(t-1)+jp-jm}z^{j}}
 -q^{-ms} \sum_{t=-\infty}^{0}\frac{q^{2jp\binom{t}{2}+p(j-2s-1)t}z^{jt}}{1-q^{2jp(t-1)+jp+jm}z^{j}}
 \notag\\
&\qquad + q^{-ms}\sum_{t=1}^{\infty}\frac{q^{2jp\binom{t}{2}+p(j-2s-1)t}z^{jt}}{1-q^{2jp(t-1)+jp+jm}z^{j}}
 - q^{-m(j-s-1)} \sum_{t=1}^{\infty}\frac{q^{2jp\binom{t}{2}+p(j-2s-1)t}z^{jt}}{1-q^{2jp(t-1)+jp-jm}z^{j}}
 \notag\\
 &=  -q^{-m(j-s-1)} \sum_{t=-\infty}^{\infty}\frac{q^{2jp\binom{t}{2}+p(j- 2s-1)t}z^{jt}}{1-q^{2jp(t-1)+jp-jm}z^{j}}
 -q^{-ms} \sum_{t=-\infty}^{\infty}\frac{q^{2jp\binom{t}{2}+p(j-2s-1)t}z^{jt}}{1-q^{2jp(t-1)+jp+jm}z^{j}}\notag
\end{align}}%
Rewriting our find (\ref{equation:initSumOver_i_tPreAppellV2}) in terms of Appell functions gives the result.
\end{proof}


\section{Mock theta conjecture-like identities for odd-spin, $1/2$-level string functions}
\label{section:mockTheta12-level}
We give a new proof of identities (\ref{equation:mockThetaConj2512rPlus1-2ndA}) and (\ref{equation:mockThetaConj2512rPlus1-2ndmu}).  First we set up the machinery, and then we prove identities (\ref{equation:mockThetaConj2512rPlus1-2ndA}) and (\ref{equation:mockThetaConj2512rPlus1-2ndmu}).  We specialize Corollary \ref{corollary:generalPolarFiniteOddSpin1p} to $p=2$:
{\allowdisplaybreaks \begin{align*}
\chi_{2r+1}^{(2,5)}&(z;q)\\
& =   z^{-\frac{1}{2}}q^{\frac{1}{2}}C_{1,2r+1}^{(2,5)}(q) j(-z;q^{4})\\
& \quad -     \frac{1}{(q)_{\infty}^3}
q^{-\frac{1}{8}+\frac{2(r+1)^2}{5}}q^{-r}z^{-\frac{1}{2}} 
j(-z;q^{4})
 \times\Big (  j(-q^{9+4r};q^{20})  -q^{3-2r}j(-q^{-1+4r};q^{20})\Big )\\
& \qquad \times \Big ( q^{-1}m(-q^{-1},-q^{4}z^{-1};q^{4}) 
+  m(-q^{3},-z;q^{4})\Big ).
\end{align*} }%
Focusing on the two theta functions within the parentheses, we us (\ref{equation:j-elliptic}) to set up for, and then apply (\ref{equation:jsplit-m2}).  This reads
\begin{align*}
 j(-q^{9+4r};q^{20})  -q^{3-2r}j(-q^{-1+4r};q^{20})
& = j(-q^{9+4r};q^{20})  -q^{3-2r}q^{-1+4r}j(-q^{19+4r};q^{20})\\
& = j(q^{2+2r};q^{5}).
\end{align*} 
This gives us
{\allowdisplaybreaks \begin{align*}
\chi_{2r+1}^{(2,5)}(z;q)
& =   z^{-\frac{1}{2}}q^{\frac{1}{2}}C_{1,2r+1}^{(2,5)}(q) j(-z;q^{4})\\
& \quad -     \frac{1}{(q)_{\infty}^3}
q^{-\frac{1}{8}+\frac{2(r+1)^2}{5}}q^{-r}z^{-\frac{1}{2}} 
j(-z;q^{4})
 \times j(q^{2+2r};q^5)\\
&\qquad  \qquad \times \Big ( q^{-1}m(-q^{-1},-q^{4}z^{-1};q^{4}) 
+  m(-q^{3},-z;q^{4})\Big ).
\end{align*} }%
Rewriting the string function using (\ref{equation:mathCalCtoStringC}) brings us to
{\allowdisplaybreaks \begin{align}
\chi_{2r+1}^{(2,5)}(z;q)
& =   z^{-\frac{1}{2}}q^{-\frac{1}{8}+\frac{2(r+1)^2}{5}}\mathcal{C}_{1,2r+1}^{(2,5)}(q) j(-z;q^{4})
\label{equation:newFourier25}\\
& \quad -     \frac{1}{(q)_{\infty}^3}
q^{-\frac{1}{8}+\frac{2(r+1)^2}{5}}q^{-r}z^{-\frac{1}{2}} 
j(-z;q^{4})
 \times j(q^{2+2r};q^5)\notag\\
&\qquad  \qquad \times \Big ( q^{-1}m(-q^{-1},-q^{4}z^{-1};q^{4}) 
+  m(-q^{3},-z;q^{4})\Big ).\notag
\end{align} }%
We specialize Proposition \ref{proposition:WeylKac} to $(p,j,\ell)=(2,1,2r+1)$ to have
\begin{align}
\chi_{\ell}^{(2,5)}(z;q)
&=z^{-r-\frac{1}{2}}q^{-\frac{1}{8}+\frac{2(r+1)^2}{5}}
\frac{j(-q^{4r+14}z^{-5};q^{20})
-z^{2r+2}j(-q^{-4r+6}z^{-5};q^{20}) }
{j(z;q)}.
\label{equation:WeylKac25}
\end{align}
Comparing (\ref{equation:WeylKac25}) and (\ref{equation:newFourier25}) gives
{\allowdisplaybreaks \begin{align*}
z^{-r}&
\frac{j(-q^{4r+14}z^{-5};q^{20})
-z^{2r+2}j(-q^{-4r+6}z^{-5};q^{20}) }
{j(z;q)}\\
& =   \mathcal{C}_{1,2r+1}^{(2,5)}(q) j(-z;q^{4})\\
& \quad -     \frac{1}{(q)_{\infty}^3}
q^{-r}
j(-z;q^{4})
 j(q^{2+2r};q^5)
  \Big ( q^{-1}m(-q^{-1},-q^{4}z^{-1};q^{4}) 
+  m(-q^{3},-z;q^{4})\Big ).
\end{align*} }%
Solving for the string function results in
{\allowdisplaybreaks \begin{align}
\mathcal{C}_{1,2r+1}^{(2,5)}(q)
&=z^{-r}
\frac{j(-q^{4r+14}z^{-5};q^{20})
-z^{2r+2}j(-q^{-4r+6}z^{-5};q^{20}) }
{j(z;q) j(-z;q^{4})}\label{equation:generalMockThetaConjLevel12}
\\
&\qquad +    \frac{1}{(q)_{\infty}^3}
q^{-r}
 j(q^{2+2r};q^5)
  \Big ( q^{-1}m(-q^{-1},-q^{4}z^{-1};q^{4}) 
+  m(-q^{3},-z;q^{4})\Big ).\notag
\end{align} }%


We prove identity (\ref{equation:mockThetaConj2512rPlus1-2ndA}).  In (\ref{equation:generalMockThetaConjLevel12}), we set $z=-q^2$ and use (\ref{equation:j-elliptic}) to get
{\allowdisplaybreaks \begin{align*}
\mathcal{C}_{1,2r+1}^{(2,5)}(q)
&=(-1)^{r}q^{-2r}
\frac{j(q^{4r+4};q^{20})
-q^{4r+4}j(q^{-4r-4};q^{20}) }
{j(-q^2;q) j(q^2;q^{4})}\\
&\qquad +    \frac{1}{(q)_{\infty}^3}
q^{-r}
 j(q^{2+2r};q^5)
  \Big ( q^{-1}m(-q^{-1},q^2;q^{4}) 
+  m(-q^{3},q^2;q^{4})\Big ).
\end{align*} }%
We use (\ref{equation:mxqz-fnq-x}) to get
{\allowdisplaybreaks \begin{align*}
\mathcal{C}_{1,2r+1}^{(2,5)}(q)
&=(-1)^{r}q^{-2r}
\frac{j(q^{4r+4};q^{20})
-q^{4r+4}j(q^{-4r-4};q^{20}) }
{j(-q^2;q) j(q^2;q^{4})}\\
&\qquad +    \frac{1}{(q)_{\infty}^3}
q^{-r}
 j(q^{2+2r};q^5)
  \Big ( 1+2q^{-1} m(-q^{-1},q^2;q^{4})\Big ).
\end{align*} }%
Using (\ref{equation:mxqz-flip}) and (\ref{equation:mxqz-fnq-z}) gives
{\allowdisplaybreaks \begin{align*}
\mathcal{C}_{1,2r+1}^{(2,5)}(q)
&=(-1)^{r}q^{-2r}
\frac{j(q^{4r+4};q^{20})
-q^{4r+4}j(q^{-4r-4};q^{20}) }
{j(-q^2;q) j(q^2;q^{4})}\\
&\qquad +    \frac{1}{(q)_{\infty}^3}
q^{-r}
 j(q^{2+2r};q^5)
  \Big ( 1-2m(-q,q^2;q^{4})\Big ).
\end{align*} }%
Using (\ref{equation:j-flip}) and (\ref{equation:j-elliptic})
{\allowdisplaybreaks \begin{align*}
\mathcal{C}_{1,2r+1}^{(2,5)}(q)
&=(-1)^{r}2q^{-2r}
\frac{j(q^{4r+4};q^{20})}
{j(-q^2;q) j(q^2;q^{4})}
+    \frac{1}{(q)_{\infty}^3}
q^{-r}
 j(q^{2+2r};q^5)
  \Big ( 1-2m(-q,q^2;q^{4})\Big ).
\end{align*} }%
Elementary product rearrangements and (\ref{equation:2nd-A(q)}) give
 \begin{equation*}
(q)_{\infty}^3\mathcal{C}_{1,2r+1}^{(2,5)}(q)
=(-1)^{r}q^{1-2r}
\frac{J_{1}^4J_{4}}{{J_{2}^4}}j(q^{4r+4};q^{20})
+   
q^{-r}j(q^{2+2r};q^5)
  \Big ( 1+2A_2(-q)\Big ).
\end{equation*}


We prove identity (\ref{equation:mockThetaConj2512rPlus1-2ndmu}).  In (\ref{equation:generalMockThetaConjLevel12}), we set $z=-q$ to get
{\allowdisplaybreaks \begin{align*}
\mathcal{C}_{1,2r+1}^{(2,5)}(q)
&=(-1)^rq^{-r}
\frac{j(q^{4r+9};q^{20})
-q^{2r+2}j(q^{-4r+1};q^{20}) }
{j(-q;q) j(q;q^{4})}
\\
&\qquad +    \frac{1}{(q)_{\infty}^3}
q^{-r}
 j(q^{2+2r};q^5)
  \Big ( q^{-1}m(-q^{-1},q^{3};q^{4}) 
+  m(-q^{3},q;q^{4})\Big ).
\end{align*} }%
We use (\ref{equation:j-flip}) and (\ref{equation:jsplit-m2}) to get
{\allowdisplaybreaks \begin{align*}
\mathcal{C}_{1,2r+1}^{(2,5)}(q)
&=(-1)^rq^{-r}
\frac{j(q^{2+2r};-q^5)}
{j(-q;q) j(q;q^{4})}
\\
&\qquad +    \frac{1}{(q)_{\infty}^3}
q^{-r}
 j(q^{2+2r};q^5)
  \Big ( q^{-1}m(-q^{-1},q^{3};q^{4}) 
+  m(-q^{3},q;q^{4})\Big ).
\end{align*} }%
Elementary product rearrangements give
{\allowdisplaybreaks \begin{align*}
(q)_{\infty}^3\mathcal{C}_{1,2r+1}^{(2,5)}(q)
&=(-1)^{r}\frac{q^{-r}}{2}
\frac{J_{1}^3}{J_{2}J_{4}}j(q^{2r+2};-q^{5})
\\
&\qquad +  q^{-r}
 j(q^{2+2r};q^5)
  \Big ( q^{-1}m(-q^{-1},q^{3};q^{4}) 
+  m(-q^{3},q;q^{4})\Big ).
\end{align*} }%
Using (\ref{equation:mxqz-flip}) and (\ref{equation:mxqz-fnq-z}) gives
{\allowdisplaybreaks \begin{align*}
(q)_{\infty}^3\mathcal{C}_{1,2r+1}^{(2,5)}(q)
&=(-1)^{r}\frac{q^{-r}}{2}
\frac{J_{1}^3}{J_{2}J_{4}}j(q^{2r+2};-q^{5})
\\
&\qquad +  q^{-r}
 j(q^{2+2r};q^5)
  \Big ( -m(-q,q;q^{4}) 
+  m(-q^{3},q;q^{4})\Big ).
\end{align*} }%
Using (\ref{equation:mxqz-fnq-x}) and (\ref{equation:mxqz-flip}) gives
{\allowdisplaybreaks \begin{align*}
(q)_{\infty}^3\mathcal{C}_{1,2r+1}^{(2,5)}(q)
&=(-1)^{r}\frac{q^{-r}}{2}
\frac{J_{1}^3}{J_{2}J_{4}}j(q^{2r+2};-q^{5})
\\
&\qquad +  q^{-r}
 j(q^{2+2r};q^5)
  \Big ( 1-m(-q,q;q^{4}) 
-m(-q,q^{-1};q^{4})\Big ).
\end{align*} }%
Corollary \ref{corollary:mxqz-flip-xz} then gives
{\allowdisplaybreaks \begin{align*}
(q)_{\infty}^3\mathcal{C}_{1,2r+1}^{(2,5)}(q)
&=(-1)^{r}\frac{q^{-r}}{2}
\frac{J_{1}^3}{J_{2}J_{4}}j(q^{2r+2};-q^{5})
\\
&\qquad +  q^{-r}
 j(q^{2+2r};q^5)
  \Big ( 1-m(-q,q;q^{4}) -m(-q,-1;q^{4}) 
\Big ).
\end{align*} }%
Identity (\ref{equation:2nd-mu(q)}) gives
\begin{equation}
(q)_{\infty}^3\mathcal{C}_{1,2r+1}^{(2,5)}(q)
=(-1)^{r}\frac{q^{-r}}{2}
\frac{J_{1}^3}{J_{2}J_{4}}j(q^{2r+2};-q^{5})
+ q^{-r}
 j(q^{2+2r};q^5) \Big ( 1-\frac{1}{2}\mu_{2}(q)\Big ).
\end{equation}


\section{Mock theta conjecture-like identities for odd-spin, $1/3$-level string functions}
\label{section:mockTheta13-level}
We prove Corollary \ref{corollary:newMockThetaIdentitiespP37m1ell2rPlus1}.  We specialize Corollary \ref{corollary:generalPolarFiniteOddSpin1p} to $p=3$:
{\allowdisplaybreaks \begin{align*}
\chi_{2r+1}^{(3,7)}(z;q)
& =   z^{-\frac{1}{2}}q^{\frac{3}{4}}C_{1,2r+1}^{(3,7)}(q)
    j(-z;q^{6})\\
& \quad -     \frac{1}{(q)_{\infty}^3}
q^{-\frac{1}{8}+\frac{3(r+1)^2}{7}}q^{3-3r}z^{-\frac{1}{2}} 
j(-z;q^{6})
 \sum_{m=1}^{2}(-1)^{m}q^{\binom{m+1}{2}+m(r-3)}\\
& \qquad \times\Big (  j(-q^{7m +6(r+1)};q^{42}) 
 -q^{7m-2m(r+1)}j(-q^{-7m+6(r+1)};q^{42})\Big )\\
& \qquad \times \Big ( q^{m-3}m(-q^{m-3},-q^{6}z^{-1};q^{6}) 
+  m(-q^{m+3},-z;q^{6})\Big ).
\end{align*} }%

We expand the sum over $m$ and change the string function notation with (\ref{equation:mathCalCtoStringC}).  Rearranging terms gives
{\allowdisplaybreaks \begin{align*}
(q)_{\infty}^3\mathcal{C}_{1,2r+1}^{(3,7)}(q)
 &=\frac{(q)_{\infty}^3}{j(-z;q^{6})}z^{\frac{1}{2}}q^{\frac{1}{8}-\frac{3(r+1)^2}{7}}\chi_{2r+1}^{(3,7)}(z;q)\\
& \quad +   q^{3-3r}
\Big [ 
-q^{-2+r}  \left (  j(-q^{13 +6r};q^{42}) 
 -q^{5-2r}j(-q^{-1+6r};q^{42})\right )\\
& \qquad \quad \times \left ( q^{-2}m(-q^{-2},-q^{6}z^{-1};q^{6}) 
+  m(-q^{4},-z;q^{6})\right )\\
&\qquad + q^{-3+2r} \left (  j(-q^{20 +6r};q^{42} )
-q^{10-4r}j(-q^{-8+6r};q^{42})\right ) \\
& \qquad \quad \times \left ( q^{-1}m(-q^{-1},-q^{6}z^{-1};q^{6}) 
+  m(-q^{5},-z;q^{6})\right )
\Big ].
\end{align*} }%

We use (\ref{equation:j-flip}) and (\ref{equation:j-elliptic}) to set up for the quintuple product identity (\ref{equation:quintuple}).  This gives
{\allowdisplaybreaks \begin{align*}
(q)_{\infty}^3\mathcal{C}_{1,2r+1}^{(3,7)}(q)
 &=\frac{(q)_{\infty}^3}{j(-z;q^{6})}z^{\frac{1}{2}}q^{\frac{1}{8}-\frac{3(r+1)^2}{7}}\chi_{2r+1}^{(3,7)}(z;q)\\
& \quad +   q^{3-3r}
\Big [ 
-q^{-2+r}  \left (  j(-q^{29 -6r};q^{42}) 
 -q^{5-2r}j(-q^{43-6r};q^{42})\right )\\
& \qquad \quad \times \left ( q^{-2}m(-q^{-2},-q^{6}z^{-1};q^{6}) 
+  m(-q^{4},-z;q^{6})\right )\\
&\qquad + q^{-3+2r} \left (  j(-q^{20 +6r};q^{42} )
-q^{2+2r}j(-q^{34+6r};q^{42})\right ) \\
& \qquad \quad \times \left ( q^{-1}m(-q^{-1},-q^{6}z^{-1};q^{6}) 
+  m(-q^{5},-z;q^{6})\right )
\Big ].
\end{align*} }%

Applying the quintuple product identity (\ref{equation:quintuple}) brings us to
{\allowdisplaybreaks \begin{align*}
(q)_{\infty}^3\mathcal{C}_{1,2r+1}^{(3,7)}(q)
 &=\frac{(q)_{\infty}^3}{j(-z;q^{6})}z^{\frac{1}{2}}q^{\frac{1}{8}-\frac{3(r+1)^2}{7}}\chi_{2r+1}^{(3,7)}(z;q)\\
& \quad 
-q^{1-2r}  \frac{j(q^{5-2r};q^{14};q^{24-4r};q^{28})}{J_{28}}
 \times \left ( q^{-2}m(-q^{-2},-q^{6}z^{-1};q^{6}) 
+  m(-q^{4},-z;q^{6})\right )\\
&\quad + q^{-r} \frac{j(q^{2+2r};q^{14};q^{18+4r};q^{28})}{J_{28}}
\times \left ( q^{-1}m(-q^{-1},-q^{6}z^{-1};q^{6}) 
+  m(-q^{5},-z;q^{6})\right ).
\end{align*} }%

Setting $z=-q^2$ brings us to
{\allowdisplaybreaks \begin{align*}
(q)_{\infty}^3\mathcal{C}_{1,2r+1}^{(3,7)}(q)
 &=\frac{(q)_{\infty}^3}{j(q^{2};q^{6})}(-q^{2})^{\frac{1}{2}}q^{\frac{1}{8}-\frac{3(r+1)^2}{7}}\chi_{2r+1}^{(3,7)}(-q^{2};q)\\
& \quad 
-q^{1-2r}  \frac{j(q^{5-2r};q^{14};q^{24-4r};q^{28})}{J_{28}}
 \times \left ( q^{-2}m(-q^{-2},q^{4};q^{6}) 
+  m(-q^{4},q^2;q^{6})\right )\\
&\quad + q^{-r} \frac{j(q^{2+2r};q^{14};q^{18+4r};q^{28})}{J_{28}}
\times \left ( q^{-1}m(-q^{-1},q^{4};q^{6}) 
+  m(-q^{5},q^2;q^{6})\right ).
\end{align*} }%
Using (\ref{equation:mxqz-flip}) and (\ref{equation:mxqz-fnq-x}) gives
{\allowdisplaybreaks \begin{align*}
(q)_{\infty}^3\mathcal{C}_{1,2r+1}^{(3,7)}(q)
 &=\frac{J_{1}^3}{J_{2}}(-q^{2})^{\frac{1}{2}}q^{\frac{1}{8}-\frac{3(r+1)^2}{7}}\chi_{2r+1}^{(3,7)}(-q^{2};q)\\
& \quad 
-q^{1-2r}  \frac{j(q^{5-2r};q^{14};q^{24-4r};q^{28})}{J_{28}}
 \times \left (-m(-q^{2},q^{-4};q^{6}) 
+ 1+q^{-2}m(-q^{-2},q^2;q^{6})\right )\\
&\quad + q^{-r} \frac{j(q^{2+2r};q^{14};q^{18+4r};q^{28})}{J_{28}}
\times \left ( -m(-q,q^{-4};q^{6}) 
+ 1+q^{-1} m(-q^{-1},q^2;q^{6})\right ).
\end{align*} }%
Using (\ref{equation:mxqz-fnq-x}) and (\ref{equation:mxqz-fnq-z}) gives
{\allowdisplaybreaks \begin{align*}
(q)_{\infty}^3\mathcal{C}_{1,2r+1}^{(3,7)}(q)
 &=\frac{J_{1}^3}{J_{2}}(-q^{2})^{\frac{1}{2}}q^{\frac{1}{8}-\frac{3(r+1)^2}{7}}\chi_{2r+1}^{(3,7)}(-q^{2};q)\\
& \quad 
-q^{1-2r}  \frac{j(q^{5-2r};q^{14};q^{24-4r};q^{28})}{J_{28}}
 \times \left (-m(-q^{2},q^{2};q^{6}) 
+ 1-m(-q^{2},q^{4};q^{6})\right )\\
&\quad + q^{-r} \frac{j(q^{2+2r};q^{14};q^{18+4r};q^{28})}{J_{28}}
\times \left ( -m(-q,q^{2};q^{6}) 
+ 1-m(-q,q^{4};q^{6})\right ).
\end{align*} }%
Rewriting the Appell functions using  (\ref{equation:3rd-f(q)}) and (\ref{equation:3rd-omega(q)}) takes us to
{\allowdisplaybreaks \begin{align}
(q)_{\infty}^3\mathcal{C}_{1,2r+1}^{(3,7)}(q)
 &=\frac{J_{1}^3}{J_{2}}(-q^{2})^{\frac{1}{2}}q^{\frac{1}{8}-\frac{3(r+1)^2}{7}}\chi_{2r+1}^{(3,7)}(-q^{2};q)
 \label{equation:fourierFinal37}\\
& \quad 
-q^{1-2r}  \frac{j(q^{5-2r};q^{14};q^{24-4r};q^{28})}{J_{28}}
 \times \left (1-\frac{1}{2}f_{3}(q^2)\right )
 \notag \\
&\quad + q^{-r} \frac{j(q^{2+2r};q^{14};q^{18+4r};q^{28})}{J_{28}}
\times \left (1-q\omega_{3}(-q)\right ).
\notag
\end{align} }%

We take the appropriate specialization of Proposition \ref{proposition:WeylKac}:
\begin{equation*}
\chi_{2r+1}^{(3,7)}(z;q)
=z^{-r-1}q^{3\frac{(2r+2)^2}{28}}
\frac{j(-q^{6r+27}z^{-7};q^{42})
-z^{2r+2}j(-q^{-6r+15}z^{-7};q^{42}) }
{z^{-\frac{1}{2}}q^{\frac{1}{8}}j(z;q)}.
\end{equation*}
Setting $z=-q^2$ gives
\begin{equation*}
\chi_{2r+1}^{(3,7)}(-q^2;q)
=(-q^{2})^{-r-\frac{1}{2}}q^{-\frac{1}{8}+3\frac{(r+1)^2}{7}}
\frac{j(-q^{6r+13};q^{42})
-q^{4r+4}j(-q^{-6r+1};q^{42}) }
{j(-q^2;q)}.
\end{equation*}
Using (\ref{equation:j-flip}) and (\ref{equation:j-elliptic}) gives
\begin{equation*}
\chi_{2r+1}^{(3,7)}(-q^2;q)
=(-q^{2})^{-r-\frac{1}{2}}q^{-\frac{1}{8}+3\frac{(r+1)^2}{7}}
\frac{j(-q^{29-6r};q^{42})
-q^{5-2r}j(-q^{43-6r};q^{42}) }
{j(-q^2;q)}.
\end{equation*}
Using the quintuple product identity (\ref{equation:quintuple}) and then (\ref{equation:j-elliptic}) gives
\begin{equation}
\chi_{2r+1}^{(3,7)}(-q^2;q)
=(-q^{2})^{-r-\frac{1}{2}}q^{1-\frac{1}{8}+3\frac{(r+1)^2}{7}}
\frac{j(-q^{5-2r};q^{14})j(q^{24-4r};q^{28})}
{j(-1;q)J_{28}}.
\label{equation:weylKacFinal37}
\end{equation}
\noindent Inserting (\ref{equation:weylKacFinal37}) into (\ref{equation:fourierFinal37}) and isolating the string function gives the result.


\section{Odd-spin $1/2$-level and $1/3$-level string functions through cross-spin identities}\label{section:crossSpinExamples}
In this section we use our cross-spin identity Theorem \ref{theorem:crossSpin-j-Odd} and corresponding results for even-spin in \cite{BoMo2025} to confirm the results in Corollaries \ref{corollary:newMockThetaIdentitiespP25m1ell2rPlus1} and \ref{corollary:newMockThetaIdentitiespP37m1ell2rPlus1}.   There is some overlap with what one finds in \cite[Section 8]{BoMo2025}, but here we carry out the details.

\smallskip
To relate odd-spin $1/2$-level string functions to even-spin $1/2$-level string functions, we use the following specialization of Theorem \ref{theorem:crossSpin-j-Odd}: 
\begin{align*}
(q)_{\infty}^3\mathcal{C}_{2i-1,2r-1}^{(2,5)}(q)
&=-q^{2(i-r)+1}(q)_{\infty}^{3}\mathcal{C}_{2i-2,4-2r}^{(2,5)}(q)\\
&\qquad -q^{i+r}
 \times
\left ( j(-q^{5-4r};q^{20})
-q^{-2r}j(-q^{5+4r};q^{20})\right ).
\end{align*}
Using (\ref{equation:j-flip}) and then (\ref{equation:jsplit-m2}) produces
\begin{align*}
(q)_{\infty}^3\mathcal{C}_{2i-1,2r-1}^{(2,5)}(q)
&=-q^{2(i-r)+1}(q)_{\infty}^{3}\mathcal{C}_{2i-2,4-2r}^{(2,5)}(q)
 -q^{i+r}
 \times j(q^{-2r};q^{5}).
\end{align*}
Again using (\ref{equation:j-flip}) and (\ref{equation:j-elliptic}) brings us to
\begin{align*}
(q)_{\infty}^3\mathcal{C}_{2i-1,2r-1}^{(2,5)}(q)
&=-q^{2(i-r)+1}(q)_{\infty}^{3}\mathcal{C}_{2i-2,4-2r}^{(2,5)}(q)
 +q^{i-r}
 \times j(q^{2r};q^{5}).
\end{align*}
Let us set $i=1$ and replace $r$ with $r+1$.  This yields
\begin{equation}
(q)_{\infty}^3\mathcal{C}_{1,2r+1}^{(2,5)}(q)
=-q^{1-2r}(q)_{\infty}^{3}\mathcal{C}_{0,2(1-r)}^{(2,5)}(q)
 +q^{-r}
 \times j(q^{2r+2};q^{5}).\label{equation:master25crossSpin}
\end{equation}
 The proofs for (\ref{equation:mockThetaConj2512rPlus1-2ndA}) and (\ref{equation:mockThetaConj2512rPlus1-2ndmu}) are similar, so we will only do the second.  We substitute (\ref{equation:mockThetaConj2502r-2ndmu}) into (\ref{equation:master25crossSpin}) to get
\begin{align*}
(q)_{\infty}^3\mathcal{C}_{1,2r+1}^{(2,5)}(q)
&=-q^{1-2r} \left( (-q)^{r-1}\frac{1}{2}
\frac{J_{1}^3}{J_{2}J_{4}}j(-q^{3-2r};-q^{5})
+ q^{r-1} \frac{1}{2}  j(q^{3-2r};q^{5}) \mu_2(q) \right) \\
&\qquad 
 +q^{-r} \times j(q^{2r+2};q^{5}).
\end{align*}
Simplifying, using (\ref{equation:j-flip}), and regrouping terms brings us to
\begin{align*}
(q)_{\infty}^3\mathcal{C}_{1,2r+1}^{(2,5)}(q)
&=(-1)^r q^{-r}\frac{1}{2}
\frac{J_{1}^3}{J_{2}J_{4}}j(q^{2+2r};-q^{5})
+q^{-r} \frac{1}{2}  j(q^{2+2r};q^{5})\left (\frac{1}{2}  -\mu_2(q) \right ).
\end{align*}


To relate odd-spin $1/3$-level string functions to even-spin $1/3$-level string functions, we use the following specialization of Theorem \ref{theorem:crossSpin-j-Odd}:
\begin{align*}
(q)_{\infty}^3\mathcal{C}_{2i-1,2r-1}^{(3,7)}(q)
&=q^{3(1+i-r)}(q)_{\infty}^{3}\mathcal{C}_{2i-2,6-2r}^{(3,7)}(q)\\
&\qquad  -q^{1+2(i-r)}
\frac{j(q^{7+2r};q^{14})j(q^{4r};q^{28})}{J_{28}}
+q^{i-r}
\frac{j(q^{2r};q^{14})j(q^{14+4r};q^{28})}{J_{28}}.
\end{align*}
We replace $i=1$ and $r$ with $r+1$ to get
\begin{align*}
(q)_{\infty}^3\mathcal{C}_{1,2r+1}^{(3,7)}(q)
&=q^{3(1-r)}(q)_{\infty}^{3}\mathcal{C}_{0,2(2-r)}^{(3,7)}(q)\\
&\qquad  -q^{1-2r}
\frac{j(q^{9+2r};q^{14})j(q^{4r+4};q^{28})}{J_{28}}
+q^{-r}
\frac{j(q^{2r+2};q^{14})j(q^{18+4r};q^{28})}{J_{28}}.
\end{align*}
We substitute in the value in Theorem \ref{theorem:newMockThetaIdentitiespP37m0ell2r} to get
\begin{align*}
(q)_{\infty}^3\mathcal{C}_{1,2r+1}^{(3,7)}(q)
&=q^{3(1-r)}
\Big [ 
(-q)^{-2+r}\frac{(q)_{\infty}^3}{J_{2}}
\frac{j(-q^{1+2(2-r)};q^{14})j(q^{16+4(2-r)};q^{28})}
{j(-1;q)J_{28}} \\
&\qquad   -q^{2-2(2-r)}\frac{j(q^{6-2(2-r)};q^{14})j(q^{26-4(2-r)};q^{28})}{J_{28}} \omega_3(-q)\\
&\qquad +  q^{-2+r} \frac{j(q^{1+2(2-r)};q^{14})j(q^{16+4(2-r)};q^{28})}{J_{28}} 
\frac{1}{2}f_{3}(q^2)
\Big ]\\
&\qquad  -q^{1-2r}
\frac{j(q^{9+2r};q^{14})j(q^{4r+4};q^{28})}{J_{28}}
 +q^{-r}
\frac{j(q^{2r+2};q^{14})j(q^{18+4r};q^{28})}{J_{28}}.
\end{align*}
Distributing takes us to
\begin{align*}
(q)_{\infty}^3\mathcal{C}_{1,2r+1}^{(3,7)}(q)
&=
(-1)^{r}q^{1-2r}\frac{(q)_{\infty}^3}{J_{2}}
\frac{j(-q^{5-2r};q^{14})j(q^{24-4r};q^{28})}
{j(-1;q)J_{28}} \\
&\qquad   -q^{1-r}\frac{j(q^{2+2r};q^{14})j(q^{18+4r};q^{28})}{J_{28}} \omega_3(-q)\\
&\qquad +  q^{1-2r} \frac{j(q^{5-2r};q^{14})j(q^{24-4r};q^{28})}{J_{28}} 
\frac{1}{2}f_{3}(q^2)\\
&\qquad  -q^{1-2r}
\frac{j(q^{9+2r};q^{14})j(q^{4r+4};q^{28})}{J_{28}}
 +q^{-r}
\frac{j(q^{2r+2};q^{14})j(q^{18+4r};q^{28})}{J_{28}}.
\end{align*}
Combine terms using (\ref{equation:j-flip}) brings us to
\begin{align*}
(q)_{\infty}^3\mathcal{C}_{1,2r+1}^{(3,7)}(q)
&=
(-1)^{r}q^{1-2r}\frac{(q)_{\infty}^3}{J_{2}}
\frac{j(-q^{5-2r};q^{14})j(q^{24-4r};q^{28})}
{j(-1;q)J_{28}} \\
&\qquad   +q^{-r}\frac{j(q^{2+2r};q^{14})j(q^{18+4r};q^{28})}{J_{28}}\left ( 1-q \omega_3(-q)\right)\\
&\qquad -  q^{1-2r} \frac{j(q^{9+2r};q^{14})j(q^{4+4r};q^{28})}{J_{28}} 
\left ( 1-\frac{1}{2}f_{3}(q^2)\right ).
\end{align*}


\section{Mock theta conjecture-like identities for odd-spin, $2/3$-level string functions}
\label{section:mockTheta23-level}
In order to prove Theorems \ref{theorem:newMockThetaIdentitiespP38m1ell2rPlus1} and \ref{theorem:newMockThetaIdentitiespP38m1ell2rPlus1Alt}, we will first need to take the appropriate specialization of our polar-finite decomposition Theorem \ref{theorem:generalPolarFiniteOddSpin}.  This will give us
\begin{proposition} \label{proposition:polarFinite23OddSpin} We have
{\allowdisplaybreaks \begin{align*}
&q^{\frac{1}{8}-\frac{3(r+1)^2}{8}}\chi_{2r+1}^{(3,8)}(z;q)\\
& =  z^{-\frac{1}{2}}\mathcal{C}_{1,2r+1}^{(3,8)}(q) j(-z^{2}q^{3};q^{12})
+z^{-\frac{3}{2}}\mathcal{C}_{3,2r+1}^{(3,8)}(q) j(-z^{2}q^{-3};q^{12})\\
& \quad -     \frac{1}{(q)_{\infty}^3}z^{-\frac{1}{2}} 
j(-q^{3}z^{2};q^{12})\\
& \qquad  \times \Big [ -q^{1-2r}
\frac{j(q^{6-2r};q^{16})j(q^{28-4r};q^{32})}{J_{32}}
 \Big ( -q^{-1}m(-q,-q^{3}z^{2};q^{12}) 
+  m(-q^{5},-q^{3}z^{2};q^{12})\Big )\\
& \qquad \quad +q^{-r}
\frac{j(q^{2+2r};q^{16})j(q^{20+4r};q^{32})}{J_{32}}
  \Big ( -q^{-2}m(-q^{-1},-q^{3}z^{2};q^{12}) 
+  m(-q^{7},-q^{3}z^{2};q^{12})\Big )\Big ]\\
& \quad -     \frac{1}{(q)_{\infty}^3}z^{-\frac{3}{2}} 
j(-q^{-3}z^{2};q^{12})\\
& \qquad  \times \Big [ -q^{4-2r}
\frac{j(q^{6-2r};q^{16})j(q^{28-4r};q^{32})}{J_{32}}
 \Big ( -m(-q^{7},-q^{9}z^{2};q^{12}) 
 + q^{-1} m(-q^{11},-q^{9}z^{2};q^{12})\Big )\\
& \qquad \quad + q^{3-r}
\frac{j(q^{2+2r};q^{16})j(q^{20+4r};q^{32})}{J_{32}}
 \Big ( -m(-q^{5},-q^{9}z^{2};q^{12}) 
 + q^{-2} m(-q^{13},-q^{9}z^{2};q^{12})\Big )\Big ].
\end{align*} }%
\end{proposition}

By choosing a value of $z$ so that the coefficient of a given family of string functions vanishes, we see that instead of vanishing \cite[Lemmas 12.2, 12.3]{BoMo2025}, the corresponding Appell function expression becomes a simple quotient of theta functions.  This happens for $z=iq^{-3/2}$, and for $z=iq^{3/2}$.  Indeed, we have

\begin{lemma}\label{lemma:polarFinite23m1AppellVanish}  We have
\begin{gather*}
\lim_{z\to iq^{3/2}} j(-q^{-3}z^2;q^{12})
\left (-m(-q^{7},-q^{9}z^2;q^{12})+q^{-1}m(-q^{11},-q^{9}z^2;q^{12}) \right ) 
=-q^{-1}\frac{J_{1}^2J_{4}J_{6}^2}{J_{2}^2J_{3}},\\
\lim_{z\to iq^{3/2}} j(-z^2q^{-3};q^{12})\left ( 
-m(-q^{5},-q^{-3}z^2;q^{12})+q^{-2}m(-q^{13},-q^{-3}z^2;q^{12})
\right )
=-q^{-1}\frac{J_{1}^2J_{4}J_{6}^2}{J_{2}^2J_{3}}. 
\end{gather*}
\end{lemma}

\begin{lemma}\label{lemma:polarFinite23m3AppellVanish} We have
\begin{gather*}
\lim_{z\to iq^{-3/2}}j(-q^3z^2;q^{12})\left ( -q^{-1}m(-q,-q^{3}z^2;q^{12})+m(-q^{5},-q^{3}z^2;q^{12})\right )
=q^{-1}\frac{J_{1}^2J_{4}J_{6}^2}{J_{2}^2J_{3}},\\ 
\lim_{z\to iq^{-3/2}}j(-q^{3}z^2;q^{12})\left ( -q^{-2}m(-q^{-1},-q^{3}z^2;q^{12})+m(-q^{7},-q^{3}z^2;q^{12})\right )
=q^{-1}\frac{J_{1}^2J_{4}J_{6}^2}{J_{2}^2J_{3}}. 
\end{gather*}
\end{lemma}

For $z=iq^{3/2}$ and $z=iq^{-3/2}$, we then need to evaluate the character on the left-hand side.  Then it is straightforward to solve for the string function.  We will only do this for the proofs of Theorems \ref{theorem:newMockThetaIdentitiespP38m1ell2rPlus1} and \ref{theorem:newMockThetaIdentitiespP38m1ell2rPlus1Alt}, for which we will need Lemma \ref{lemma:polarFinite23m1AppellVanish}.  If we were to use this route to prove the identities in Corollaries \ref{corollary:newMockThetaIdentitiespP38m3ell2rPlus1} and \ref{corollary:newMockThetaIdentitiespP38m3ell2rPlus1Alt} then we would use Lemma \ref{lemma:polarFinite23m3AppellVanish}.  For the Corollaries, we will instead be using the cross-spin identity in Theorem \ref{theorem:crossSpin-j-Odd}.  So Lemma \ref{lemma:polarFinite23m3AppellVanish} is just stated for completeness and the curious reader.
\begin{proposition}\label{proposition:weylKac23ell2rzVal} We have
\begin{gather}
\chi_{2r+1}^{(3,8)}(iq^{3/2};q)
=-i(iq^{3/2})^{-r-\frac{1}{2}}q^{-\frac{1}{8}+3\frac{(r+1)^2}{8}}
q^{\frac{1}{2}}\frac{j((-q^{3})^{r+1};q^{12})}
{J_{2,4}},
 \label{equation:weylKac23ell2rPlus1zVal1}\\
\chi_{2r+1}^{(3,8)}(iq^{-3/2};q)
=-(-1)^{r}(iq^{-3/2})^{-r-\frac{1}{2}}q^{-\frac{1}{8}+3\frac{(r+1)^2}{8}}
q^{-3r-1}
\frac{j((-q^3)^{(r+1)};q^{12})}
{J_{2,4}}.
 \label{equation:weylKac23ell2rPlus1zVal2}
\end{gather}
\end{proposition}

\begin{proof}[Proof of Proposition \ref{proposition:polarFinite23OddSpin}]
Let us specialize Theorem \ref{theorem:generalPolarFiniteOddSpin}  to $p=3, j=2$:
{\allowdisplaybreaks \begin{align*}
&\chi_{2r+1}^{(3,8)}(z;q)\\
& =    \sum_{s=0}^{1}z^{-\frac{1}{2}(2s+1)}q^{\frac{3}{8}(2s+1)^2}C_{2s+1,2r+1}^{(3,8)}(q)
    j(-z^{2}q^{3(1-2s)};q^{12})\\
& \quad -     \frac{1}{(q)_{\infty}^3}
\sum_{s=0}^{1}q^{-\frac{1}{8}+\frac{3(2r+2)^2}{32}}q^{3-3(r-s)}z^{-\frac{1}{2}(2s+1)} 
j(-q^{3(1-2s)}z^{2};q^{12})\\
& \qquad  \times \sum_{m=1}^{2}(-1)^{m}q^{\binom{m+1}{2}+m(r-3)}
 \times\Big (  j(-q^{8m+6r+6};q^{48}) 
 -q^{6m-2mr}j(-q^{-8m+6r+6};q^{48})\Big )\\
& \qquad \qquad \times \Big ( q^{m(s+1)-3(2s+1)}m(-q^{2m-3(2s+1)},-q^{3(3+2s)}z^{-2};q^{12}) \\
&\qquad  \qquad \qquad + q^{-ms} m(-q^{2m+3(2s+1)},-q^{3(1-2s)}z^{2};q^{12})\Big ).
\end{align*}}%
Expanding the sum over $s$ and then using (\ref{equation:mxqz-fnq-z}) gives
{\allowdisplaybreaks \begin{align*}
&\chi_{2r+1}^{(3,8)}(z;q)\\
& =  z^{-\frac{1}{2}}q^{\frac{3}{8}}C_{1,2r+1}^{(3,8)}(q) j(-z^{2}q^{3};q^{12})
+z^{-\frac{3}{2}}q^{\frac{27}{8}}C_{3,2r+1}^{(3,8)}(q) j(-z^{2}q^{-3};q^{12})\\
& \quad -     \frac{1}{(q)_{\infty}^3}
q^{-\frac{1}{8}+\frac{3(2r+2)^2}{32}}q^{3-3r}z^{-\frac{1}{2}} 
j(-q^{3}z^{2};q^{12})\\
& \qquad  \times \sum_{m=1}^{2}(-1)^{m}q^{\binom{m+1}{2}+m(r-3)}
 \times\Big (  j(-q^{8m+6r+6};q^{48}) 
 -q^{6m-2mr}j(-q^{-8m+6r+6};q^{48})\Big )\\
& \qquad \qquad \times \Big ( q^{m-3}m(-q^{2m-3},-q^{9}z^{-2};q^{12}) 
+  m(-q^{2m+3},-q^{3}z^{2};q^{12})\Big )\\
& \quad -     \frac{1}{(q)_{\infty}^3}
q^{-\frac{1}{8}+\frac{3(2r+2)^2}{32}}q^{6-3r}z^{-\frac{3}{2}} 
j(-q^{-3}z^{2};q^{12})\\
& \qquad  \times \sum_{m=1}^{2}(-1)^{m}q^{\binom{m+1}{2}+m(r-3)}
 \times\Big (  j(-q^{8m+6r+6};q^{48}) 
 -q^{6m-2mr}j(-q^{-8m+6r+6};q^{48})\Big )\\
& \qquad \qquad \times \Big ( q^{2m-9}m(-q^{2m-9},-q^{15}z^{-2};q^{12}) 
 + q^{-m} m(-q^{2m+9},-q^{-3}z^{2};q^{12})\Big ).
\end{align*} }%
Expanding the two sums over $m$ brings us to
{\allowdisplaybreaks \begin{align*}
&\chi_{2r+1}^{(3,8)}(z;q)\\
& =  z^{-\frac{1}{2}}q^{\frac{3}{8}}C_{1,2r+1}^{(3,8)}(q) j(-z^{2}q^{3};q^{12})
+z^{-\frac{3}{2}}q^{\frac{27}{8}}C_{3,2r+1}^{(3,8)}(q) j(-z^{2}q^{-3};q^{12})\\
& \quad -     \frac{1}{(q)_{\infty}^3}
q^{-\frac{1}{8}+\frac{3(2r+2)^2}{32}}q^{3-3r}z^{-\frac{1}{2}} 
j(-q^{3}z^{2};q^{12})\\
& \qquad  \times \Big [ -q^{-2+r}
 \times\Big (  j(-q^{14+6r};q^{48}) 
  -q^{6-2r}j(-q^{-2+6r};q^{48})\Big )\\
& \qquad \qquad \qquad \times \Big ( q^{-2}m(-q^{-1},-q^{9}z^{-2};q^{12}) 
+  m(-q^{5},-q^{3}z^{2};q^{12})\Big )\\
& \qquad \qquad +q^{-3+2r}
 \times\Big (  j(-q^{22+6r};q^{48}) 
 -q^{12-4r}j(-q^{-10+6r};q^{48})\Big )\\
& \qquad \qquad \qquad \times \Big ( q^{-1}m(-q,-q^{9}z^{-2};q^{12}) 
+  m(-q^{7},-q^{3}z^{2};q^{12})\Big )\Big ]\\
& \quad -     \frac{1}{(q)_{\infty}^3}
q^{-\frac{1}{8}+\frac{3(2r+2)^2}{32}}q^{6-3r}z^{-\frac{3}{2}} 
j(-q^{-3}z^{2};q^{12})\\
& \qquad  \times \Big [ -q^{-2+r}
 \times\Big (  j(-q^{14+6r};q^{48}) 
 -q^{6-2r}j(-q^{-2+6r};q^{48})\Big )\\
& \qquad \qquad \qquad \times \Big ( q^{-7}m(-q^{-7},-q^{15}z^{-2};q^{12}) 
 + q^{-1} m(-q^{11},-q^{-3}z^{2};q^{12})\Big )\\
& \qquad \qquad + q^{-3+2r}
 \times\Big (  j(-q^{22+6r};q^{48}) 
 -q^{12-4r}j(-q^{-10+6r};q^{48})\Big )\\
& \qquad \qquad \qquad \times \Big ( q^{-5}m(-q^{-5},-q^{15}z^{-2};q^{12}) 
 + q^{-2} m(-q^{13},-q^{-3}z^{2};q^{12})\Big )\Big ].
\end{align*} }%
We now use (\ref{equation:j-elliptic}) and (\ref{equation:j-flip}) to set up for the quintuple product identity. 
{\allowdisplaybreaks \begin{align*}
\chi_{2r+1}^{(3,8)}(z;q)
& =  z^{-\frac{1}{2}}q^{\frac{3}{8}}C_{1,2r+1}^{(3,8)}(q) j(-z^{2}q^{3};q^{12})
+z^{-\frac{3}{2}}q^{\frac{27}{8}}C_{3,2r+1}^{(3,8)}(q) j(-z^{2}q^{-3};q^{12})\\
& \quad -     \frac{1}{(q)_{\infty}^3}
q^{-\frac{1}{8}+\frac{3(2r+2)^2}{32}}q^{3-3r}z^{-\frac{1}{2}} 
j(-q^{3}z^{2};q^{12})\\
& \qquad  \times \Big [ -q^{-2+r}
 \times\Big (  j(-q^{34-6r};q^{48}) 
  -q^{6-2r}j(-q^{50-6r};q^{48})\Big )\\
& \qquad \qquad \qquad \times \Big ( q^{-2}m(-q^{-1},-q^{9}z^{-2};q^{12}) 
+  m(-q^{5},-q^{3}z^{2};q^{12})\Big )\\
& \qquad \qquad +q^{-3+2r}
 \times\Big (  j(-q^{22+6r};q^{48}) 
 -q^{2+2r}j(-q^{38+6r};q^{48})\Big )\\
& \qquad \qquad \qquad \times \Big ( q^{-1}m(-q,-q^{9}z^{-2};q^{12}) 
+  m(-q^{7},-q^{3}z^{2};q^{12})\Big )\Big ]\\
& \quad -     \frac{1}{(q)_{\infty}^3}
q^{-\frac{1}{8}+\frac{3(2r+2)^2}{32}}q^{6-3r}z^{-\frac{3}{2}} 
j(-q^{-3}z^{2};q^{12})\\
& \qquad  \times \Big [ -q^{-2+r}
 \times\Big (  j(-q^{34-6r};q^{48}) 
 -q^{6-2r}j(-q^{50-6r};q^{48})\Big )\\
& \qquad \qquad \qquad \times \Big ( q^{-7}m(-q^{-7},-q^{15}z^{-2};q^{12}) 
 + q^{-1} m(-q^{11},-q^{-3}z^{2};q^{12})\Big )\\
& \qquad \qquad + q^{-3+2r}
 \times\Big (  j(-q^{22+6r};q^{48}) 
 -q^{2+2r}j(-q^{38+6r};q^{48})\Big )\\
& \qquad \qquad \qquad \times \Big ( q^{-5}m(-q^{-5},-q^{15}z^{-2};q^{12}) 
 + q^{-2} m(-q^{13},-q^{-3}z^{2};q^{12})\Big )\Big ].
\end{align*} }%
Using the quintuple product identity (\ref{equation:quintuple}) and then collecting the powers of $q$ in each of the summands gives
{\allowdisplaybreaks \begin{align*}
&\chi_{2r+1}^{(3,8)}(z;q)\\
& =  z^{-\frac{1}{2}}q^{\frac{3}{8}}C_{1,2r+1}^{(3,8)}(q) j(-z^{2}q^{3};q^{12})
+z^{-\frac{3}{2}}q^{\frac{27}{8}}C_{3,2r+1}^{(3,8)}(q) j(-z^{2}q^{-3};q^{12})\\
& \quad -     \frac{1}{(q)_{\infty}^3}
q^{-\frac{1}{8}+\frac{3(2r+2)^2}{32}}q^{3-3r}z^{-\frac{1}{2}} 
j(-q^{3}z^{2};q^{12})\\
& \qquad  \times \Big [ -q^{-2+r}
\frac{j(q^{6-2r};q^{16})j(q^{28-4r};q^{32})}{J_{32}}
 \Big ( q^{-2}m(-q^{-1},-q^{9}z^{-2};q^{12}) 
+  m(-q^{5},-q^{3}z^{2};q^{12})\Big )\\
& \qquad \qquad +q^{-3+2r}
\frac{j(q^{2+2r};q^{16})j(q^{20+4r};q^{32})}{J_{32}}
  \Big ( q^{-1}m(-q,-q^{9}z^{-2};q^{12}) 
+  m(-q^{7},-q^{3}z^{2};q^{12})\Big )\Big ]\\
& \quad -     \frac{1}{(q)_{\infty}^3}
q^{-\frac{1}{8}+\frac{3(2r+2)^2}{32}}q^{6-3r}z^{-\frac{3}{2}} 
j(-q^{-3}z^{2};q^{12})\\
& \qquad  \times \Big [ -q^{-2+r}
\frac{j(q^{6-2r};q^{16})j(q^{28-4r};q^{32})}{J_{32}}\\
&\qquad \qquad \qquad \times  \Big ( q^{-7}m(-q^{-7},-q^{15}z^{-2};q^{12}) 
 + q^{-1} m(-q^{11},-q^{-3}z^{2};q^{12})\Big )\\
& \qquad \qquad + q^{-3+2r}
\frac{j(q^{2+2r};q^{16})j(q^{20+4r};q^{32})}{J_{32}}\\
&\qquad \qquad \qquad \times  \Big ( q^{-5}m(-q^{-5},-q^{15}z^{-2};q^{12}) 
 + q^{-2} m(-q^{13},-q^{-3}z^{2};q^{12})\Big )\Big ].
\end{align*} }%
Using the Appell function properties (\ref{equation:mxqz-flip}) and (\ref{equation:mxqz-fnq-z}) brings us to
{\allowdisplaybreaks \begin{align*}
&\chi_{2r+1}^{(3,8)}(z;q)\\
& =  z^{-\frac{1}{2}}q^{\frac{3}{8}}C_{1,2r+1}^{(3,8)}(q) j(-z^{2}q^{3};q^{12})
+z^{-\frac{3}{2}}q^{\frac{27}{8}}C_{3,2r+1}^{(3,8)}(q) j(-z^{2}q^{-3};q^{12})\\
& \quad -     \frac{1}{(q)_{\infty}^3}
q^{-\frac{1}{8}+\frac{3(2r+2)^2}{32}}q^{3-3r}z^{-\frac{1}{2}} 
j(-q^{3}z^{2};q^{12})\\
& \qquad  \times \Big [ -q^{-2+r}
\frac{j(q^{6-2r};q^{16})j(q^{28-4r};q^{32})}{J_{32}}
 \Big ( -q^{-1}m(-q,-q^{3}z^{2};q^{12}) 
+  m(-q^{5},-q^{3}z^{2};q^{12})\Big )\\
& \qquad \qquad +q^{-3+2r}
\frac{j(q^{2+2r};q^{16})j(q^{20+4r};q^{32})}{J_{32}}
  \Big ( -q^{-2}m(-q^{-1},-q^{3}z^{2};q^{12}) 
+  m(-q^{7},-q^{3}z^{2};q^{12})\Big )\Big ]\\
& \quad -     \frac{1}{(q)_{\infty}^3}
q^{-\frac{1}{8}+\frac{3(2r+2)^2}{32}}q^{6-3r}z^{-\frac{3}{2}} 
j(-q^{-3}z^{2};q^{12})\\
& \qquad  \times \Big [ -q^{-2+r}
\frac{j(q^{6-2r};q^{16})j(q^{28-4r};q^{32})}{J_{32}}
 \Big ( -m(-q^{7},-q^{9}z^{2};q^{12}) 
 + q^{-1} m(-q^{11},-q^{9}z^{2};q^{12})\Big )\\
& \qquad \qquad + q^{-3+2r}
\frac{j(q^{2+2r};q^{16})j(q^{20+4r};q^{32})}{J_{32}}
 \Big ( -m(-q^{5},-q^{9}z^{2};q^{12}) 
 + q^{-2} m(-q^{13},-q^{9}z^{2};q^{12})\Big )\Big ].
\end{align*} }%
We change the string function notation using (\ref{equation:mathCalCtoStringC}).   This takes us to
\begin{equation*}
C_{1,2r+1}^{(3,8)}(q) 
=q^{-\frac{1}{8}+\frac{3(2r+2)^2}{4\cdot 8}-\frac{3\cdot 1^2}{4\cdot 2}}\mathcal{C}_{1,2r+1}^{(3,8)}(q) 
=q^{-\frac{1}{8}+\frac{3(r+1)^2}{ 8}-\frac{3}{8}}\mathcal{C}_{1,2r+1}^{(3,8)}(q), 
\end{equation*}
and
\begin{equation*}
C_{3,2r+1}^{(3,8)}(q) 
=q^{-\frac{1}{8}+\frac{3(2r+2)^2}{4\cdot 8}-\frac{3\cdot 3^2}{4\cdot 2}}\mathcal{C}_{3,2r+1}^{(3,8)}(q) 
=q^{-\frac{1}{8}+\frac{3(r+1)^2}{ 8}-\frac{27}{8}}\mathcal{C}_{3,2r+1}^{(3,8)}(q),
\end{equation*}
and the result follows.
\end{proof}

\begin{proof}[Proof of Lemma \ref{lemma:polarFinite23m1AppellVanish}]  We rewrite the sum of Appell functions.  We first note that (\ref{equation:mxqz-fnq-z}) gives
\begin{align*}
-m(-q^{7},-q^{9}z^2;q^{12})&+q^{-1}m(-q^{11},-q^{9}z^2;q^{12})\\
&=-m(-q^{7},-q^{-3}z^2;q^{12})+q^{-1}m(-q^{11},-q^{-3}z^2;q^{12}).
\end{align*}
Hence for the first limit, using the definition of our Appell function (\ref{equation:m-def}) yields
\begin{align*}
\lim_{z\to iq^{3/2}}j(-q^{-3}z^2;q^{12})
&\left ( -m(-q^{7},-q^{-3}z^2;q^{12})+q^{-1}m(-q^{11},-q^{-3}z^2;q^{12})\right )\\ 
&=-\sum_{n\in\mathbb{Z}}\frac{(-1)^nq^{12\binom{n}{2}}}{1-q^{12(n-1)}(-q^7)}
+q^{-1}\sum_{n\in\mathbb{Z}}\frac{(-1)^nq^{12\binom{n}{2}}}{1-q^{12(n-1)}(-q^{11})}.
\end{align*}
Setting $n \to n+1$ and then using the reciprocal formula (\ref{equation:jacobiThetaReciprocal}) yields
{\allowdisplaybreaks \begin{align*}
\lim_{z\to iq^{3/2}}j(-q^{-3}z^2;q^{12})&\left ( -m(-q^{7},-q^{-3}z^2;q^{12})+q^{-1}m(-q^{11},-q^{-3}z^2;q^{12})\right )\\ 
&=\sum_{n\in\mathbb{Z}}\frac{(-1)^nq^{12\binom{n+1}{2}}}{1+q^{12n+7}}
-q^{-1}\sum_{n\in\mathbb{Z}}\frac{(-1)^nq^{12\binom{n+1}{2}}}{1+q^{12n+11}}\\
&=\frac{J_{12}^3}{j(-q^{7};q^{12})}-q^{-1}\frac{J_{12}^3}{j(-q^{11};q^{12})}\\
&=-q^{-1}\frac{J_{12}^3}{j(-q^{11};q^{12})j(-q^7;q^{12})}\left (j(-q^{7};q^{12})-q j(-q^{11};q^{12}) \right ).
\end{align*}}%
Use (\ref{equation:j-flip}) and (\ref{equation:j-elliptic}), and then combine the theta functions using (\ref{equation:jsplit-m2}) to get
{\allowdisplaybreaks \begin{align*}
\lim_{z\to iq^{3/2}}j(-q^{-3}z^2;q^{12})&\left ( -m(-q^{7},-q^{-3}z^2;q^{12})+q^{-1}m(-q^{11},-q^{-3}z^2;q^{12})\right )\\ 
&=-q^{-1}\frac{J_{12}^3}{j(-q;q^{12})j(-q^7;q^{12})}\left (j(-q^{7};q^{12})-q j(-q;q^{12}) \right ) \\
&=-q^{-1}\frac{J_{12}^3}{j(-q;q^{12})j(-q^7;q^{12})}\left (j(-q^{7};q^{12})-q^2 j(-q^{13};q^{12}) \right )\\
&=-q^{-1}\frac{J_{12}^3j(q^{2};q^{3})}{j(-q;q^{12})j(-q^7;q^{12})}.
\end{align*}}%
Elementary product rearrangements gives the result.

For the second limit, we use (\ref{equation:mxqz-fnq-z}) to rewrite the sum of Appell functions to get
\begin{align*}
-m(-q^{5},-q^{-3}z^2;q^{12})&+q^{-2}m(-q^{13},-q^{-3}z^2;q^{12})\\
&=-m(-q^{5},-q^{-3}z^2;q^{12})+q^{-2}m(-q^{13},-q^{-3}z^2;q^{12}).
\end{align*}
Using the definition of our Appell function (\ref{equation:m-def}) immediately yields
\begin{align*}
\lim_{z\to iq^{3/2}}j(-q^{-3}z^2;q^{12})
&\left ( -m(-q^{5},-q^{-3}z^2;q^{12})+q^{-2}m(-q^{13},-q^{-3}z^2;q^{12})\right )\\ 
&=-\sum_{n\in\mathbb{Z}}\frac{(-1)^nq^{12\binom{n}{2}}}{1-q^{12(n-1)}(-q^5)}
+q^{-2}\sum_{n\in\mathbb{Z}}\frac{(-1)^nq^{12\binom{n}{2}}}{1-q^{12(n-1)}(-q^{13})}.
\end{align*}
Setting $n \to n+1$ and then using the reciprocal formula (\ref{equation:jacobiThetaReciprocal}) yields
{\allowdisplaybreaks \begin{align*}
\lim_{z\to iq^{3/2}}j(-q^{-3}z^2;q^{12})
&\left ( -m(-q^{5},-q^{-3}z^2;q^{12})+q^{-2}m(-q^{13},-q^{-3}z^2;q^{12})\right )\\ 
&=-\sum_{n\in\mathbb{Z}}\frac{(-1)^nq^{12\binom{n}{2}}}{1-q^{12(n-1)}(-q^5)}
+q^{-2}\sum_{n\in\mathbb{Z}}\frac{(-1)^nq^{12\binom{n}{2}}}{1-q^{12(n-1)}(-q^{13})}\\
&=\frac{J_{12}^3}{j(-q^{5};q^{12})}-q^{-2}\frac{J_{12}^3}{j(-q^{13};q^{12})}.
\end{align*}}%
Combining terms with (\ref{equation:jsplit-m2}) gives
{\allowdisplaybreaks \begin{align*}
\lim_{z\to iq^{3/2}}j(-q^{-3}z^2;q^{12})
&\left ( -m(-q^{5},-q^{-3}z^2;q^{12})+q^{-2}m(-q^{13},-q^{-3}z^2;q^{12})\right )\\ 
&=-q^{-2}\frac{J_{12}^3}{j(-q^{5};q^{12})j(-q^{13};q^{12})}\left ( j(-q^{5};q^{12})-q^2j(-q^{13};q^{12})\right ) \\
&=-q^{-1}\frac{J_{12}^3}{j(-q^{7};q^{12})j(-q;q^{12})}j(q^2;q^3).
\end{align*}}%
The result then follows from elementary product rearrangements.
\end{proof}

\begin{proof}[Proof of Lemma \ref{lemma:polarFinite23m3AppellVanish}]   The proof is similar to that of Lemma \ref{lemma:polarFinite23m1AppellVanish}, so we will omit it.
\end{proof}

\begin{proof}[Proof of Proposition \ref{proposition:weylKac23ell2rzVal}]
We prove (\ref{equation:weylKac23ell2rPlus1zVal1}).  Taking the appropriate specialization of Proposition \ref{proposition:WeylKac} gives
\begin{equation}
\chi_{2r+1}^{(3,8)}(z;q)
=z^{-r-\frac{1}{2}}q^{-\frac{1}{8}+3\frac{(r+1)^2}{8}}
\frac{j(-q^{30+6r}z^{-8};q^{48})
-z^{2r+2}j(-q^{18-6r}z^{-8};q^{48}) }
{j(z;q)}.
 \label{equation:weylKac23genz}
\end{equation}
Specializing (\ref{equation:weylKac23genz}) to $z=iq^{-3/2}$ gives
\begin{equation*}
\chi_{2r+1}^{(3,8)}(iq^{-3/2};q)
=(iq^{-3/2})^{-r-\frac{1}{2}}q^{-\frac{1}{8}+3\frac{(r+1)^2}{8}}
\frac{j(-q^{42+6r};q^{48})
-(-1)^{r+1}q^{-3(r+1)}j(-q^{30-6r};q^{48}) }
{j(iq^{-3/2};q)}.
\end{equation*}
We rewrite the numerator.  Pulling out a common factor and using (\ref{equation:j-flip}) gives
\begin{align*}
j(-q^{42+6r};q^{48})&-(-1)^{r+1}q^{-3(r+1)}j(-q^{30-6r};q^{48}) \\
&=-(-1)^{r+1}q^{-3(r+1)}\left (j(-q^{30-6r};q^{48})  -(-1)^{r+1}q^{3(r+1)}  j(-q^{42+6r};q^{48})\right ) \\
&=-(-1)^{r+1}q^{-3(r+1)}\left (j(-q^{18+6r};q^{48})  -(-q^{3})^{r+1}  j(-q^{42+6r};q^{48})\right ).
\end{align*}
Using (\ref{equation:jsplit-m2}) gives
\begin{align*}
j(-q^{42+6r};q^{48})&-(-1)^{r+1}q^{-3(r+1)}j(-q^{30-6r};q^{48}) \\
&=-(-1)^{r+1}q^{-3(r+1)}\left (j(-q^{12+6(r+1)};q^{48})  -(-q^{3})^{r+1}  j(-q^{36+6(r+1)};q^{48})\right ) \\
&=-(-1)^{r+1}q^{-3(r+1)}j((-q^3)^{(r+1)};q^{12}).
\end{align*}
Property (\ref{equation:j-elliptic}) and elementary product rearrangements give
\begin{align*}
j(iq^{-3/2};q)=j(-iq^{5/2};q)
=q^{-1}(-iq^{1/2})^{-2}j(-iq^{1/2};q)
=-q^{-2}J_{2,4},
\end{align*}
and the result follows.

The proof of  (\ref{equation:weylKac23ell2rPlus1zVal2}) is similar and will be omitted.
\end{proof}


\subsection{The $2/3$-level string functions with quantum number $m=1$, first theorem}
We prove Theorem \ref{theorem:newMockThetaIdentitiespP38m1ell2rPlus1}.  We specialize Proposition \ref{proposition:polarFinite23OddSpin} to $z=iq^{3/2}$ and use Lemma \ref{lemma:polarFinite23m1AppellVanish} to produce the much smaller
{\allowdisplaybreaks \begin{align*}
&q^{\frac{1}{8}-\frac{3(r+1)^2}{8}}\chi_{2r+1}^{(3,8)}(iq^{3/2};q)\\
& =  (iq^{3/2})^{-\frac{1}{2}}\mathcal{C}_{1,2r+1}^{(3,8)}(q) j(q^{6};q^{12})\\
& \quad -     \frac{1}{(q)_{\infty}^3}(iq^{3/2})^{-\frac{1}{2}} 
j(q^{6};q^{12})\\
& \qquad  \times \Big [ -q^{1-2r}
\frac{j(q^{6-2r};q^{16})j(q^{28-4r};q^{32})}{J_{32}}
 \Big (-q^{-1}m(-q,q^{6};q^{12}) 
+  m(-q^{5},q^{6};q^{12})\Big )\\
& \qquad \quad +q^{-r}
\frac{j(q^{2+2r};q^{16})j(q^{20+4r};q^{32})}{J_{32}}
  \Big ( -q^{-2}m(-q^{-1},q^{6};q^{12}) 
+  m(-q^{7},q^{6};q^{12})\Big )\Big ]\\
& \quad -     \frac{1}{(q)_{\infty}^3}(iq^{3/2})^{-\frac{3}{2}} 
 \Big (-q^{-1}\frac{J_{1}^2J_{4}J_{6}^2}{J_{2}^2J_{3}}\Big )\\
& \qquad  \times \Big [ -q^{4-2r}
\frac{j(q^{6-2r};q^{16})j(q^{28-4r};q^{32})}{J_{32}}
+ q^{3-r}
\frac{j(q^{2+2r};q^{16})j(q^{20+4r};q^{32})}{J_{32}}\Big ].
\end{align*} }%

Substituting in the expression for the character found in (\ref{equation:weylKac23ell2rPlus1zVal1}) and the mock theta functions found in Proposition \ref{proposition:alternateAppellFormsLvl23FirstTwoPairs} yields
{\allowdisplaybreaks \begin{align*}
&-i(iq^{3/2})^{-r-\frac{1}{2}}
q^{\frac{1}{2}}\frac{j((-q^{3})^{r+1};q^{12})}
{J_{2,4}}\\
& =  (iq^{3/2})^{-\frac{1}{2}}\mathcal{C}_{1,2r+1}^{(3,8)}(q) j(q^{6};q^{12})\\
& \quad -     \frac{1}{(q)_{\infty}^3}(iq^{3/2})^{-\frac{1}{2}} 
j(q^{6};q^{12})\\
& \qquad  \times \Big [ -q^{1-2r}
\frac{j(q^{6-2r};q^{16})j(q^{28-4r};q^{32})}{J_{32}}
 \left ( -\psi_{3}(-q)\right)\\
& \qquad \quad +q^{-r}
\frac{j(q^{2+2r};q^{16})j(q^{20+4r};q^{32})}{J_{32}}
  \left ( 1-\left (\chi_{3}(q)-\frac{J_{3}J_{4}^3}{J_{2}^2J_{12}}\right ) \right )\Big ]\\
& \quad -     \frac{1}{(q)_{\infty}^3}(iq^{3/2})^{-\frac{3}{2}} 
 \Big (-q^{-1}\frac{J_{1}^2J_{4}J_{6}^2}{J_{2}^2J_{3}}\Big )\\
& \qquad  \times \Big [ -q^{4-2r}
\frac{j(q^{6-2r};q^{16})j(q^{28-4r};q^{32})}{J_{32}}
+ q^{3-r}
\frac{j(q^{2+2r};q^{16})j(q^{20+4r};q^{32})}{J_{32}}\Big ].
\end{align*} }%
Isolating the string function and regrouping terms then reads
{\allowdisplaybreaks \begin{align*}
&(iq^{3/2})(q)_{\infty}^{3}\mathcal{C}_{1,2r+1}^{(3,8)}(q) \\
&=-i(iq^{3/2})^{-r+1}
q^{\frac{1}{2}}\frac{j((-q^{3})^{r+1};q^{12})}
{J_{2,4}j(q^{6};q^{12})}\\
&\quad  + (iq^{3/2}) 
q^{-r} \frac{j(q^{2+2r};q^{16})j(q^{20+4r};q^{32})}{J_{32}}\frac{J_{3}J_{4}^3}{J_{2}^2J_{12}}\\
& \quad +q^{-1}\frac{J_{1}^2J_{4}J_{12}}{J_{2}^2J_{3}}
  \times \left ( q^{4-2r}
\frac{j(q^{6-2r};q^{16})j(q^{28-4r};q^{32})}{J_{32}}
- q^{3-r}
\frac{j(q^{2+2r};q^{16})j(q^{20+4r};q^{32})}{J_{32}}\right )\\
& \quad +    (iq^{3/2})   \times \Big [ -q^{1-2r}
\frac{j(q^{6-2r};q^{16})j(q^{28-4r};q^{32})}{J_{32}}
 \left ( -\psi_{3}(-q)\right)\\
& \qquad \quad +q^{-r}
\frac{j(q^{2+2r};q^{16})j(q^{20+4r};q^{32})}{J_{32}}
  \left ( 1- \chi_{3}(q)\right) \Big ].
\end{align*} }%
More simplifying and then using Lemma \ref{lemma:unusualThetaIdentity2} yields

{\allowdisplaybreaks \begin{align*}
(q)_{\infty}^{3}&\mathcal{C}_{1,2r+1}^{(3,8)}(q) \\
&=-i(iq^{3/2})^{-r}
q^{\frac{1}{2}}\frac{J_{1}^3j((-q^{3})^{r+1};q^{12})}
{J_{2,4}j(q^{6};q^{12})}
 + q^{-r} \frac{j(q^{2+2r};q^{16})j(q^{20+4r};q^{32})}{J_{32}}\frac{J_{3}J_{4}^3}{J_{2}^2J_{12}}\\
& \quad +\frac{1}{iq^{3/2} }q^{-1}\frac{J_{1}^2J_{4}J_{12}}{J_{2}^2J_{3}}
  \times \left ( -q^{3-r}\frac{j(q^{2+2r};q^{8})}{J_{8}}j(q^{1-r};q^{4})\right )\\
& \quad -q^{1-2r}
\frac{j(q^{6-2r};q^{16})j(q^{28-4r};q^{32})}{J_{32}}
 \left ( -\psi_{3}(-q)\right)\\
& \qquad \quad +q^{-r}
\frac{j(q^{2+2r};q^{16})j(q^{20+4r};q^{32})}{J_{32}}
  \left ( 1- \chi_{3}(q)\right).
\end{align*} }%

The result then follows from Proposition \ref{proposition:masterThetaIdentitypP38m1ell2rPlus1}.


\subsection{The $2/3$-level string functions with quantum number $m=1$, second theorem}

We prove Theorem \ref{theorem:newMockThetaIdentitiespP38m1ell2rPlus1Alt}.  We specialize Proposition \ref{proposition:polarFinite23OddSpin} to $z=iq^{3/2}$ and use Lemma \ref{lemma:polarFinite23m1AppellVanish} to produce the much smaller
{\allowdisplaybreaks \begin{align*}
&q^{\frac{1}{8}-\frac{3(r+1)^2}{8}}\chi_{2r+1}^{(3,8)}(iq^{3/2};q)\\
& =  (iq^{3/2})^{-\frac{1}{2}}\mathcal{C}_{1,2r+1}^{(3,8)}(q) j(q^{6};q^{12})\\
& \quad -     \frac{1}{(q)_{\infty}^3}(iq^{3/2})^{-\frac{1}{2}} 
j(q^{6};q^{12})\\
& \qquad  \times \Big [ -q^{1-2r}
\frac{j(q^{6-2r};q^{16})j(q^{28-4r};q^{32})}{J_{32}}
 \Big (-q^{-1}m(-q,q^{6};q^{12}) 
+  m(-q^{5},q^{6};q^{12})\Big )\\
& \qquad \quad +q^{-r}
\frac{j(q^{2+2r};q^{16})j(q^{20+4r};q^{32})}{J_{32}}
  \Big ( -q^{-2}m(-q^{-1},q^{6};q^{12}) 
+  m(-q^{7},q^{6};q^{12})\Big )\Big ]\\
& \quad -     \frac{1}{(q)_{\infty}^3}(iq^{3/2})^{-\frac{3}{2}} 
 \Big (-q^{-1}\frac{J_{1}^2J_{4}J_{6}^2}{J_{2}^2J_{3}}\Big )\\
& \qquad  \times \Big [ -q^{4-2r}
\frac{j(q^{6-2r};q^{16})j(q^{28-4r};q^{32})}{J_{32}}
+ q^{3-r}
\frac{j(q^{2+2r};q^{16})j(q^{20+4r};q^{32})}{J_{32}}\Big ].
\end{align*} }%

Substituting in the expression for the character found in (\ref{equation:weylKac23ell2rPlus1zVal1}) and the mock theta functions found in Proposition \ref{proposition:alternateAppellFormsLvl23FirstTwoPairsAlt} yields
{\allowdisplaybreaks \begin{align*}
&-i(iq^{3/2})^{-r-\frac{1}{2}}
q^{\frac{1}{2}}\frac{j((-q^{3})^{r+1};q^{12})}
{J_{2,4}}\\
& =  (iq^{3/2})^{-\frac{1}{2}}\mathcal{C}_{1,2r+1}^{(3,8)}(q) j(q^{6};q^{12})\\
& \quad -     \frac{1}{(q)_{\infty}^3}(iq^{3/2})^{-\frac{1}{2}} 
j(q^{6};q^{12})\\
& \qquad  \times \Big [ -q^{1-2r}
\frac{j(q^{6-2r};q^{16})j(q^{28-4r};q^{32})}{J_{32}}
 \Big (\frac{1}{4}f_{3}(q)-\frac{1}{4}\frac{J_{1}^3}{J_{2}^2}\Big )\\
& \qquad \quad +q^{-r}
\frac{j(q^{2+2r};q^{16})j(q^{20+4r};q^{32})}{J_{32}}
  \Big ( 1-\left ( \frac{1}{4}f_{3}(q)-\frac{1}{4}\frac{J_{1}^3}{J_{2}^2}\right ) \Big )\Big ]\\
& \quad -     \frac{1}{(q)_{\infty}^3}(iq^{3/2})^{-\frac{3}{2}} 
 \Big (-q^{-1}\frac{J_{1}^2J_{4}J_{6}^2}{J_{2}^2J_{3}}\Big )\\
& \qquad  \times \Big [ -q^{4-2r}
\frac{j(q^{6-2r};q^{16})j(q^{28-4r};q^{32})}{J_{32}}
+ q^{3-r}
\frac{j(q^{2+2r};q^{16})j(q^{20+4r};q^{32})}{J_{32}}\Big ].
\end{align*} }%
Regrouping terms and isolating the string function yields
{\allowdisplaybreaks \begin{align*}
(q)_{\infty}^3&\mathcal{C}_{1,2r+1}^{(3,8)}(q)\\
&=-i(iq^{3/2})^{-r}
q^{\frac{1}{2}}\frac{J_{1}^3j((-q^{3})^{r+1};q^{12})}
{J_{2,4}J_{6,12}}\\
& \quad  -q^{1-2r}
\frac{j(q^{6-2r};q^{16})j(q^{28-4r};q^{32})}{J_{32}}
 \frac{1}{4}f_{3}(q)\\
& \qquad \quad +q^{-r}
\frac{j(q^{2+2r};q^{16})j(q^{20+4r};q^{32})}{J_{32}}
  \left ( 1- \frac{1}{4}f_{3}(q) \right )\\
& \quad +   \frac{1}{4}\frac{J_{1}^3}{J_{2}^2} 
  \times \Big [ q^{1-2r}
\frac{j(q^{6-2r};q^{16})j(q^{28-4r};q^{32})}{J_{32}}
 +q^{-r}
\frac{j(q^{2+2r};q^{16})j(q^{20+4r};q^{32})}{J_{32}}
  \Big ]\\  
& \quad +   (iq^{3/2})^{-1} 
 \Big (q^{2}\frac{J_{1}^2J_{4}J_{12}}{J_{2}^2J_{3}}\Big )\\
& \qquad  \times \Big [ q^{1-2r}
\frac{j(q^{6-2r};q^{16})j(q^{28-4r};q^{32})}{J_{32}}
- q^{-r}
\frac{j(q^{2+2r};q^{16})j(q^{20+4r};q^{32})}{J_{32}}\Big ].
\end{align*} }%
Lemma \ref{lemma:unusualThetaIdentity2} then gives a more compact form
{\allowdisplaybreaks \begin{align*}
(q)_{\infty}^3\mathcal{C}_{1,2r+1}^{(3,8)}(q)
&=-i(iq^{3/2})^{-r}
q^{\frac{1}{2}}\frac{J_{1}^3j((-q^{3})^{r+1};q^{12})}
{J_{2,4}J_{6,12}}\\
& \quad  -q^{1-2r}
\frac{j(q^{6-2r};q^{16})j(q^{28-4r};q^{32})}{J_{32}}
 \frac{1}{4}f_{3}(q)\\
& \qquad \quad +q^{-r}
\frac{j(q^{2+2r};q^{16})j(q^{20+4r};q^{32})}{J_{32}}
  \left ( 1- \frac{1}{4}f_{3}(q) \right )\\
& \quad +   \frac{1}{4}\frac{J_{1}^3}{J_{2}^2} 
  \times \left ( q^{-r}
\frac{j(q^{2+2r};q^{8})}{J_{8}}j(-q^{1-r};q^{4})
  \right ) \\  
& \quad +   (iq^{3/2})^{-1} 
 \Big (q^{2}\frac{J_{1}^2J_{4}J_{12}}{J_{2}^2J_{3}}\Big )
   \times \left (  -q^{-r}
\frac{j(q^{2+2r};q^{8})}{J_{8}}j(q^{1-r};q^{4}) \right ) .
\end{align*} }%
The result then follows from Proposition \ref{proposition:masterThetaIdentitypP38m1ell2rPlus1Alt}.


\subsection{The $2/3$-level string functions with quantum number $m=3$, a cross-spin approach}
We prove Corollaries \ref{corollary:newMockThetaIdentitiespP38m3ell2rPlus1} and \ref{corollary:newMockThetaIdentitiespP38m3ell2rPlus1Alt}, but the proofs are similar, so we will only carry out the first.  To relate odd-spin $2/3$-level string functions with quantum number $3$ to the already established quantum number $m=1$, odd-spin $2/3$-level string functions, we use the following specialization of Theorem \ref{theorem:crossSpin-j-Odd}: 
\begin{align*}
(q)_{\infty}^3&\mathcal{C}_{2i-1,2r-1}^{(3,8)}(q)-q^{3(i-r)+3}(q)_{\infty}^{3}\mathcal{C}_{2i-3,7-2r}^{(3,8)}(q)\\
&= -q^{3+3(i+r)}
\sum_{m=1}^{2}(-1)^{m}q^{\binom{m+1}{2}-m(3+i+r)}
 \times
\left ( j(-q^{8m-6r};q^{48})
-q^{2r(m-3)}j(-q^{8m+6r};q^{48})\right ).
\end{align*}
We set $i=2$ to get
\begin{align*}
(q)_{\infty}^3&\mathcal{C}_{3,2r-1}^{(3,8)}(q)-q^{9-3r}(q)_{\infty}^{3}\mathcal{C}_{1,7-2r}^{(3,8)}(q)\\
&= -q^{9+3r}
\sum_{m=1}^{2}(-1)^{m}q^{\binom{m+1}{2}-m(5+r)}
 \times
\left ( j(-q^{8m-6r};q^{48})
-q^{2r(m-3)}j(-q^{8m+6r};q^{48})\right ) .
\end{align*}
Expanding the sum over $m$ to brings us to
\begin{align*}
(q)_{\infty}^3&\mathcal{C}_{3,2r-1}^{(3,8)}(q)-q^{9-3r}(q)_{\infty}^{3}\mathcal{C}_{1,7-2r}^{(3,8)}(q)\\
&= -q^{9+3r}\Big [ 
-q^{-4-r}
 \times
\left ( j(-q^{8-6r};q^{48}) -q^{-4r}j(-q^{8+6r};q^{48})\right )\\
&\qquad +q^{-7-2r}
 \times
\left ( j(-q^{16-6r};q^{48}) -q^{-2r}j(-q^{16+6r};q^{48})\right )\Big ].
\end{align*}
We rewrite the theta functions using (\ref{equation:j-flip}) and (\ref{equation:j-elliptic}) to get
\begin{align*}
(q)_{\infty}^3&\mathcal{C}_{3,2r-1}^{(3,8)}(q)-q^{9-3r}(q)_{\infty}^{3}\mathcal{C}_{1,7-2r}^{(3,8)}(q)\\
&= -q^{9+3r}\Big [ 
-q^{-4-r}
 \times
\left ( j(-q^{40+6r};q^{48}) -q^{8+2r}j(-q^{56+6r};q^{48})\right )\\
&\qquad +q^{-7-2r}
 \times
\left ( j(-q^{32+6r};q^{48}) -q^{-2r}j(-q^{16+6r};q^{48})\right )\Big ].
\end{align*}
We rearrange terms to set up for quintuple product identity:
\begin{align*}
(q)_{\infty}^3&\mathcal{C}_{3,2r-1}^{(3,8)}(q)-q^{9-3r}(q)_{\infty}^{3}\mathcal{C}_{1,7-2r}^{(3,8)}(q)\\
&= -q^{9+3r}\Big [ 
-q^{-4-r}
 \times
\left ( j(-q^{40+6r};q^{48}) -q^{8+2r}j(-q^{56+6r};q^{48})\right )\\
&\qquad -q^{-7-4r}
 \times
\left ( j(-q^{16+6r};q^{48})-q^{2r}j(-q^{32+6r};q^{48})\right )\Big ].
\end{align*}
Using the quintuple product identity (\ref{equation:quintuple}) yields
{\allowdisplaybreaks \begin{align*}
(q)_{\infty}^3&\mathcal{C}_{3,2r-1}^{(3,8)}(q)-q^{9-3r}(q)_{\infty}^{3}\mathcal{C}_{1,7-2r}^{(3,8)}(q)\\
&= -q^{9+3r}\Big [ 
-q^{-4-r}
 \frac{j(q^{8+2r};q^{16})j(q^{24+2r};q^{32})}{J_{32}}
  -q^{-7-4r}
 \frac{j(q^{2r};q^{16})j(q^{16+2r};q^{32})}{J_{32}}\Big ]\\
&= q^{5+2r}
 \frac{j(q^{8+2r};q^{16})j(q^{32+4r};q^{32})}{J_{32}}
 +q^{2-r}
 \frac{j(q^{2r};q^{16})j(q^{16+4r};q^{32})}{J_{32}}.
\end{align*}}%
Replacing $r$ with $r+1$ and using (\ref{equation:j-elliptic}) brings us to
\begin{align}
(q)_{\infty}^3&\mathcal{C}_{3,2r+1}^{(3,8)}(q)-q^{6-3r}(q)_{\infty}^{3}\mathcal{C}_{1,2(2-r)+1}^{(3,8)}(q)
\label{equation:crossSpin23master}\\
&= -q^{3-2r}
 \frac{j(q^{10+2r};q^{16})j(q^{4+4r};q^{32})}{J_{32}}
 +q^{1-r}
 \frac{j(q^{2+2r};q^{16})j(q^{20+4r};q^{32})}{J_{32}}.\notag
\end{align}
Let us substitute the result of Theorem \ref{theorem:newMockThetaIdentitiespP38m1ell2rPlus1} into (\ref{equation:crossSpin23master}).  This gives
{\allowdisplaybreaks \begin{align*}
(q)_{\infty}^3&\mathcal{C}_{3,2r+1}^{(3,8)}(q)\\
&=q^{6-3r}\Big [\Theta_{1,2(2-r)+1}(q)\\
&\qquad    -q^{1-2(2-r)}
\frac{j(q^{6-2(2-r)};q^{16})j(q^{28-4(2-r)};q^{32})}{J_{32}}
 \Big ( -\psi_{3}(-q)\Big )\\
& \qquad  +q^{-(2-r)}
\frac{j(q^{2+2(2-r)};q^{16})j(q^{20+4(2-r)};q^{32})}{J_{32}}
  \Big ( 1-\chi_{3}(q)\Big ) \Big ] \\
&\qquad  -q^{3-2r}
 \frac{j(q^{10+2r};q^{16})j(q^{4+4r};q^{32})}{J_{32}}
 +q^{1-r}
 \frac{j(q^{2+2r};q^{16})j(q^{20+4r};q^{32})}{J_{32}}.
\end{align*}}%
Simplifying gives
\begin{align*}
(q)_{\infty}^3\mathcal{C}_{3,2r+1}^{(3,8)}(q)
&=q^{6-3r}\Theta_{1,2(2-r)+1}(q)\\
&\qquad    -q^{3-r}
\frac{j(q^{2+2r};q^{16})j(q^{20+4r};q^{32})}{J_{32}}
 \Big ( -\psi_{3}(-q)\Big )\\
& \qquad  +q^{4-2r}
\frac{j(q^{6-2r};q^{16})j(q^{28-4r};q^{32})}{J_{32}}
  \Big ( 1-\chi_{3}(q)\Big )  \\
&\qquad  -q^{3-2r}
 \frac{j(q^{10+2r};q^{16})j(q^{4+4r};q^{32})}{J_{32}}
 +q^{1-r}
 \frac{j(q^{2+2r};q^{16})j(q^{20+4r};q^{32})}{J_{32}}.
\end{align*}
Using (\ref{equation:j-flip}) and combining terms gives
{\allowdisplaybreaks \begin{align*}
(q)_{\infty}^3\mathcal{C}_{3,2r+1}^{(3,8)}(q)
&=q^{6-3r}\Theta_{1,2(2-r)+1}(q)\\
&\qquad    +q^{3-r}
\frac{j(q^{2+2r};q^{16})j(q^{20+4r};q^{32})}{J_{32}}
 \Big ( q^{-2}+\psi_{3}(-q)\Big )\\
& \qquad  +q^{4-2r}
\frac{j(q^{6-2r};q^{16})j(q^{28-4r};q^{32})}{J_{32}}
  \Big ( 1-q^{-1}-\chi_{3}(q)\Big ),
\end{align*}}%
and the result follows.


\section{Acknowledgments}

The work is supported by the Ministry of Science and Higher Education of the Russian
Federation (agreement no. 075-15-2025-343).

\end{document}